\documentclass{amsart}

\usepackage{amsfonts}
\usepackage{amssymb}
\usepackage{tikz}
\usetikzlibrary{arrows,decorations.pathreplacing}
\usetikzlibrary{shapes,decorations.pathmorphing}
\usetikzlibrary{arrows,calc} 

\allowdisplaybreaks
\newcommand\scalemath[2]{\scalebox{#1}{\mbox{\ensuremath{\displaystyle #2}}}}

\def\circledarrow#1#2#3{ 
\draw[#1,-] (#2) +(157:#3) arc(157:-157:#3);
}

\usepackage{float}

\usepackage{graphicx}
\usepackage{bm}
\usepackage[
	pdftitle={},
	pdfauthor={},
	ocgcolorlinks,
	linkcolor=linkblue,
	citecolor=linkred,
	urlcolor=linkred]
{hyperref}
\usepackage{color}
\definecolor{linkred}{rgb}{0.75,0,0}
\definecolor{linkblue}{rgb}{0,0,0.75}
\usepackage{mathpazo}
\usepackage{microtype}
\usepackage{multicol}
\usepackage{booktabs}
\usepackage{latexsym}
\usepackage{multirow}
\usepackage{setspace}
\usepackage{enumerate}

\theoremstyle{plain}
\newtheorem{theorem}{Theorem}
\newtheorem{proposition}{Proposition}[section]
\newtheorem{lemma}[proposition]{Lemma}
\newtheorem{conjecture}{Conjecture}
\newtheorem{corollary}[proposition]{Corollary}
\newtheorem{example}[proposition]{Example}

\theoremstyle{definition}

\newtheorem{definition}[proposition]{Definition}
\newtheorem{remark}[proposition]{Remark}
\newcommand{\beq}{\begin{equation}}
\newcommand{\eeq}{\end{equation}}

\newcommand{\cal}{\mathcal}
\newcommand{\Res}{\mathop{\,\rm Res\,}}

\newcommand{\shift}{12}

\newcommand{\ca}{\mathcal{A}}
\newcommand{\cb}{\mathcal{B}}
\newcommand{\cc}{\mathcal{C}}

\newcommand{\ce}{\mathcal{E}}

\newcommand{\cl}{\mathcal{L}}
\newcommand{\co}{\mathcal{O}}
\newcommand{\cp}{{\cal P}}

\newcommand{\bc}{\mathbb{C}}
\newcommand{\bn}{\mathbb{N}}
\newcommand{\bp}{\mathbb{P}}
\newcommand{\bq}{\mathbb{Q}}
\newcommand{\br}{\mathbb{R}}
\newcommand{\bz}{\mathbb{Z}}
\newcommand*{\Cdot}{\raisebox{-0.5ex}{\scalebox{1.8}{$\cdot$}}}

\newcommand{\diag}{\mathrm{diag}}
\newcommand{\modm}{\cal M}
\newcommand{\un}{1\!\!1}

\begin{document}
	
\title{A new cohomology class on the moduli space of curves}
\author{Paul Norbury}
\address{School of Mathematics and Statistics, University of Melbourne, VIC 3010, Australia}
\email{\href{mailto:norbury@unimelb.edu.au}{norbury@unimelb.edu.au}}
\thanks{}
\subjclass[2010]{32G15; 14D23; 53D45}
\date{\today}

\begin{abstract}
We define a collection $\Theta_{g,n}\in H^{4g-4+2n}(\overline{\cal M}_{g,n},\mathbb{Q})$  for $2g-2+n>0$ of cohomology classes that restrict naturally to boundary divisors.  We prove that the intersection numbers $\int_{\overline{\cal M}_{g,n}}\Theta_{g,n}\prod_{i=1}^n\psi_i^{m_i}$ can be recursively calculated.  We conjecture that a generating function for these intersection numbers is a tau function of the KdV hierarchy.  This is analogous to the conjecture of Witten proven by Kontsevich that a generating function for the intersection numbers $\int_{\overline{\cal M}_{g,n}}\prod_{i=1}^n\psi_i^{m_i}$ is a tau function of the KdV hierarchy.
\end{abstract}

\maketitle

\tableofcontents

\section{Introduction}

Let $\overline{\modm}_{g,n}$ be the moduli space 
of genus $g$ stable curves---curves with only nodal singularities and finite automorphism group---with $n$ labeled points disjoint from nodes.   Define $\psi_i=c_1(L_i)\in H^{2}(\overline{\mathcal{M}}_{g,n},\mathbb{Q}) $ to be the first Chern class of the line bundle $L_i\to\overline{\mathcal{M}}_{g,n}$ with fibre above $[(C,p_1,\ldots,p_n)]$ given by $T_{p_i}^*C$.  Consider the natural maps given by the forgetful map which forgets the last point
\begin{equation}\label{forgetful}
\overline{\modm}_{g,n+1}\stackrel{\pi}{\longrightarrow}\overline{\modm}_{g,n}
\end{equation}
and the gluing maps which glue the last two points
\begin{equation}\label{gluing}
\overline{\modm}_{g-1,n+2}\stackrel{\phi_{\text{irr}}}{\longrightarrow}\overline{\modm}_{g,n},\quad\overline{\modm}_{h,|I|+1}\times\overline{\modm}_{g-h,|J|+1}\stackrel{\phi_{h,I}}{\longrightarrow}\overline{\modm}_{g,n},\quad I\sqcup J=\{1,...,n\}
\end{equation}

In this paper we construct cohomology classes $\Theta_{g,n}\in H^*(\overline{\modm}_{g,n},\bq)$ for $g\geq 0$, $n\geq 0$ and $2g-2+n>0$ satisfying the following four properties:
\begin{enumerate}[(i)]
\setlength{\itemindent}{20pt}
\item $\Theta_{g,n}\in H^*(\overline{\modm}_{g,n},\bq)$ is of pure degree, \label{pure}
\item $\phi_{\text{irr}}^*\Theta_{g,n}=\Theta_{g-1,n+2}$,\quad $\phi_{h,I}^*\Theta_{g,n}=\pi_1^*\Theta_{h,|I|+1}\cdot\pi_2^* \Theta_{g-h,|J|+1}$,  \label{glue}
\item $\Theta_{g,n+1}=\psi_{n+1}\cdot\pi^*\Theta_{g,n}$,  \label{forget}
\item  $\Theta_{1,1}\neq 0$  \label{base}
\end{enumerate}
where $\pi_i$ is projection onto the $i$th factor of $\overline{\modm}_{h,|I|+1}\times\overline{\modm}_{g-h,|J|+1}$.
We prove below that the properties \eqref{pure}-\eqref{base} uniquely define intersection numbers of the classes $\Theta_{g,n}$ with the classes $\psi_i$ and more generally with classes in the tautological ring $RH^*(\overline{\modm}_{g,n})\subset H^{2*}(\overline{\modm}_{g,n},\bq)$.  
\begin{remark}  \label{gluestab}
One can replace \eqref{glue} by the equivalent property
$$\phi_{\Gamma}^*\Theta_{g,n}=\Theta_{\Gamma}.$$
for any stable graph $\Gamma$, defined in Section~\ref{sec:unique}, of genus $g$ with $n$ external edges.  Here
$$\phi_{\Gamma}:\overline{\modm}_{\Gamma}=\prod_{v\in V(\Gamma)}\overline{\modm}_{g(v),n(v)}\to\overline{\modm}_{g,n},\quad \Theta_{\Gamma}=\prod_{v\in V(\Gamma)}\pi_v^*\Theta_{g(v),n(v)}\in H^*(\overline{\modm}_{\Gamma},\bq)$$
where $\pi_v$ is projection onto the factor $\overline{\modm}_{g(v),n(v)}$.
This generalises \eqref{glue} from 1-edge stable graphs given by $\phi_{\Gamma_{\text{irr}}}=\phi_{\text{irr}}$ and $\phi_{\Gamma_{h,I}}=\phi_{h,I}$. 
\end{remark}
\begin{remark}
The sequence of classes $\Theta_{g,n}$ satisfies many properties of a cohomological field theory (CohFT).  It is essentially a 1-dimensional CohFT with vanishing genus zero classes, not to be confused with Hodge classes which are trivial in genus zero but do not vanish there.  The trivial cohomology class $1\in H^0(\overline{\modm}_{g,n},\bq)$, which is a trivial example of a CohFT known as a topological field theory, satisfies conditions \eqref{pure} and \eqref{glue}, while the forgetful map property \eqref{forget} is replaced by $\Theta_{g,n+1}=\pi^*\Theta_{g,n}$.  
\end{remark}

\begin{theorem}  \label{main}
There exists a class $\Theta_{g,n}$ satisfying \eqref{pure} - \eqref{base} and furthermore any such class satisfies the following properties.
\begin{enumerate}[(I)]
\item 
$\Theta_{g,n}\in H^{4g-4+2n}(\overline{\modm}_{g,n},\bq)$. \label{degree}
\item $\Theta_{0,n}=0$ for all $n$ and $\phi_{\Gamma}^*\Theta_{g,n}=0$ for any $\Gamma$ with a genus 0 vertex. \label{genus0}
\item $\Theta_{g,n}\in H^*(\overline{\modm}_{g,n},\bq)^{S_n}$, i.e. it is symmetric under the $S_n$ action.  \label{symmetric}
\item $\Theta_{1,1}=3\psi_1$. \label{initial}

\vspace{-1.5mm}
\item  For any $\eta\in RH^{g-1}(\overline{\modm}_{g,n})$, 
the intersection number $ \displaystyle\int_{\overline{\modm}_{g,n}}\hspace{-3.5mm}\Theta_{g,n}\eta\in\bq$
is uniquely 

\vspace{-1.8mm} \noindent determined by \eqref{pure} - \eqref{forget} and \eqref{initial}.  \label{unique}
\end{enumerate}
\end{theorem}
The main content of Theorem~\ref{main} is the existence of $\Theta_{g,n}$, the rigidity property \eqref{initial} and the uniqueness property \eqref{unique}.  The existence of $\Theta_{g,n}$ is constructed via the push-forward of a class over the moduli space of spin curves in Section~\ref{existence}.  The rigidity property \eqref{initial} is proven in Section~\ref{sec:unique} by starting with $\Theta_{1,1}=\lambda\psi_1$ and determining constraints on $\lambda$ to arrive at $\lambda=3$ which does occur due to the construction of $\Theta_{g,n}$. The uniqueness result \eqref{unique} involving classes in the tautological ring $RH^*(\overline{\modm}_{g,n})$ is non-constructive since it relies on the existence of non-explicit tautological relations.  The proofs of properties~\eqref{degree} - \eqref{symmetric} are straightforward and presented in Section~\ref{sec:unique}.  Section~\ref{sec:cohft} describes how the classes $\Theta_{g,n}$ naturally combine with any cohomological field theory. 

\begin{remark}
Properties \eqref{pure} - \eqref{base} uniquely define the classes $\Theta_{g,n}$ for $g\leq 4$ and all $n$, but it is not known if they uniquely define the classes $\Theta_{g,n}$ in general. 
Uniqueness would follow from injectivity of the pull-back map to the boundary 
\[
RH^{2g-2}(\overline{\modm}_{g})\to RH^{2g-2}(\partial\overline{\modm}_{g})
\]
which holds for $g=2$, 3 and 4.  It would show that $\Theta_g\in RH^{2g-2}(\overline{\modm}_{g})$ is uniquely determined from its restriction, and consequently $\Theta_{g,n}$ would coincide with the classes constructed in Section~\ref{existence} for all $n\geq 0$. 
\end{remark}

The following conjecture allows one to recursively calculate all intersection numbers $\int_{\overline{\modm}_{g,n}}\Theta_{g,n}\prod_{i=1}^n\psi_i^{m_i}$ via relations coming out of the KdV hierarchy.  Such a recursive calculation would strengthen property \eqref{unique} since intersections of $\Theta_{g,n}$ with $\psi$ classes determine all tautological intersections with $\Theta_{g,n}$ algorithmically.
\begin{conjecture}  \label{kdvconj}
The  function
\[ Z^{\Theta}(\hbar,t_0,t_1,...)=\exp\sum_{g,n,\vec{k}}\frac{\hbar^{g-1}}{n!}\int_{\overline{\modm}_{g,n}}\Theta_{g,n}\cdot\prod_{j=1}^n\psi_j^{k_j}\prod t_{k_j}\]
is the Br\'ezin-Gross-Witten tau function of the KdV hierarchy. \label{thetatau}
\end{conjecture}
The Br\'ezin-Gross-Witten KdV tau function $Z^{\text{BGW}}$ was defined in \cite{BGrExt,GWiPos}.
Conjecture~\ref{kdvconj} has been verified up to $g=7$, i.e. the coefficients of the expansion of the logarithm of the Br\'ezin-Gross-Witten tau function are given by intersection numbers of the classes $\Theta_{g,n}$ for $g\leq 7$ and all $n$.  Progress towards Conjecture~\ref{kdvconj}, including a purely combinatorial formulation that can be stated without reference to the moduli space of stable curves or the KdV hierarchy is discussed in Section~\ref{conj}.
\\

\noindent {\em Acknowledgements.}   I would like to thank Dimitri Zvonkine for his ongoing interest in this work which benefited immensely from many conversations together.  I would also like to thank Vincent Bouchard, Alessandro Chiodo, Alessandro Giacchetto, Oliver Leigh, Danilo Lewanksi, Rahul Pandharipande, Johannes Schmitt, Mehdi Tavakol, Ran Tessler, Ravi Vakil and Edward Witten for useful conversations, the anonymous referee for comments which improved the paper, 
and the Institut Henri Poincar\'e where part of this work was carried out.

\section{Existence}   \label{existence}
The existence of a cohomology class $\Theta_{g,n}\in H^*(\overline{\modm}_{g,n},\bq)$ satisfying \eqref{pure} - \eqref{base} is proven here using the moduli space of stable twisted spin curves $\overline{\modm}_{g,n}^{\rm spin}$ which consist of pairs $(\Sigma,\theta)$ given by a twisted stable curve $\Sigma$ equipped with an orbifold line bundle $\theta$ together with an isomorphism $\theta^{\otimes 2}\cong \omega_{\Sigma}^{\text{log}}$.  See precise definitions below.  We first construct a cohomology class on $\overline{\modm}_{g,n}^{\rm spin}$ and then push it forward to a cohomology class on $\overline{\modm}_{g,n}$.

A stable twisted curve, with group $\bz_2$, is a 1-dimensional orbifold, or stack, $\cc$ such that generic points of $\cc$ have trivial isotropy group and non-trivial orbifold points have isotropy group $\bz_2$.  A stable twisted curve is equipped with a map which forgets the orbifold structure $\rho:\cc\to C$ where $C$ is a stable curve known as the coarse curve of $\cc$.  We say that $\cc$ is smooth if its coarse curve $C$ is smooth.  Each nodal point of $\cc$ (corresponding to a nodal point of $C$) has non-trivial isotropy group, the local picture at each node is $\{xy = 0\}/\bz_2$ with $\bz_2$-action given by $(-1)\cdot(x, y) = (-x, -y)$, and all other points of $\cc$ with non-trivial isotropy group are labeled points of $\cc$.  

A line bundle $L$ over $\cc$ is a locally equivariant bundle over the local charts, such that at each nodal point there is an equivariant isomorphism of fibres.  Hence each orbifold point $p$ associates a representation of $\bz_2$ on $L|_p$ acting by multiplication by $\exp(2\pi i\lambda_p)$ for $\lambda_p=0$ or $\frac12$.  One says $L$ is {\em banded} at $p$ by $\lambda_p$.  The equivariant isomorphism at nodes guarantees that the representations agree on each local irreducible component at the node.

The canonical bundle $\omega_\cc$ of $\cc$ is generated by $dz$ for any local coordinate $z$.  At an orbifold point $x=z^2$ the canonical bundle $\omega_\cc$ is generated by $dz$ hence it is banded by $\frac12$ i.e. $dz\mapsto-dz$ under $z\mapsto -z$.  Over the coarse curve $\omega_C$ is generated by $dx=2zdz$.  In other words $\rho^*\omega_C\not\cong\omega_\cc$ however $\omega_C\cong\rho_*\omega_\cc$.  Moreover, $\deg\omega_C=2g-2$ and 
$$\deg\omega_\cc=2g-2+\frac12n.$$
For $\omega_\cc^{\text{log}}=\omega_\cc(p_1,...,p_n)$, locally $\frac{dx}{x}=2\frac{dz}{z}$ so $\rho^*\omega_C^{\text{log}}\cong\omega_\cc^{\text{log}}$ and $\deg\omega_C^{\text{log}}=2g-2+n=\deg\omega_\cc^{\text{log}}$.

Following \cite{AJaMod}, define the moduli space of stable twisted spin curves by
$$\overline{\modm}_{g,n}^{\rm spin}=\{(\cc,\theta,p_1,...,p_n,\phi)\mid \phi:\theta^2\stackrel{\cong}{\longrightarrow}\omega_{\cc}^{\text{log}}\}.
$$
Here $\omega_{\cc}^{\text{log}}$ and $\theta$ are line bundles over the stable twisted curve $\cc$ with labeled orbifold points $p_j$ and $\deg\theta=g-1+\frac12n$.  The pair $(\theta,\phi)$ is a {\em spin structure} on $\cc$.  The relation $\theta^2\stackrel{\cong}{\longrightarrow}\omega_{\cc}^{\text{log}}$ is possible because the representation associated to $\omega_{\cc}^{\text{log}}$ at $p_i$ is trivial---$dz/z\stackrel{z\mapsto-z}{\longrightarrow}dz/z$.  The equivariant isomorphism of fibres over nodal points forces the balanced condition $\lambda_{p_+}=\lambda_{p_-}$ for $p_\pm$ corresponding to $p$ on each irreducible component.

We can now define a vector bundle over $\overline{\modm}_{g,n}^{\rm spin}$ using the dual bundle $\theta^{\vee}$ on each stable twisted curve.  Denote by $\ce$ the universal spin structure on the universal stable twisted spin curve over $\overline{\modm}_{g,n}^{\rm spin}$.  Given a map $S\to\overline{\modm}_{g,n}^{\rm spin}$, $\ce$ pulls back to $\theta$ giving a family $(\cc,\theta,p_1,...,p_n,\phi)$ where $\pi:\cc\to S$ has stable twisted curve fibres, $p_i:S\to\cc$ are sections with orbifold isotropy $\bz_2$ and $\phi:\theta^2\stackrel{\cong}{\longrightarrow}\omega_{\cc/S}^{\text{log}}=\omega_{\cc/S}(p_1,..,p_n)$.  Consider the push-forward sheaf $\pi_*\ce^{\vee}$ over $\overline{\modm}_{g,n}^{\rm spin}$.  We have
$$\deg\theta^{\vee}=1-g-\frac12n<0.
$$
Furthermore, for any irreducible component $\cc'\stackrel{i}{\to}\cc$, the pole structure on sections of the log canonical bundle at nodes yields $i^*\omega_{\cc/S}^{\text{log}}=\omega_{\cc'/S}^{\text{log}}$.  Hence $\phi':(\theta|_{\cc'})^2\stackrel{\cong}{\longrightarrow}\omega_{\cc'/S}^{\text{log}}$, where $\phi'=i^*\circ\phi|_{\cc'}$.  Since the irreducible component $\cc'$ is stable its log canonical bundle has negative degree and
$$\deg\theta^{\vee}|_{\cc'}<0.
$$
Negative degree of $\theta^{\vee}$ restricted to any irreducible component implies $R^0\pi_*\ce^{\vee}=0$ and the following definition makes sense.
\begin{definition}  \label{obsbun}
Define a bundle $E_{g,n}=-R\pi_*\ce^\vee$ over $\overline{\modm}_{g,n}^{\rm spin}$ with fibre $H^1(\theta^{\vee})$.   
\end{definition}

Represent the band of $\theta$ at the labeled points by $\vec{\sigma}=(\sigma_1,...,\sigma_n)\in\{0,1\}^n$
so that at each labeled point $p_i$ the representation of $\bz_2$ on $\theta|_{p_i}$ is given by multiplication by $\exp(2\pi i\lambda_{p_i})$ for $\lambda_{p_i}=\frac12\sigma_i\in\{0,\tfrac12\}$.  The number of $p_i$ with $\lambda_{p_i}=0$ is even due to evenness of the degree of the push-forward sheaf $|\theta|:=\rho_*\co_\cc(\theta)$ on the coarse curve $C$, \cite{JKVMod}.  In the smooth case, the
boundary type of a spin structure is determined by an associated quadratic form applied to each of the $n$ boundary classes which vanishes since it is a homological invariant, again implying that the number of $p_i$ with $\lambda_{p_i}=0$ is even. The moduli space of stable twisted spin curves decomposes into components  determined by the band $\vec{\sigma}$:
$$\overline{\modm}_{g,n}^{\rm spin}=\bigsqcup_\sigma\overline{\modm}_{g,n,\vec{\sigma}}^{\rm spin}
$$
where $\overline{\modm}_{g,n,\vec{\sigma}}^{\rm spin}$ consists of those spin curves with $\theta$ banded by $\vec{\sigma}$, and the union is over the $2^{n-1}$ functions $\vec{\sigma}$ satisfying $\displaystyle|\vec{\sigma}|+n=\sum_{i=1}^n(\sigma_i+1)\in 2\bz$.  Each component $\overline{\modm}_{g,n,\vec{\sigma}}^{\rm spin}$ is connected except when $|\vec{\sigma}|=n$, in which case there are two connected components determined by their Arf invariant, and known as even and odd spin structures.  This follows from the case of smooth spin curves proven in \cite{NatMod}.

Restricted to $\overline{\modm}_{g,n,\vec{\sigma}}^{\rm spin}$, the bundle $E_{g,n}$ has rank 
\begin{equation}  \label{rank}
\text{rank\ }E_{g,n}=2g-2+\tfrac12(n+|\vec{\sigma}|)
\end{equation}
by the following Riemann-Roch calculation.   Orbifold Riemann-Roch takes into account the representation information \begin{align*}
h^0(\theta^{\vee})-h^1(\theta^{\vee})&=1-g+\deg \theta^{\vee}-\sum_{i=1}^n\lambda_{p_i}=1-g+1-g-\tfrac12n-\tfrac12|\vec{\sigma}|\\
&=2-2g-\tfrac12(n+|\vec{\sigma}|).
\end{align*}
Alternatively, one can use the usual Riemann-Roch calculation on the push-forward of $\theta$ to the underlying coarse curve $C$ as follows.  The sheaf of local sections $\co_\cc(L)$ of any line bundle $L$ on $\cc$ pushes forward to a sheaf $|L|:=\rho_*\co_\cc(L)$ on $C$ which can be identified with the local sections of $L$ invariant under the $\bz_2$ action.  Away from nodal points $|L|$ is locally free, hence a line bundle.  At nodal points, the push-forward $|L|$ is locally free when $L$ is banded by the trivial representation, and $|L|$ is a torsion-free sheaf that is not locally free when $L$ is banded by the non-trivial representation---see \cite{FJRQua}. 
The pull-back bundle is given by 
$$\rho^*(|\theta^\vee|)=\theta^\vee\otimes\bigotimes_{i\in I}\co(-\sigma_ip_i)$$
since locally invariant sections must vanish when the representation is non-trivial.
Hence $\deg|\theta^\vee|=\deg \theta^\vee-\tfrac12|\vec{\sigma}|$. Hence Riemann-Roch on the coarse curve yields the same result as above: $h^0(|\theta^{\vee}|)-h^1(|\theta^{\vee}|)=2-2g-\tfrac12(n+|\vec{\sigma}|)$.  It is proven in \cite{FJRQua} that  $H^i(\theta^{\vee})=H^i(|\theta^{\vee}|)$ so the calculations agree.

We have $h^0(\theta^{\vee})=0$ since $\deg\theta^{\vee}=1-g-\frac12n<0$, and the restriction of $\theta^{\vee}$ to any irreducible component $C'$, say of type $(g',n')$, also has negative degree, $\deg\theta^{\vee}|_{C'}=1-g'-\frac12n'<0$.  Hence $h^1(\theta^{\vee})=2g-2+\frac12(n+|\vec{\sigma}|)$.  Thus $H^1(\theta^{\vee})$ gives fibres of a rank $2g-2+\frac12(n+|\vec{\sigma}|)$ vector bundle.

The analogue of the boundary maps $\phi_{\text{irr}}$ and $\phi_{h,I}$ defined in \eqref{gluing} are multivalued maps 
defined as follows.  Consider a node $p\in\cc$ for $(\cc,\theta,p_1,...,p_n,\phi)\in\overline{\modm}_{g,n}^{\rm spin}$.  Denote the normalisation by $\nu:\tilde{\cc}\to\cc$ with points $p_\pm\in\tilde{\cc}$ that map to the node $p=\nu(p_\pm)$.   When $\tilde{\cc}$ is not connected, the spin structure $\nu^*\theta$ decomposes into two spin structures $\theta_1$ and $\theta_2$.  Any two spin structures $\theta_1$ and $\theta_2$ with bands at $p_+$ and $p_-$ that agree can glue, but not uniquely, to give a spin structure on $\cc$.  This gives rise to a multivalued map, as described in \cite[p.27]{FJRWit}, which uses the fibre product:
$$\begin{array}{ccc}\left(\overline{\modm}_{h,|I|+1}\times\overline{\modm}_{g-h,|J|+1}\right)\times_{\overline{\modm}_{g,n}}\overline{\modm}^{\rm spin}_{g,n}&\to&\overline{\modm}^{\rm spin}_{g,n}\\
\downarrow&&\downarrow\\
\overline{\modm}_{h,|I|+1}\times\overline{\modm}_{g-h,|J|+1}&\to&\overline{\modm}_{g,n}\end{array}
$$
and is given by
$$
\begin{array}{c}
\left(\overline{\modm}_{h,|I|+1}\times\overline{\modm}_{g-h,|J|+1}\right)
\times_{\overline{\modm}_{g,n}}\overline{\modm}^{\rm spin}_{g,n}\\\hspace{1.5cm}\swarrow\hat{\nu}\hspace{3cm}\phi_{h,I}\searrow\\
\qquad\overline{\modm}^{\rm spin}_{h,|I|+1}\times\overline{\modm}^{\rm spin}_{g-h,|J|+1}\hspace{1cm}\dashrightarrow\hspace{1cm}\overline{\modm}^{\rm spin}_{g,n}\end{array}
$$
where $I\sqcup J=\{1,...,n\}$.  The map $\hat{\nu}$ is given by the pull back of the spin structure obtained from $\overline{\modm}^{\rm spin}_{g,n}$ to the normalisation defined by the points of $\overline{\modm}_{h,|I|+1}$ and $\overline{\modm}_{g-h,|J|+1}$. The broken arrow $\dashrightarrow$ represents the multiply-defined map $\phi_{h,I}\circ\hat{\nu}^{-1}$.   The multivalued map $\phi_{h,I}\circ\hat{\nu}^{-1}$ naturally restricts to components $\overline{\modm}^{\rm spin}_{h,|I|+1,\sigma_1}\times\overline{\modm}^{\rm spin}_{g-h,|J|+1,\sigma_2}\dashrightarrow\overline{\modm}^{\rm spin}_{g,n,\vec{\sigma}}$ where $\vec{\sigma}$ and $I$ uniquely determine $\sigma_1$ and $\sigma_2$ since $\theta$ must be banded by $\lambda_p=0$ at an even number of orbifold points, which uniquely determines the band $\lambda_{p_+}=\lambda_{p_-}$ at the separating node.

When $\tilde{\cc}$ is connected, a spin structure $\theta$ on $\cc$ pulls back to a spin structure $\tilde{\theta}=\nu^*\theta$ on $\tilde{\cc}$.  As above, any spin structure $\tilde{\theta}$ with bands at $p_+$ and $p_-$ that agree glues non-uniquely, to give a spin structure on $\cc$, and defines a multiply-defined map which uses the fibre product:
$$\begin{array}{ccc}\overline{\modm}_{g-1,n+2}\times_{\overline{\modm}_{g,n}}\overline{\modm}^{\rm spin}_{g,n}&\to&\overline{\modm}^{\rm spin}_{g,n}\\
\downarrow&&\downarrow\\
\overline{\modm}_{g-1,n+2}&\to&\overline{\modm}_{g,n}\end{array}
$$
and is given by
$$
\begin{array}{c}\overline{\modm}_{g-1,n+2}\times_{\overline{\modm}_{g,n}}\overline{\modm}^{\rm spin}_{g,n}\\\swarrow\hat{\nu}\hspace{2cm}\phi_{\text{irr}}\searrow\\\overline{\modm}^{\rm spin}_{g-1,n+2}\hspace{1cm}\dashrightarrow\hspace{1cm}\overline{\modm}^{\rm spin}_{g,n}\end{array}
$$
Again, $\phi_{\text{irr}}\circ\hat{\nu}^{-1}$ naturally restricts to components $\overline{\modm}^{\rm spin}_{g-1,n+2,\vec{\sigma}'}\dashrightarrow\overline{\modm}^{\rm spin}_{g,n,\vec{\sigma}}$ but unlike the case of $\phi_{h,I}\circ\hat{\nu}^{-1}$ above, $\vec{\sigma}$ does not uniquely determine $\vec{\sigma}'$.  The map $\hat{\nu}$ now depends on $\theta$ and there are two cases corresponding to the decomposition of the fibre product $\overline{\modm}_{g-1,n+2}\times_{\overline{\modm}_{g,n}}\overline{\modm}^{\rm spin}_{g,n,\vec{\sigma}}$ into two components which depend on the behaviour of $\theta$ at the nodal point $p_\pm$.  Either $\theta$ is banded by $\lambda_{p_\pm}=\frac12$, or it is banded by $\lambda_{p_\pm}=0$, corresponding to $\vec{\sigma}'=(\vec{\sigma},1,1)$, respectively $\vec{\sigma}'=(\vec{\sigma},0,0)$.

The bundle $E_{g,n}$ behaves naturally with respect to the boundary divisors.
\begin{lemma}   \label{pullback}
On components where $\theta$ is banded by $\lambda_{p_\pm}=\frac12$ at the node:
$$
\phi_{\text{irr}}^*E_{g,n}\cong  \hat\nu^*E_{g-1,n+2},\quad\phi_{h,I}^*E_{g,n}\cong \hat\nu^*\left(\pi_1^*E_{h,|I|+1}\oplus\pi_2^*E_{g-h,|J|+1}\right)
$$
where $\pi_i$ is projection from $\overline{\modm}^{\rm spin}_{h,|I|+1}\times\overline{\modm}^{\rm spin}_{g-h,|J|+1}$ onto the $i$th factor, $i=1,2$. 
\end{lemma}
\begin{proof}
A spin structure $\tilde{\theta}$ on a connected normalisation $\tilde{\cc}$ has 
$$\deg\tilde{\theta}^{\vee}=1-(g-1)-\frac12(n+2)<0$$ 
and also negative degree on all irreducible components, hence $H^0(\tilde{\theta}^{\vee})=0$.
By Riemann-Roch 
$$h^0(\tilde{\theta}^{\vee})-h^1(\tilde{\theta}^{\vee})=1-(g-1)+\deg\tilde{\theta}^{\vee}-\frac12(n+2)=2-2g-n.$$ 
Hence $\dim H^1(\tilde{\theta}^{\vee})=\dim H^1(\theta^{\vee})$ and the natural map
$$ 0\to H^1(\cc,\theta^{\vee})\to H^1(\tilde{\cc},\tilde{\theta}^{\vee})
$$
is an isomorphism.  In other words $\phi_{\text{irr}}^*E_{g,n}\cong  \nu^*E_{g-1,n+2}$.

The argument is analogous when $\tilde{\cc}$ is not connected and $\lambda_{p_\pm}=\frac12$.  Again $\deg\theta_i^{\vee}<0$, and it has negative degree on all irreducible components, hence $H^0(\theta_i^{\vee})=0$ for $i=1,2$.  By Riemann-Roch 
$$\dim H^1(\theta_1^{\vee})+\dim H^1(\theta_2^{\vee})=\dim H^1(\theta^{\vee})$$ 
so the natural map
$$ 0\to H^1(\cc,\theta^{\vee})\to H^1(\tilde{\cc}_1,\theta_1^{\vee})\oplus H^1(\tilde{\cc}_2,\theta_2^{\vee})
$$
is an isomorphism.  In other words $\phi_{h,I}^*E_{g,n}\cong \hat\nu^*\left(\pi_1^*E_{h,|I|+1}\oplus\pi_2^*E_{g-h,|J|+1}\right)$.
\end{proof}
The pull-back of $E_{g,n}$ to boundary divisors with trivial isotropy at the node is described in the following lemma.
\begin{lemma}  \label{pullback1}
On components where $\theta$ is banded by $\lambda_{p_\pm}=0$ at the node:
\begin{equation}  \label{pb1}
0\to \co_{X_{h,I}}\to \phi_{h,I}^*E_{g,n}\to \hat\nu^*\left(\pi_1^*E_{h,|I|+1}\oplus\pi_2^*E_{g-h,|J|+1}\right)
 \to 0
\end{equation} 
for $X_{h,I}=\left(\overline{\modm}_{h,|I|+1}\times\overline{\modm}_{g-h,|J|+1}\right)
\times_{\overline{\modm}_{g,n}}\overline{\modm}^{\rm spin}_{g,n}$
and
\begin{equation}  \label{pb2}
0\to\co_{X_{\text{irr}}} \to \phi_{\text{irr}}^*E_{g,n}\to\hat\nu^*E_{g-1,n+2}  \to 0
\end{equation}
for $X_{\text{irr}}=\overline{\modm}_{g-1,n+2}\times_{\overline{\modm}_{g,n}}\overline{\modm}^{\rm spin}_{g,n}$.
\end{lemma}
\begin{proof}
When the bundle $\theta$ is banded by $\lambda_{p_\pm}=0$, the map between sheaves of local holomorphic sections
$$\co_C({\theta},U)\to\co_{\tilde{C}}({\nu^*\theta},\nu^{-1}U)$$ 
is not surjective whenever $U\ni p$.  The image consists of local sections that agree, under an identification of fibres, at $p_+$ and $p_-$.  Hence the dual bundle $\theta^\vee$ on $\cc$ is a quotient sheaf
\begin{equation}  \label{normcomp}
0\to I\to\nu^*\theta^\vee\to\theta^\vee\to 0
\end{equation}

\begin{equation}  
0\to \theta^\vee\to\nu_*\nu^*\theta^\vee\to\bc_p\to 0
\end{equation} 
where $\co_{\tilde{C}}(I,U)$ is generated by the element of the dual that sends a local section $s\in \co_{\tilde{C}}({\nu^*\theta},\nu^{-1}U)$ to $s(p_+)-s(p_-)$.  Note that evaluation $s(p_\pm)$ only makes sense after a choice of trivialisation of $\nu^*\theta$ at $p_+$ and $p_-$, but the ideal $I$ is independent of this choice.  The complex \eqref{normcomp} splits as follows.  We can choose a representative $\phi$ upstairs of any element from the quotient space so that $\phi(p_+)=0$, i.e. $\co_C(\theta^\vee,U)$ corresponds to elements of $\co_{\tilde{C}}({\nu^*\theta^\vee},\nu^{-1}U)$ that vanish at $p_+$.  This is achieved by adding the appropriate multiple of $s(p_+)-s(p_-)$ to a given $\phi\in\co_{\tilde{C}}({\nu^*\theta^\vee},\nu^{-1}U)$.  (Note that $\phi(p_-)$ is arbitrary.  One could instead arrange $\phi(p_-)=0$ with $\phi(p_+)$ arbitrary.)  In other words we can identify $\theta^\vee$ with $\nu^*\theta^\vee(-p_+)$ in the complex:
$$0\to \nu^*\theta^\vee(-p_+)\to \nu^*\theta^\vee\to \nu^*\theta^\vee|_{p_+}\to 0.
$$
In a family $\pi:C\to S$, $R^0\pi_*(\nu^*\theta^\vee)=0=R^0\pi_*(\nu^*\theta^\vee(-p_+))$ since $\deg\nu^*\theta^\vee<0$, and it has negative degree on all irreducible components.  Also $R^1\pi_*(\nu^*\theta^\vee|_{p_+})=0$ since $p_+$ has relative dimension 0.  Thus
\begin{equation}  \label{pb3}
0\to R^0\pi_*(\nu^*\theta^\vee|_{p_+})\to R^1\pi_*(\nu^*\theta^\vee(-p_+))\to R^1\pi_*(\nu^*\theta^\vee) \to 0.
\end{equation}
We can identify the dual of the sequence \eqref{pb3} with the sequences \eqref{pb1} and \eqref{pb2} as follows.  For the first term of \eqref{pb3}, we have $\nu^*\theta^\vee|_{p_+}\cong\bc$ canonically via evaluation, hence $R^0\pi_*(\nu^*\theta^\vee|_{p_+})\cong\co_S$.  The second and third terms of \eqref{pb3} are  identified with the corresponding terms of \eqref{pb1}, respectively the second and third terms of \eqref{pb2}, by $\hat\nu^*\left(\pi_1^*E_{h,|I|+1}\oplus\pi_2^*E_{g-h,|J|+1}\right)=R^1\pi_*(\nu^*\theta^\vee)$ and $\phi_{h,I}^*E_{g,n}=R^1\pi_*(\nu^*\theta^\vee(-p_+))$, respectively
$\hat\nu^*E_{g-1,n+2}=$ $R^1\pi_*(\nu^*\theta^\vee)$ and $\phi_{\text{irr}}^*E_{g,n}=R^1\pi_*(\nu^*\theta^\vee(-p_+))$.
\end{proof}
\begin{remark}
In Lemma~\ref{pullback}, the nodal band is $\lambda_{p_\pm}=\frac12$ and we have $\lambda_{p_+}+\lambda_{p_-}=1$.  We see from Lemma~\ref{pullback1} that $\lambda_{p_\pm}=0$ really wants one of $\lambda_{p_\pm}$ to be 1 to preserve $\lambda_{p_+}+\lambda_{p_-}=1$.
\end{remark}

\begin{definition}  \label{fundclass}
For $2g-2+n>0$ define the Chern class 
$$\Omega_{g,n}:=c_{2g-2+n}(E_{g,n})\in H^{4g-4+2n}(\overline{\modm}_{g,n}^{\rm spin},\bq).
$$
\end{definition} 
On the component $\overline{\modm}_{g,n,\vec{\sigma}}^{\rm spin}$ of $\overline{\modm}_{g,n}^{\rm spin}$ for $|\vec{\sigma}|=n$ this defines the top Chern class, or Euler class.   The Chern class vanishes on all other components because by \eqref{rank} the rank of $E_{g,n}=2g-2+\frac12(|\vec{\sigma}|+n)<2g-2+n$ when $|\vec{\sigma}|<n$.  Note that $\Omega_{0,n}=0$ for $n\geq 3$ because rank$(E_{0,n})=n-2$ is greater than $\dim\overline{\modm}_{0,n}^{\rm spin}=n-3$ so its top Chern class vanishes.  

The cohomology classes $\Omega_{g,n}$ behave well with respect to inclusion of strata.
\begin{lemma} \label{omeganat}
$$\phi_{\text{irr}}^*\Omega_{g,n}=\hat\nu^*\Omega_{g-1,n+2},\quad \phi_{h,I}^*\Omega_{g,n}=\hat\nu^*\left(\pi_1^*\Omega_{h,|I|+1}\cdot\pi_2^*\Omega_{g-h,|J|+1}\right).$$ 
\end{lemma}
 
\begin{proof}
When $|\vec{\sigma}|=n$ and $\theta$ is banded by $\frac12$ at the nodal point
then this is an immediate application of Lemma~\ref{pullback}:
$$
\phi_{\text{irr}}^*E_{g,n}\equiv  \hat\nu^*E_{g-1,n+2},\quad\phi_{h,I}^*E_{g,n}\cong \hat\nu^*\left(\pi_1^*E_{h,|I|+1}\oplus\pi_2^*E_{g-h,|J|+1}\right)
$$
and the naturality of $c_{2g-2+n}=c_{\rm top}$.  We have 
\begin{align*}
\phi_{\text{irr}}^*c_{\rm top}(E_{g,n})&=\hat\nu^*c_{\rm top}(E_{g-1,n+2}),\\
\phi_{h,I}^*c_{\rm top}(E_{g,n})&=\hat\nu^*\left(\pi_1^*c_{\rm top}(E_{h,|I|+1})\cdot\pi_2^*c_{\rm top}(E_{g-h,|J|+1})\right).
\end{align*}

When $|\vec{\sigma}|=n$ and $\theta$ is banded by $0$ at the nodal point then the nodal point is necessarily non-separating and we must consider the restriction of $\Omega_{g,n}$ to the component $\overline{\modm}_{g-1,n+2,\vec{\sigma}'}^{\rm spin}$ of $\overline{\modm}_{g-1,n+2}^{\rm spin}$ with $|\vec{\sigma}'|=n$.  On this component we have the exact sequence of Lemma~\ref{pullback1}
$$0\to E_{g-1,n+2}\to \phi_{\text{irr}}^*E_{g,n}\to \co_{\overline{\modm}^{\rm spin}_{g-1,n+2,\vec{\sigma}'}} \to 0
$$
which implies $\phi_{\text{irr}}^*c_{2g-2+n}(E_{g,n})=c_{2g-3+n}(E_{g-1,n+2,\vec{\sigma}'})\cdot c_1(\co_{\overline{\modm}^{\rm spin}_{g-1,n+2,\vec{\sigma}'}})=0$.  This vanishing result is a special case of the pull-back by $\phi_{\text{irr}}^*$ since $\Omega_{g-1,n+2}$ vanishes on $\overline{\modm}_{g-1,n+2,\vec{\sigma}'}^{\rm spin}$ for $|\vec{\sigma}'|=n$.

Finally, when $|\vec{\sigma}|<n$ this is simply the pull-back of the trivial class being trivial, since in each case the restriction to an irreducible component has at least one labeled point with band = 0 so that the right hand side vanishes.
\end{proof}

\smallskip

The cohomology classes $\Omega_{g,n}$ also behave well with respect to the forgetful map 
$$\pi:\overline{\modm}_{g,n+1}^{\rm spin}\to\overline{\modm}_{g,n}^{\rm spin}
$$
which is defined on components with $\theta$ banded by $\frac12$ at $p_{n+1}$ as follows.  Define 
$$\pi(\cc,\theta,p_1,...,p_{n+1},\phi)=(\rho(\cc),\rho_*\theta,p_1,...,p_n,\rho_*\phi)$$ 
where $\rho(\cc)$ forgets the orbifold structure at $p_{n+1}$.  The push-forward sheaf $\rho_*\theta$ consists of local sections invariant under the $\bz_2$ action.  Since the representation at $p_{n+1}$ is given by multiplication by $-1$, any invariant local section must vanish at $p_{n+1}$.  In terms of a local orbifold coordinate $x=z^2$, an invariant section is of the form $zf(x)s$ for $s$ a generator of $\theta$ and its square 
\[(zf(x)s)^2=z^2f(x)^2s^2=xf(x)^2\tfrac{dx}{x}=f(x)^2dx\]
has no pole.  In other words its square is a section of $\omega_\cc^{\text{log}}$ with no pole at $p_{n+1}$ and hence a section of $\omega_{\rho(\cc)}^{\text{log}}=\omega_{\rho(\cc)}(p_1+p_2+...+p_n)$.  Furthermore, we have $\rho_*\theta=\rho_*\{\theta(-p_{n+1})\}$, $\rho^*\rho_*\theta=\{\theta(-p_{n+1})\}$ and $\deg\rho_*\theta=\deg\theta-\frac12$.  The forgetful map $\pi$ is used to denote any family $\pi:\cc\to S$ since $\overline{\modm}_{g,n+1}^{\rm spin}$ is essentially the universal curve of $\overline{\modm}_{g,n}^{\rm spin}$.

Tautological line bundles $L_{p_i}\to\overline{\modm}_{g,n}^{\rm spin}$, $i=1,...,n$ are defined analogously to those defined over $\overline{\modm}_{g,n}$ as follows.  Consider a family $\pi:\cc\to S$ with sections $p_i:S\to\cc$, $i=1,...,n$, and define 
\[L_{p_i}:=p_i^*(\omega_{\cc/S}),\quad  \psi_i=c_1(L_{p_i})\in H^*(\overline{\modm}_{g,n}^{\rm spin},\bq).\]

\begin{lemma}  \label{omegaforget}
$$\Omega_{g,n+1}=-\psi_{p_{n+1}}\cdot \pi^*\Omega_{g,n}.$$ 
\end{lemma}
 
\begin{proof} 
Over a family $\pi:\cc\to S$ where $S\to\overline{\modm}_{g,n+1}^{\rm spin}$ and $\theta\to\cc$ is the universal spin structure (also denoted by $\ce$), tensor the exact sequence of sheaves 
$$0\to \co_\cc(-p_{n+1})\to\co_\cc\to\co_\cc |_{p_{n+1}}\to 0$$
with $\theta^\vee(p_{n+1})$ to get:
$$0\to \theta^\vee\to\theta^\vee(p_{n+1})\to\theta^\vee(p_{n+1}) |_{p_{n+1}}\to 0.$$
This induces a long exact sequence which simplifies to the following short exact sequence:
$$0\to R^0\pi_*\theta^\vee(p_{n+1}) |_{p_{n+1}} \to R^1\pi_*\theta^\vee\to R^1\pi_*\theta^\vee(p_{n+1})\to 0$$
due to the vanishing $R^0\pi_*\theta^\vee(p_{n+1})=0=R^1\pi_*\theta^\vee(p_{n+1}) |_{p_{n+1}}$.  The first of these vanishing results uses the identification $\theta^\vee(p_{n+1})=\pi^*\theta^\vee$ described below together with the vanishing $R^0\pi_*\theta^{\vee}=0$ due to negative degree on each irreducible component described earlier.  The second of these vanishing results uses the simple dimension argument that $R^1\pi_*$ vanishes on the image of $p_{n+1}$ which has relative dimension 0.  

Recall that the forgetful map $(\cc,\theta,p_1,...,p_{n+1},\phi)\mapsto(\pi(\cc),\pi_*\theta,p_1,...,p_n,\pi_*\phi)$ pushes forward $\theta$ via $\pi$ which forgets the orbifold structure at $p_{n+1}$.  As described earlier, $\pi^*\pi_*\theta=\{\theta(-p_{n+1})\}$ since the push-forward gives the sheaf of locally invariant sections which necessarily vanish since the isotropy group acts by multiplication by $-1$.  Hence $\theta^\vee(p_{n+1})=\pi^*\theta^\vee$, which is used to calculate $R^0$ above, and also to give $R^1\pi_*\left(\theta^\vee(p_{n+1})\right)=R^1\pi_*\left(\pi^*\theta^\vee\right)=\pi^*R^1\pi_*\left(\theta^\vee\right)$.  Thus the last two terms of the short exact sequence become $E_{g,n+1}\to\pi^*E_{g,n}$.

For the first term of the short exact sequence, the residue map produces a canonical isomorphism $\omega_{\cc/S}^{\text{log}}\cong\bc$, so in particular
\[\pi_*\omega_{\cc/S}^{\text{log}}|_{p_{n+1}}=\co_S.\]
Thus $\pi_*\theta |_{p_{n+1}}$ and $\pi_*\theta^\vee |_{p_{n+1}}$ define line bundles over $S$ with square $\co_S$ and hence trivial Chern class $c(\pi_*\theta |_{p_{n+1}})=1=c(\pi_*\theta^\vee |_{p_{n+1}})$.   The first term of the short exact sequence $R^0\pi_*\theta^\vee(p_{n+1}) |_{p_{n+1}}$ defines a line bundle $\xi\to S$ with Chern class 
\[c(\xi)=c(R^0\pi_*\co_\cc(p_{n+1}) |_{p_{n+1}})\] 
that fits into the short exact sequence:
\[0\to\xi\to E_{g,n+1}\to\pi^*E_{g,n}\to 0.
\]
The triviality of $\pi_*\omega_{\cc/S}^{\text{log}}|_{p_{n+1}}$ implies
\[ L_{p_{n+1}}=R^0\pi_*\omega_{\cc/S}|_{p_{n+1}}=-R^0\pi_*\co_\cc(p_{n+1}) |_{p_{n+1}}
\]
hence 
\[c(\xi)=\frac{1}{c(L_{p_{n+1}})}=1-\psi_{p_{n+1}}.
\]
The short exact sequence then gives $c_{2g-2+n+1}(E_{g,n+1})=-\psi_{p_{n+1}}\cdot \pi^*c_{2g-2+n}(E_{g,n})$ as required.  
\end{proof} 
\begin{definition}
For $p:\overline{\modm}_{g,n}^{\rm spin}\to\overline{\modm}_{g,n}$ define
$$\Theta_{g,n}=(-1)^n2^{g-1+n}p_*\Omega_{g,n}\in H^{4g-4+2n}(\overline{\modm}_{g,n},\bq).
$$
\end{definition}
Lemma~\ref{omegaforget} and the relation
$$\psi_{n+1}=\frac12p^*\psi_{n+1}$$
proven in \cite[Prop. 2.4.1]{FJRWit}, together with the factor of $2^n$ in the definition of $\Omega_{g,n}$, immediately gives property \eqref{forget} of $\Theta_{g,n}$
$$\Theta_{g,n+1}=\psi_{n+1}\cdot\pi^*\Theta_{g,n}.$$  
Property \eqref{base} of $\Theta_{g,n}$ is given by the following calculation.
\begin{proposition}  \label{theta11}
$\Theta_{1,1}=3\psi_1\in H^2(\overline{\modm}_{1,1},\bq)$.
\end{proposition}
\begin{proof}
A one-pointed twisted elliptic curve $(\ce,p)$ is a one-pointed elliptic curve $(E,p)$ such that $p$ has isotropy $\bz_2$.  The degree of the divisor $p$ in $\ce$ is $\frac12$ and the degree of every other point in $\ce$ is 1.  If $dz$ is a holomorphic differential on $E$ (where $E=\bc/\Lambda$ and $z$ is the identity function on the universal cover $\bc$) then locally near $p$ we have $z=t^2$ so $dz=2tdt$ vanishes at $p$.  In particular, the canonical divisor $(\omega_{\ce})=p$ has degree $\frac12$ and $(\omega^{\text{log}}_{\ce})=(\omega_{\ce}(p))=2p$ has degree 1.  

A spin structure on $\ce$ is a degree $\frac12$ line bundle $\cl$ satisfying $\cl^2=\omega^{\text{log}}_{\ce}$.  Line bundles on $\ce$ correspond to divisors on $\ce$ up to linear equivalence.  Note that meromorphic functions on $\ce$ are exactly the meromorphic functions on $E$.  The four spin structures on $\ce$ are given by the divisors $\theta_0=p$ and $\theta_i=q_i-p$, $i=1,2,3$, where $q_i$ is a non-trivial order 2 element in the group $E$ with identity $p$.  Clearly $\theta_0^2=2p=\omega^{\text{log}}_{\ce}$.  For $i=1,2,3$, $\theta_i^2=2q_i-2p\sim 2p$ since there is a meromorphic function $\wp(z)-\wp(q_i)$ on $E$ with a double pole at $p$ and a double zero at $q_i$.  Its divisor on $\ce$ is $2q_i-4p$, since $p$ has isotropy $\bz_2$, hence $2q_i-2p\sim 2p$.

Since $H^2(\overline{\modm}_{1,1},\bq)$ is generated by $\psi_1$ it is enough to calculate $\int_{\overline{\modm}_{1,1}}\Theta_{1,1}$.  The Chern character of the push-forward bundle $E_{1,1}$ is calculated via the Grothendieck-Riemann-Roch theorem
$$ \text{Ch}(R\pi_*\ce^\vee)=\pi_*(\text{Ch}(\ce^\vee)\text{Td}(\omega_\pi^\vee)).
$$
In fact we need to use the orbifold Grothendieck-Riemann-Roch theorem \cite{ToeThe}.  The calculation we need is a variant of the calculation in \cite[Theorem 6.3.3]{FJRWit} which applies to $\ce$ such that $\ce^2=\omega_{\cc}^{\text{log}}$ instead of $\ce^\vee$.  Importantly, this means that the Todd class has been worked out, and it remains to adjust the $\text{Ch}(\ce^\vee)$ term.  We get
\begin{align*}
\int_{\overline{\modm}_{1,1}}p_*c_1(E_{1,1})&=- \text{Ch}(R\pi_*\ce^\vee)\\
&=-2\int_{\overline{\modm}_{1,1}}\left[\frac{11}{24}\kappa_1+\frac{1}{24}\psi_1+\frac12 \left(-\frac{1}{24}+\frac{1}{12}\right)(i_{\Gamma})_*(1)\right]\\
&=-2\left(\frac{11}{24^2}+\frac{1}{24^2}+\frac12\cdot \frac{1}{24}\cdot\frac12\right)=-\frac{1}{16}
\end{align*}
which agrees with
$$-\int_{\overline{\modm}_{1,1}}\frac32\psi_1=-\frac32\cdot\frac{1}{24}=-\frac{1}{16}.
$$
Hence $p_*c_1(E_{1,1})=-\frac32\psi_1$ and $\Theta_{1,1}=-2p_*c_1(E_{1,1})=3\psi_1$.
\end{proof}
\begin{proposition}
The classes $\Theta_{g,n}\in H^{4g-4+2n}(\overline{\modm}_{g,n},\bq)$ satisfy property \eqref{glue}.
\end{proposition}
\begin{proof}
The two properties \eqref{glue} of $\Theta_{g,n}$, follow from the analogous properties for $\Omega_{g,n}$.  This uses the relationship between compositions of pull-backs and push-forwards in the following diagrams:
\begin{center}
\begin{tikzpicture}[scale=0.5]
\draw (0,0) node {$\overline{\modm}_{g-1,n+2}^{\rm spin}$};
\draw [->, line width=1pt,dashed] (2,0)--(4,0);
\draw (3,.5) node {$\phi_{\text{irr}}\circ\hat{\nu}^{-1}$};
\draw (3,-3.5) node {$\phi_{\text{irr}}$};
\draw (6,0) node {$\overline{\modm}_{g,n}^{\rm spin}$};
\draw [->, line width=1pt] (0,-1)--(0,-3);
\draw (-.5,-2) node {$p$};
\draw (0,-4) node {$\overline{\modm}_{g-1,n+2}$};
\draw [->, line width=1pt] (2,-4)--(4,-4);
\draw (6,-4) node {$\overline{\modm}_{g,n}$};
\draw [->, line width=1pt] (6,-1)--(6,-3);
\draw (5.5,-2) node {$p$};
\end{tikzpicture}
\hspace{3cm}
\begin{tikzpicture}[scale=0.5]
\draw (0,0) node {$\overline{\modm}_{h,|I|+1}^{\rm spin}\times \overline{\modm}_{g-h,|J|+1}^{\rm spin}$};
\draw [->, line width=1pt,dashed] (4,0)--(6,0);
\draw (5,.5) node {$\phi_{h,I}\circ\hat{\nu}^{-1}$};
\draw (5,-3.5) node {$\phi_{h,I}$};
\draw (8,0) node {$\overline{\modm}_{g,n}^{\rm spin}$};
\draw [->, line width=1pt] (0,-1)--(0,-3);
\draw (-.5,-2) node {$p$};
\draw (0,-4) node {$\overline{\modm}_{h,|I|+1}\times\overline{\modm}_{g-h,|J|+1}$};
\draw [->, line width=1pt] (4,-4)--(6,-4);
\draw (8,-4) node {$\overline{\modm}_{g,n}$};
\draw [->, line width=1pt] (8,-1)--(8,-3);
\draw (7.5,-2) node {$p$};
\end{tikzpicture}
\end{center}
where the broken arrows signify multiply-defined maps which are defined above using fibre products.

On cohomology, we have $\phi_{\text{irr}}^*p_*=2p_*\hat\nu_*\phi_{\text{irr}}^*$ and $\phi_{h,I}^*p_*=2p_*\hat\nu_*\phi_{h,I}^*$ where the factor of 2 is due to the degree of $\hat\nu$ ramification of $p$ and the isotropy of the orbifold divisor---see (39) in \cite{JKVMod}.  Hence 
\begin{align*}
\phi_{\text{irr}}^*\Theta_{g,n}&=\phi_{\text{irr}}^*p_*(-1)^n2^{g-1+n}\Omega_{g,n}=2p_*\hat\nu_*\phi_{\text{irr}}^*(-1)^n2^{g-1+n}\Omega_{g,n}\\
&=p_*(-1)^{n+2}2^{g+n}\Omega_{g-1,n+2}=\Theta_{g-1,n+2}
\end{align*}
and similarly $\phi_{h,I}^*\Theta_{g,n}=\pi_1^*\Theta_{h,|I|+1}\cdot\pi_2^* \Theta_{g-h,|J|+1}$ which uses 
\[2\cdot(-1)^n2^{g-1+n}=(-1)^n2^{g+n}=(-1)^{|I|+1}2^{h-1+|I|+1}(-1)^{|J|+1}2^{g-h-1+|J|+1}.\]
\end{proof}

\begin{remark}
The construction of $\Omega_{g,n}$ should also follow from the cosection construction in \cite{CLLWit} using the moduli space of spin curves with fields 
$$\overline{\modm}_{g,n}(\bz_2)^p=\{(C,\theta,\rho)\mid (C,\theta)\in\overline{\modm}_{g,n}^{\rm spin},\ \rho\in H^0(\theta)\}.$$
A cosection of the pull-back of $E_{g,n}$ to $\overline{\modm}_{g,n}(\bz_2)^p$ is given by $\rho^{-3}$ since it pairs well with $H^1(\theta)$---we have $\rho^{-3}\in H^0((\theta^\vee)^3)$ while $H^1(\theta)\cong H^0(\omega\otimes\theta^\vee)^\vee=H^0((\theta^\vee)^3)^\vee$.  Using the cosection $\rho^{-3}$ a virtual fundamental class is constructed in \cite{CLLWit} that likely gives rise to $\Omega_{g,n}\in H^{4g-4+2n}(\overline{\modm}_{g,n}^{\rm spin},\bq)$.  The virtual fundamental class is constructed away from the zero set of $\rho$.
\end{remark}

\section{Uniqueness}   \label{sec:unique}
The degree property \eqref{degree} of Theorem~\ref{main}, $\Theta_{g,n}\in H^{4g-4+2n}(\overline{\modm}_{g,n},\bq)$, proven below, implies the initial value
\[\Theta_{1,1}=\lambda\psi,\ \lambda\in\bq.
\]
It leads leads to uniqueness of intersection numbers $\displaystyle \int_{\overline{\modm}_{g,n}}\Theta_{g,n}\prod_{i=1}^n\psi_i^{m_i}\prod_{j=1}^N\kappa_{\ell_j}$ via a reduction argument, and consequently property \eqref{unique} of Theorem~\ref{main}.  The proofs in this section of properties \eqref{genus0}, \eqref{symmetric} and \eqref{unique} apply for any $\lambda\neq 0$.  We finish the section with a rigidity result given by Theorem~\ref{rigid} proving that necessarily $\lambda=3$.

We first prove the following lemma which will be needed later.
\begin{lemma}  \label{nonzero}
Properties \eqref{pure}-\eqref{base} imply that $\Theta_{g,n}\neq 0$ for $g>0$ and all $n$.
\end{lemma}
\begin{proof}
We have $\Theta_{1,1}=a$ or $\Theta_{1,1}=a\psi_1$ for $a\neq 0$ by \eqref{pure} and \eqref{base}.   Using the pull-back property \eqref{forget} together with the equality $\psi_n\psi_i=\psi_n\pi^*\psi_i$ for $i<n$, we have $\Theta_{1,n}=a\psi_2...\psi_n$ or $\Theta_{1,n}=a\psi_1\psi_2...\psi_n$ hence $(1+\psi_1)\Theta_{1,n}=a\psi_1\psi_2...\psi_n$ and $\int_{\overline{\modm}_{g,n}}(1+\psi_1)\Theta_{1,n}=a(n-1)!/24$, proving $\Theta_{1,n}\neq 0$.  

Now we proceed by induction on $g$.  For the base case of $g=1$, we have $\Theta_{1,n}\neq 0$ for all $n>0$.  Assume $\Theta_{h,n}\neq 0$ for $0<h<g$ and all $n$.  For $g>1$, let $\Gamma$ be the stable graph consisting of a genus $g-1$ vertex attached by a single edge to a genus 1 vertex with $n$ labeled leaves (denoted {\em ordinary} leaves in Section~\ref{twloop}).  Then by \eqref{glue}
$$\phi_{\Gamma}^*\Theta_{g,n}=\Theta_{g-1,1}\otimes\Theta_{1,n+1}
$$
which is non-zero since $\Theta_{g-1,1}\neq 0$ by the inductive hypothesis and $\Theta_{1,n+1}\neq 0$ by the calculation above.
\end{proof}

\begin{proof}[Proof of \eqref{degree}]

Write \[d(g,n)=\text{degree}(\Theta_{g,n})\] which exists by \eqref{pure}.  Note that the degree here is half the cohomological degree so $\Theta_{g,n}\in H^{2d(g,n)}(\overline{\modm}_{g,n},\bq)$.  Using \eqref{glue}, $\phi_{\text{irr}}^*\Theta_{g,n}=\Theta_{g-1,n+2}$ implies that 
\[d(g,n)=d(g-1,n+2)\] 
since $\Theta_{g-1,n+2}\neq 0$ by Lemma~\ref{nonzero}.  Hence 
$d(g,n)=f(2g-2+n)$ 
is a function of $2g-2+n$.  Similarly, using \eqref{glue}, $\phi_{h,I}^*\Theta_{g,n}=\Theta_{h,|I|+1}\otimes \Theta_{g-h,|J|+1}$ implies that 
$f(a+b)=f(a)+f(b)=(a+b)f(1)$ since $\Theta_{h,|I|+1}\neq 0$ and $\Theta_{g-h,|J|+1}\neq 0$ again by Lemma~\ref{nonzero}.
Hence 
\[d(g,n)=(2g-2+n)k\] 
for an integer $k$.  But  $d(g,n)\leq 3g-3+n$ implies $k\leq 1$.  When $k=0$, this gives $\deg\Theta_{g,n}=0$ which contradicts  \eqref{forget} together with Lemma~\ref{nonzero} hence $k=1$ and $\deg\Theta_{g,n}=2g-2+n$.
\end{proof}
\begin{proof}[Proof of \eqref{genus0}]
This is an immediate consequence of \eqref{degree} since 
\[\deg\Theta_{0,n}=n-2>n-3=\dim\overline{\modm}_{0,n}
\] 
hence $\Theta_{0,n}=0$.  For any stable graph $\Gamma$ with a genus 0 vertex, Remark~\ref{gluestab} gives $\phi_{\Gamma}^*\Theta_{g,n}=\Theta_{\Gamma}=\prod_{v\in V(\Gamma)}\pi_v^*\Theta_{g(v),n(v)}=0$ since the genus 0 vertex contributes a factor of 0 to the product.
\end{proof}
\begin{proof}[Proof of \eqref{symmetric}]
Property \eqref{forget} implies that 
\[\Theta_{g,n}=\prod_{i=1}^n\psi_i\cdot\pi^*\Theta_g\] 
where $\pi:\overline{\modm}_{g,n}\to\overline{\modm}_{g}$ is the forgetful map.  Since $\pi^*\omega\in H^*(\overline{\modm}_{g,n},\bq)^{S_n}$ for any class $\omega\in H^*(\overline{\modm}_{g},\bq)$ and clearly $\prod_{i=1}^n\psi_i\in H^*(\overline{\modm}_{g,n},\bq)^{S_n}$ we have $\Theta_{g,n}\in H^*(\overline{\modm}_{g,n},\bq)^{S_n}$ as required.
\end{proof}
The proof of \eqref{unique} follows from the special case of the intersection of $\Theta_{g,n}$ with a polynomial in $\kappa$ and $\psi$ classes.
\begin{proposition}   \label{th:unique}
For any $\Theta_{g,n}$ satisfying properties \eqref{pure} - \eqref{forget}, the intersection numbers
\begin{equation}  \label{corr}
\int_{\overline{\modm}_{g,n}}\Theta_{g,n}\prod_{i=1}^n\psi_i^{m_i}\prod_{j=1}^N\kappa_{\ell_j}
\end{equation}
are uniquely determined from the initial condition $\Theta_{1,1}=\lambda\psi_1$ for $\lambda\in\bq$. \end{proposition}
\begin{proof}
For $n>0$, we will push forward the integral \eqref{corr} via the forgetful map $\pi:\overline{\modm}_{g,n}\to\overline{\modm}_{g,n-1}$ as follows.  Consider first the case when there are no $\kappa$ classes.  The presence of $\psi_n$ in $\Theta_{g,n}=\psi_n\cdot\pi^*\Theta_{g,n-1}$ gives
$$\Theta_{g,n}\psi_k=\Theta_{g,n}\pi^*\psi_k,\quad k<n
$$
since $\psi_n\psi_k=\psi_n\pi^*\psi_k$ for $k<n$.  Hence
\begin{align*}
\int_{\overline{\modm}_{g,n}}\Theta_{g,n}\prod_{i=1}^n\psi_i^{m_i}&=
\int_{\overline{\modm}_{g,n}}\pi^*\Big(\Theta_{g,n-1}\prod_{i=1}^{n-1}\psi_i^{m_i}\Big)\psi_n^{m_n+1}\\
=\int_{\overline{\modm}_{g,n-1}}&\pi_*\left\{\pi^*\Big(\Theta_{g,n-1}\prod_{i=1}^{n-1}\psi_i^{m_i}\Big)\psi_n^{m_n+1}\right\}
=\int_{\overline{\modm}_{g,n-1}}\Theta_{g,n-1}\prod_{i=1}^{n-1}\psi_i^{m_i}\kappa_{m_n}
\end{align*}
so we have reduced an intersection number over $\overline{\modm}_{g,n}$ to an intersection number over $\overline{\modm}_{g,n-1}$.
In the presence of $\kappa$ classes, replace $\kappa_{\ell_j}$ by $\kappa_{\ell_j}=\pi^*\kappa_{\ell_j}+\psi_n^{\ell_j}$ and repeat the push-forward as above on all summands.  By induction, we see that for $g>1$
$$\int_{\overline{\modm}_{g,n}}\Theta_{g,n}\prod_{i=1}^n\psi_i^{m_i}\prod_{j=1}^N\kappa_{\ell_j}=\int_{\overline{\modm}_{g}}\Theta_{g}\cdot p(\kappa_1,\kappa_2,...,\kappa_{3g-3})$$
i.e. the intersection number \eqref{corr} reduces to an intersection number over $\overline{\modm}_{g}$ of $\Theta_g$ times a polynomial in the $\kappa$ classes.  When $g=1$, the right hand side is instead $\int_{\overline{\modm}_{1,1}}\Theta_{1,1}\cdot p$ for $p\in\bq$ a constant.

For $g>1$, by \eqref{degree} $\deg\Theta_g=2g-2$, so we may assume the polynomial $p$ consists only of terms of homogeneous degree $g-1$ (where $\deg\kappa_r=r$).  But by a result of Faber and Pandharipande \cite[Proposition 2]{FPaRel}, which strengthens Looijenga's theorem \cite{LooTau}, a homogeneous degree $g-1$ monomial in the $\kappa$ classes is equal in the tautological ring to the sum of boundary terms, i.e. the sum of push-forwards of polynomials in $\psi$ and $\kappa$ classes by the the maps $(\phi_\Gamma)_*$.  Such relations arise from Pixton's relations and are described algorithmically in \cite{CGJZPow}.  Now property \eqref{glue} of $\Theta_g$, shows that the pull-back of $\Theta_g$ to these boundary terms is $\Theta_{g',n'}$ for $g'<g$ so we have expressed \eqref{corr} as a sum of integrals of $\Theta_{g',n'}$ against $\psi$ and $\kappa$ classes.  By induction, one can reduce to the integral $\int_{\overline{\modm}_{1,1}}\Theta_{1,1}=\frac{\lambda}{24}$ and the proposition is proven.
\end{proof}
A consequence of Proposition~\ref{th:unique} is property \eqref{unique} of Theorem~\ref{main} stated as Corollary~\ref{intaut} below.  Let us first recall the definition of tautological classes in $H^*(\overline{\modm}_{g,n},\bq)$.  Dual to any point $(C,p_1,...,p_n)\in\overline{\modm}_{g,n}$ is its stable graph $\Gamma$ with vertices $V(\Gamma)$ representing irreducible components of $C$, internal edges representing nodal singularities and a (labeled) external edge for each $p_i$.  Each vertex is labeled by a genus $g(v)$ and has valency $n(v)$.  The genus of a stable graph is $g(\Gamma)=\displaystyle b_1(\Gamma)+\hspace{-3mm}\sum_{v\in V(\Gamma)}g(v)$.  

The {\em strata algebra} $S_{g,n}$ is a finite-dimensional vector space over $\bq$ with basis given by isomorphism classes of pairs $(\Gamma, \omega)$, for $\Gamma$ a stable graph of genus $g$ with $n$ external edges and $\omega\in H^*(\overline{\modm}_{\Gamma},\bq)$ a product of $\kappa$ and $\psi$ classes in each $\overline{\modm}_{g(v),n(v)}$ for each vertex $v\in V(\Gamma)$.  There is a natural map 
\[q:S_{g,n}\to H^*(\overline{\modm}_{g,n},\bq)\] 
defined by the push-forward $q(\Gamma, \omega)=\phi_{\Gamma}^*(\omega)\in H^*(\overline{\modm}_{g,n},\bq)$.  The map $q$ allows one to define a multiplication on $S_{g,n}$, essentially coming from intersection theory in $\overline{\modm}_{g,n}$, which can be described purely graphically.  The image $q(S_{g,n})\subset H^*(\overline{\modm}_{g,n},\bq)$ is the tautological ring $RH^*(\overline{\modm}_{g,n})$ and an element of the kernel of $q$ is a tautological relation.  See \cite[Section~0.3]{PPZRel} for a detailed description of $S_{g,n}$.  

\begin{corollary}  \label{intaut}
For all $\eta\in RH^*(\overline{\modm}_{g,n})$,
$ \displaystyle\int_{\overline{\modm}_{g,n}}\hspace{-3.5mm}\Theta_{g,n}\eta\in\bq$
is uniquely determined by properties \eqref{pure} - \eqref{forget} and \eqref{initial}. 
\end{corollary}
\begin{proof}
The tautological ring $RH^*(\overline{\modm}_{g,n})$ consists of polynomials in the classes $\kappa_i$, $\psi_i$ and boundary classes, which are pushforwards under $(\phi_\Gamma)_*$ of polynomials in $\kappa_i$ and $\psi_i$.  By the natural restriction property \eqref{glue} satisfied by $\Theta_{g,n}$, 
given a monomial in $\kappa$ and $\psi$ classes $\omega\in H^*(\overline{\modm}_{\Gamma},\bq)$,
\[  \int_{\overline{\modm}_{g,n}}\Theta_{g,n}\cdot(\phi_{\Gamma})_*(\omega)=\int_{\overline{\modm}_{\Gamma}}\phi_{\Gamma}^*(\Theta_{g,n})\cdot\omega=\int_{\overline{\modm}_{\Gamma}}\Theta_{\Gamma}\cdot\omega=\frac{1}{|\text{Aut}\Gamma|}\prod_{v\in\Gamma}w(v)
\]
The final term is a product over the vertices of $\Gamma$ of intersections $\Theta$ classes with monomials in $\kappa$ and $\psi$ classes
$w(v)=\int_{\overline{\modm}_{g(v),n(v)}}\Theta_{g(v),n(v)}\cdot\prod_{i=1}^{n(v)}P_v(\{\psi_i,\kappa_j\})$
which by Proposition~\ref{th:unique} are uniquely determined by \eqref{pure} - \eqref{forget} and \eqref{initial}.
\end{proof}
\begin{remark}   \label{removekappa}
The intersection numbers $\int_{\overline{\modm}_{g,n}}\Theta_{g,n}\prod_{i=1}^n\psi_i^{m_i}\prod_{j=1}^N\kappa_{\ell_j}$ can be calculated algorithmically from the intersection numbers $\int_{\overline{\modm}_{g,n}}\Theta_{g,n}\prod_{i=1}^n\psi_i^{m_i}$ with no $\kappa$ classes.  This essentially reverses the reduction shown in the proof of Proposition~\ref{th:unique}.  Explicitly,
for $\pi:\overline{\modm}_{g,n+N}\to\overline{\modm}_{g,n}$ and ${\bf m}=(m_1,...,m_N)$ define a polynomial in $\kappa$ classes by 
$$R_{\bf m}(\kappa_1,\kappa_2,...)=\pi_*\Big(\psi_{n+1}^{m_1+1}...\psi_{n+N}^{m_N+1}\Big)$$
so, for example, $R_{(m_1,m_2)}=\kappa_{m_1}\kappa_{m_1}+\kappa_{m_1+m_2}$.
Then
\begin{align}  \label{thetakappa}
\Theta_{g,n}\cdot R_{\bf m}&=\Theta_{g,n}\cdot \pi_*\Big(\psi_{n+1}^{m_1+1}...\psi_{n+N}^{m_N+1}\Big)
=\pi_*\Big(\pi^*\Theta_{g,n}\cdot\psi_{n+1}^{m_1+1}...\psi_{n+N}^{m_N+1}\Big)\\
&=\pi_*\Big(\Theta_{g,n+N}\cdot\psi_{n+1}^{m_1}...\psi_{n+N}^{m_N}\Big). \nonumber
\end{align}
The polynomials $R_{\bf m}(\kappa_1,\kappa_2,...)$ generate all polynomials in the $\kappa_i$ so \eqref{thetakappa} can be used to remove any $\kappa$ class.
\end{remark}
The following example demonstrates Proposition~\ref{th:unique} with an explicit genus 2 relation.
\begin{example}  \label{gen2rel}
A genus two relation proven by Mumford \cite[equation (8.5)]{MumTow}, relating $\kappa_1$ and the divisors defined by the double covers ${\overline{\modm}_{1,1}}\times{\overline{\modm}_{1,1}}\to\modm_{\Gamma_1}$ and ${\overline{\modm}_{1,2}}\to\modm_{\Gamma_2}$ in $\overline{\modm}_{2}$, labeled by stable graphs $\Gamma_i$ is given by
$$\kappa_1-\frac{7}{5}[\modm_{\Gamma_1}]-\frac{1}{5}[\modm_{\Gamma_2}]=0$$
which induces the relation 
$$\Theta_{2}\cdot\kappa_1-\frac{7}{5}\Theta_{2}\cdot[\modm_{\Gamma_1}]-\frac{1}{5}\Theta_{2}\cdot[\modm_{\Gamma_2}]=0.$$ 
Property \eqref{glue} of $\Theta_{g,n}$ yields
\begin{align*}
\int_{\overline{\modm}_{2}}\Theta_{2}\cdot[\modm_{\Gamma_1}]&=\int_{\modm_{\Gamma_1}}\phi_{\Gamma_1}^*\Theta_{2}=\int_{\overline{\modm}_{1,1}}\Theta_{1,1}\cdot\int_{\overline{\modm}_{1,1}}\Theta_{1,1}\cdot\frac{1}{|\text{Aut}(\Gamma_1)|}\\
\int_{\overline{\modm}_{2}}\Theta_{2}\cdot[\modm_{\Gamma_2}]&=\int_{\modm_{\Gamma_2}}\phi_{\Gamma_2}^*\Theta_{2}=\int_{\overline{\modm}_{1,2}}\Theta_{1,2}\cdot\frac{1}{|\text{Aut}(\Gamma_2)|}
\end{align*}
hence the relation on the level of intersection numbers is given by
\begin{equation*}  
\int_{\overline{\modm}_{2}}\Theta_{2}\cdot\kappa_1-\frac{7}{5}\cdot\int_{\overline{\modm}_{1,1}}\Theta_{1,1}\cdot\int_{\overline{\modm}_{1,1}}\Theta_{1,1}\cdot\frac{1}{|\text{Aut}(\Gamma_1)|}-\frac{1}{5}\cdot\int_{\overline{\modm}_{1,2}}\Theta_{1,2}\cdot\frac{1}{|\text{Aut}(\Gamma_2)|}=0.
\end{equation*}
We have $\int_{\overline{\modm}_{1,1}}\Theta_{1,1}=\frac{\lambda}{24}=\int_{\overline{\modm}_{1,2}}\Theta_{1,2}$ from \eqref{forget}, and $|\text{Aut}(\Gamma_1)|=\hspace{-.5mm}2\hspace{-.5mm}=|\text{Aut}(\Gamma_2)|$.
Hence
\begin{align*}  
\int_{\overline{\modm}_{2}}\Theta_{2}\cdot\kappa_1 &=\frac{7}{5}\cdot\int_{\overline{\modm}_{1,1}}\Theta_{1,1}\cdot\int_{\overline{\modm}_{1,1}}\Theta_{1,1}\cdot\frac{1}{|\text{Aut}(\Gamma_1)|}+\frac{1}{5}\cdot\int_{\overline{\modm}_{1,2}}\Theta_{1,2}\cdot\frac{1}{|\text{Aut}(\Gamma_2)|}\\&=\frac{7}{5}\cdot\left(\frac{\lambda}{24}\right)^2\cdot\frac{1}{2}+\frac{1}{5}\cdot\frac{\lambda}{24}\cdot\frac{1}{2}=\frac{7\lambda^2+24\lambda}{5760}. \nonumber
\end{align*} 
\end{example}
Until now $\Theta_{1,1}=\lambda\psi_1$ for any non-zero $\lambda\in\bq$.  The following theorem proves the rigidity condition \eqref{initial} that $\lambda=3$.  The proof of the theorem relies on the fact that for low genus and small $n$, the cohomology is tautological.  This allows us to work in the tautological ring in order to construct $\Theta_{g,n}$ from properties \eqref{pure} - \eqref{base}. 
\begin{theorem}   \label{rigid}
Let $\Theta_{g,n}\in H^*(\overline{\modm}_{g,n},\bq)$ satisfy \eqref{pure} - \eqref{base} and set the initial condition to be $\Theta_{1,1}=\lambda\psi_1\neq0$.  Then $\lambda=3$.
\end{theorem}
\begin{proof}
The existence proof in Section~\ref{existence} shows that $\lambda=3$ is possible but it does not exclude other values.  The strategy of proof of this theorem is to attempt to construct classes, beginning with the initial condition $\Theta_{1,1}=\lambda\psi_1$.  Importantly, condition \eqref{forget} determines $\Theta_{g,n}$ for all $n>0$ uniquely from $\Theta_g$ so the main calculation occurs over $\overline {\modm}_g$.  We consider classes in $RH^{2g-2}(\overline {\modm}_g)$ since for small values of $g$ it is known that $H^{2*}(\overline{\modm})_g=RH^*(\overline{\modm}_{g})$.
The essential idea is as follows.  
A class $\Theta_g\in H^{2g-2}(\overline{\modm}_{g},\bq)$ pulls back under boundary maps to $\Theta_{g-1,2}$ and $\Theta_{g-1,1}\otimes\Theta_{1,1}$.  The relationship 
\[\Theta_{g-1,2}=\psi_2\pi^*\Theta_{g-1,1}
\]
constrains the class $\Theta_g$.  We find that $\Theta_2$ exists (and hence also $\Theta_{2,n}$ exists for all $n$) for all $\lambda\in\bq$, but that $\Theta_3$ (and $\Theta_{3,n}$) exists only for $\lambda=3$ or $\lambda=-11/15$.  The existence of $\Theta_4$ constrains $\lambda$ further, allowing only $\lambda=3$.\\

\noindent $\mathbf{g=1}$.  From $\Theta_{1,1}=\lambda\psi_1$, condition \eqref{forget} yields
\[\Theta_{1,n}=\lambda\psi_1\psi_2...\psi_n
\]
since $\psi_n\psi_j=\psi_n\pi^*\psi_j$ for any $j<n$.\\

\noindent $\mathbf{g=2}$. The cohomology group $H^4(\overline{\modm}_{2},\bq)$ has basis $\{\kappa_1^2,\kappa_2\}$.  Set $\Theta_2=a_{11}\kappa_1^2+a_2\kappa_2$ and deduce $a_{11}$ and $a_2$ from restriction to $\modm_{\Gamma_i}\subset\overline{\modm}_{2}$ for $i=1,2$ defined in Example~\ref{gen2rel}.  Since $\kappa_2\cdot\modm_{\Gamma_1}=0$ we deduce that $a_{11}=\frac12\lambda^2$ and restriction to $\modm_{\Gamma_2}$ then uniquely determines
\[\Theta_2=\tfrac12\lambda^2\kappa_1^2+(\lambda-\tfrac32\lambda^2)\kappa_2. 
\]
Commutativity of the boundary maps with the forgetful map shown in the diagrams below implies that the classes $\Theta_{2,n}=\psi_1...\psi_n\pi^*\Theta_2$ restrict consistently to the boundary to give the correct genus 1 classes $\Theta_{1,n'}$ for all $\lambda\in\bq$.
\begin{center}
\begin{tikzpicture}[scale=0.5]
\draw (0,0) node {$\overline{\modm}_{g-1,n+2}$};
\draw [->, line width=1pt] (2,0)--(4,0);
\draw (3,.5) node {$\phi_{\text{irr}}$};
\draw (3,-3.5) node {$\phi_{\text{irr}}$};
\draw (6,0) node {$\overline{\modm}_{g,n}$};
\draw [->, line width=1pt] (0,-1)--(0,-3);
\draw (-.5,-2) node {$\pi$};
\draw (0,-4) node {$\overline{\modm}_{g-1,2}$};
\draw [->, line width=1pt] (2,-4)--(4,-4);
\draw (6,-4) node {$\overline{\modm}_{g}$};
\draw [->, line width=1pt] (6,-1)--(6,-3);
\draw (5.5,-2) node {$\pi$};
\draw (0+\shift,0) node {$\overline{\modm}_{h,|I|+1}\times \overline{\modm}_{g-h,|J|+1}$};
\draw [->, line width=1pt] (4+\shift,0)--(6+\shift,0);
\draw (5+\shift,.5) node {$\phi_{h,I}$};
\draw (5+\shift,-3.5) node {$\phi_{h}$};
\draw (8+\shift,0) node {$\overline{\modm}_{g,n}$};
\draw [->, line width=1pt] (0+\shift,-1)--(0+\shift,-3);
\draw (-.5+\shift,-2) node {$\pi$};
\draw (0+\shift,-4) node {$\overline{\modm}_{h,1}\times\overline{\modm}_{g-h,1}$};
\draw [->, line width=1pt] (4+\shift,-4)--(6+\shift,-4);
\draw (8+\shift,-4) node {$\overline{\modm}_{g}$};
\draw [->, line width=1pt] (8+\shift,-1)--(8+\shift,-3);
\draw (7.5+\shift,-2) node {$\pi$};
\end{tikzpicture}
\end{center}

\noindent $\mathbf{g=3}$. In genus 3, $H^{2*}(\overline{\modm}_3,\bq)=RH^*(\overline{\modm}_{3})$ due to the calculation of the cohomology $H^*(\overline{\modm}_3,\bq)$, for example by using the calculation of  $H^*(\overline{\modm}_{3,1},\bq)$ in \cite{GloHod}, together with the calculation of the tautological ring $RH^*(\overline{\modm}_{3},\bq)$ via Pixton's relations \cite{PPZRel} implemented using the Sage package admcycles \cite{DSZadm}.  We have $\dim RH^4(\overline{\modm}_{3},\bq)=7$ and we write $\Theta_{3}$ as a general linear combination of basis vectors in $RH^4(\overline{\modm}_{3})$:
\[\Theta_{3}=a_{1111}\kappa_1^4+a_{112}\kappa_1^2\kappa_2+a_{13}\kappa_1\kappa_3+a_{22}\kappa_2^2+a_4\kappa_4+b_1B_1+b_2B_2
\]
where $B_i\in RH^4(\overline{\modm}_3)$ are given by $B_1=$
\begin{tikzpicture}[thick,scale=.6, every node/.style={transform shape}]
\draw (0,0) node [shape=circle,draw]        {1};
\draw  (.3,0) -- (1.7,0);
\draw (2,0) node [shape=circle,draw]        {2};
\draw (0,.5) node     {$\kappa_1$};
\draw (2,.5) node     {$\kappa_2$};
\end{tikzpicture}
and $B_2=$
\begin{tikzpicture}[thick,scale=.6, every node/.style={transform shape}]
\draw (0,0) node [shape=circle,draw]        {1};
\draw  (.3,0) -- (1.7,0);
\draw (2,0) node [shape=circle,draw]        {2};
\draw (2,.5) node     {$\kappa_3$};
\end{tikzpicture}.
The following pull-back map is injective
\[
RH^{4}(\overline{\modm}_{3})\to RH^{4}(\overline {\modm}_{2,2})\oplus RH^3(\overline {\modm}_{2,1})\otimes RH^1(\overline{\modm}_{1,1})
\]
(which implies that the map from $RH^{4}(\overline{\modm}_{3})$ to the boundary is injective).  The restriction map 
\[RH^{4}(\overline{\modm}_{3})\to RH^{4}(\overline {\modm}_{2,2})
\] 
has 2-dimensional kernel and is surjective onto the $S_2$-invariant part of $RH^{4}(\overline {\modm}_{2,2})$.  Hence the condition 
\[\phi_{\text{irr}}^*\Theta_3=\Theta_{2,2}=\psi_1\psi_2\pi^*\Theta_{2}\] 
determines $\Theta_3$ up to parameters $s,t\in\bq$:
\begin{align*}
a_{1111}&=s\\
a_{112}& = \tfrac{11}{10}\lambda+\tfrac{17}{15}\lambda^2-18s+\tfrac{4}{3}t\\ 
a_{13} &= -12\lambda-12\lambda^2+104s-13t\\
a_{22} &= -\tfrac{33}{10}\lambda-\tfrac{29}{10}\lambda^2+27s-5t\\ 
a_4 &= \tfrac{376}{5}\lambda+\tfrac{1933}{30}\lambda^2-426s+\tfrac{250}{3}t\\ 
b_1&=t\\
b_2 & = \tfrac{2}{5}\lambda(3-\lambda)
\end{align*}
The pull-back map
\[
RH^{4}(\overline{\modm}_{3})\to  RH^3(\overline {\modm}_{2,1})\otimes RH^1(\overline{\modm}_{1,1})
\]
has three dimensional image, and the condition 
\[\phi_\Gamma^*\Theta_3=\Theta_{2,1}\otimes\Theta_{1,1}=(\psi_1\pi^*\Theta_{2})\otimes(\lambda\psi_1)\] 
cannot be satisfied with only two parameters in $\Theta_3$ for general $\lambda$, forcing $\lambda$ to  satisfy a polynomial relation.  We find
\begin{align*}
a_{1111}&=\tfrac{5}{24}\lambda^3-\tfrac{19}{120}\lambda^2-\tfrac{11}{40}\lambda\\
a_{112}& = \tfrac{5}{4}\lambda^3-\tfrac{147}{20}\lambda^2-\tfrac{99}{20}\lambda\\ 
a_{13} &= \tfrac{403}{24}\lambda^3-\tfrac{209}{12}\lambda^2-\tfrac{239}{8}\lambda=a_{13}-\tfrac{3108}{53}b_1\\
a_{22} &= -\tfrac{3867}{212}\lambda^3+\tfrac{99471}{2120}\lambda^2+\tfrac{22143}{530}\lambda=a_{22}+12b_1\\
a_4 &=-\tfrac{115}{2}\lambda^3+\tfrac{1221}{20}\lambda^2+\tfrac{618}{5}\lambda\\ 
b_1&=\tfrac{1}{40}\lambda(\lambda-3)(15\lambda+11)\\
b_2 & = \tfrac{2}{5}\lambda(3-\lambda)
\end{align*}
which is consistent only when $b_1=0$ hence
\[\lambda(\lambda-3)(15\lambda+11)=0.\]

\noindent $\mathbf{g=4}$. In genus 4, $H^{2*}(\overline{\modm}_4)=RH^*(\overline{\modm}_{4})$ is due to the calculation by Bergstr{\"o}m and Tommasi \cite{BToRat} of the Hodge polynomial of $\overline{\modm}_4$ together with the calculation of the tautological ring $RH^*(\overline{\modm}_{4})$ via Pixton's relations using admcycles \cite{DSZadm}.  We choose a general element $\theta_4\in RH^{6}(\overline{\modm}_{4})$ which is a linear combination of basis vectors for the 32 dimensional space $RH^{6}(\overline{\modm}_{4})$.  The pull-back map of $RH^{6}(\overline{\modm}_{4})$ to the boundary can be shown to be injective using admcycles.

The main purpose of the $g=4$ calculation is to prove that $\lambda=-\frac{11}{15}$ is impossible, so we substitute  $\lambda=-\frac{11}{15}$ into $\Theta_3$ above to get
\[\Theta_{3}=\tfrac{2783}{81000}\kappa_1^4-\tfrac{11011}{13500}\kappa_1^2\kappa_2+\tfrac{59939}{10125}\kappa_1\kappa_3+\tfrac{16093}{9000}\kappa_2^2-\tfrac{474287}{13500}\kappa_4-\tfrac{1232}{1125}B_2
\]
As in the $g=3$ case above we consider the pull-back map
\[
RH^{6}(\overline{\modm}_{4})\to RH^{6}(\overline {\modm}_{3,2})\]
which has a six dimensional kernel.  The $S_2$-invariant part of $H^{12}(\overline {\modm}_{3,2})$ is proven in \cite{BerCoh} to be 31 dimensional, and using admcycles it can be shown to be tautological. The condition $\phi_{\text irr}^*\Theta_4=\Theta_{3,2}=\psi_1\psi_2\pi^*\Theta_{3}$ produces a system of 31 equations in 32 unknowns.  Using admcycles, we find that $\Theta_{3,2}$ lies in the image of the pull-back map, and constrains $\Theta_4$ to depend linearly on 6 parameters.  The pull-back map composed with projection
\[RH^{6}(\overline{\modm}_{4})\to RH^{5}(\overline {\modm}_{3,1})\otimes RH^{1}(\overline {\modm}_{1,1})\]
uniquely determines the 6 parameters and finally the resulting class $\Theta_4$ is shown under
pull-back map composed with projection
\[RH^{6}(\overline{\modm}_{4})\to RH^{3}(\overline {\modm}_{2,1})\otimes RH^{3}(\overline {\modm}_{2,1})\]
to disagree with $\Theta_{2,1}\otimes\Theta_{2,1}$.  We conclude that $\lambda=-11/15$ is impossible, leaving $\lambda=3$.

\end{proof}

\section{Cohomological field theories}  \label{sec:cohft}

The class $\Theta_{g,n}$ combines with known enumerative invariants, such as Gromov-Witten invariants, to give rise to new invariants.  More generally, $\Theta_{g,n}$ pairs with any cohomological field theory, which is fundamentally related to the moduli space of curves $\overline{\modm}_{g,n}$, retaining many of the properties of the cohomological field theory, and is in particular often calculable.

A {\em cohomological field theory} is a pair $(H,\eta)$ composed of a finite-dimensional complex vector space $H$ equipped with a symmetric, bilinear, nondegenerate form, or metric, $\eta$ and a sequence of $S_n$-equivariant maps.  Many CohFTs are naturally defined on $H$ defined over $\bq$, nevertheless we use $\bc$ in order to relate them to Frobenius manifolds, and to use normalised canonical coordinates, defined later.
\[ \Omega_{g,n}:H^{\otimes n}\to H^*(\overline{\modm}_{g,n},\bc)\]
that satisfy compatibility conditions from inclusion of strata: 
$$\phi_{\text{irr}}:\overline{\modm}_{g-1,n+2}\to\overline{\modm}_{g,n},\quad \phi_{h,I}:\overline{\modm}_{h,|I|+1}\times\overline{\modm}_{g-h,|J|+1}\to\overline{\modm}_{g,n},\  I\sqcup J=\{1,...,n\}$$
given by
\begin{align}
\phi_{\text{irr}}^*\Omega_{g,n}(v_1\otimes...\otimes v_n)&=\Omega_{g-1,n+2}(v_1\otimes...\otimes v_n\otimes\Delta) 
  \label{glue1}
\\
\phi_{h,I}^*\Omega_{g,n}(v_1\otimes...\otimes v_n)&=\Omega_{h,|I|+1}\otimes \Omega_{g-h,|J|+1}\big(\bigotimes_{i\in I}v_i\otimes\Delta\otimes\bigotimes_{j\in J}v_j\big) \label{glue2}
\end{align}
where $\Delta\in H\otimes H$ is dual to $\eta\in H^*\otimes H^*$.  When $n=0$, $\Omega_g:=\Omega_{g,0}\in H^*(\overline{\modm}_{g},\bc)$.  The CohFT has {\em flat identity} if there exists a vector $\un\in H$ compatible with the forgetful map $\pi:\overline{\modm}_{g,n+1}\to\overline{\modm}_{g,n}$ by
\begin{equation}   \label{cohforget}
\Omega_{g,n+1}(\un\otimes v_1\otimes...\otimes v_n)=\pi^*\Omega_{g,n}(v_1\otimes...\otimes v_n)
\end{equation}
for $2g-2+n>0$,  and
$$\Omega_{0,3}(\un\otimes v_1\otimes v_2)=\eta(v_1,v_2).
$$

For a one-dimensional CohFT, i.e. $\dim H=1$, identify $\Omega_{g,n}$ with the image $\Omega_{g,n}(\un^{\otimes n})$, so we write $\Omega_{g,n}\in H^*(\overline{\modm}_{g,n},\bc)$.  A trivial example of a CohFT is $\Omega_{g,n}=1\in H^0(\overline{\modm}_{g,n},\bc)$ which is a topological field theory as we now describe.

A two-dimensional topological field theory (TFT) is a vector space $H$ and a sequence of symmetric linear maps
\[ \Omega^0_{g,n}:H^{\otimes n}\to \bc\]
for integers $g\geq 0$ and $n>0$ satisfying the following conditions.  The map $\Omega^0_{0,2}=\eta$ defines a symmetric, bilinear, nondegenerate form $\eta$, and together with $\Omega^0_{0,3}$ it defines a
product $\Cdot$ on $H$ via
\begin{equation}   \label{prod}
\eta(v_1\Cdot v_2,v_3)=\Omega^0_{0,3}(v_1,v_2,v_3)
\end{equation}
with identity $\un$ given by the dual of $\Omega^0_{0,1}=\un^*=\eta(\un,\cdot)$.  It satisfies 
$$\Omega^0_{g,n+1}(\un\otimes v_1\otimes...\otimes v_n)=\Omega^0_{g,n}(v_1\otimes...\otimes v_n)$$ 
and the gluing conditions 
\begin{align*}
\Omega^0_{g,n}(v_1\otimes...\otimes v_n)&=\Omega^0_{g-1,n+2}(v_1\otimes...\otimes v_n\otimes\Delta)\\
\Omega^0_{g,n}(v_1\otimes...\otimes v_n)&=\Omega^0_{g_1,|I|+1}\otimes \Omega^0_{g_2,|J|+1}\big(\bigotimes_{i\in I}v_i\otimes\Delta\otimes\bigotimes_{j\in J}v_j\big)
\end{align*}
for $g=g_1+g_2$ and $I\sqcup J=\{1,...,n\}$.

Consider the natural isomorphism $H^0({\overline{\modm}_{g,n})}\cong\bc$.  The degree zero part of a CohFT $\Omega_{g,n}$ is a TFT: 
$$\Omega^0_{g,n}:H^{\otimes n}\stackrel{\Omega_{g,n}}{\to} H^*({\overline{\modm}_{g,n}},\bc)\to H^0({\overline{\modm}_{g,n}},\bc).
$$
We often write $\Omega_{0,3}=\Omega^0_{0,3}$ interchangeably.  Associated to $\Omega_{g,n}$ is the product \eqref{prod} built from $\eta$ and $\Omega_{0,3}$.

\begin{remark}  \label{theco}
The classes $\Theta_{g,n}$ satisfy properties~\eqref{glue1} and \eqref{glue2} of a one-dimensional CohFT.  In place of property~\eqref{cohforget}, they satisfy $\Theta_{g,n+1}(\un\otimes v_1\otimes...\otimes v_n)=\psi_{n+1}\cdot\pi^*\Theta_{g,n}(v_1\otimes...\otimes v_n)$ and $\Theta_{0,3}=0$.
\end{remark}

 The product defined in \eqref{prod} is {\em semisimple} if it is diagonal $H\cong\bc\oplus\bc\oplus...\oplus\bc$, i.e. there is a canonical basis $\{ u_1,...,u_N\}\subset H$ such that $u_i\cdot u_j=\delta_{ij}u_i$.  The metric is then necessarily diagonal with respect to the same basis, $\eta(u_i,u_j)=\delta_{ij}\eta_i$ for some $\eta_i\in\bc \setminus \{0\}$, $i=1,...,N$.  The Givental-Teleman theorem described in Section~\ref{givental} gives a construction of semisimple CohFTs.

\subsection{Cohomological field theories coupled to \texorpdfstring{$\Theta_{g,n}$}{THgn}}
\begin{definition}  \label{cohftheta}
For any CohFT $\Omega$ defined on $(H,\eta)$ define $\Omega^\Theta=\{\Omega^\Theta_{g,n}\}$ to be the sequence of $S_n$-equivariant maps $\Omega^\Theta_{g,n}:H^{\otimes n}\to H^*(\overline{\modm}_{g,n},\bc)$ given by 
\[\Omega^\Theta_{g,n}(v_1\otimes...\otimes v_n):=\Theta_{g,n}\cdot\Omega_{g,n}(v_1\otimes...\otimes v_n).\]   
\end{definition}
This is essentially the tensor products of CohFTs, albeit involving $\Theta_{g,n}$.  The tensor products of CohFTs is obtained as above by cup product on $H^*(\overline{\modm}_{g,n},\bc)$, generalising Gromov-Witten invariants of target products and the K\"unneth formula $H^*(X_1\times X_2)\cong H^*X_1\otimes H^*X_2$.  

Generalising Remark~\ref{theco}, $\Omega^\Theta_{g,n}$ satisfies properties~\eqref{glue1} and \eqref{glue2} of a CohFT on $(H,\eta)$.  In place of property~\eqref{cohforget}, it satisfies 
$$\Omega^\Theta_{g,n+1}(\un\otimes v_1\otimes...\otimes v_n)=\psi_{n+1}\cdot\pi^*\Omega^\Theta_{g,n}(v_1\otimes...\otimes v_n)$$ and $\Omega^\Theta_{0,3}=0$.

Given a CohFT $\Omega=\{\Omega_{g,n}\}$, or a more general collection of classes such as $\Omega=\{\Omega^\Theta_{g,n}\}$,  and a basis $\{e_1,...,e_N\}$ of $H$, the partition function of $\Omega$ is defined by:
\begin{equation}   \label{partfun}
Z_{\Omega}(\hbar,\{t^{\alpha}_k\})=\exp\sum_{g,n,\vec{k}}\frac{\hbar^{g-1}}{n!}\int_{\overline{\modm}_{g,n}}\Omega_{g,n}(e_{\alpha_1}\otimes...\otimes e_{\alpha_n})\cdot\prod_{j=1}^n\psi_j^{k_j}\prod t^{\alpha_j}_{k_j}
\end{equation}
for $\alpha_i\in\{1,...,N\}$ and $k_j\in\bn$.  For $\dim H=1$ and $\Omega_{g,n}=1\in H^*(\overline{\modm}_{g,n},\bc)$, its partition function is $Z_{\Omega}(\hbar,\{t_k\})=Z^{\text{KW}}(\hbar,\{t_k\})$ which is defined in Section~\ref{sec:kdv}.  

For $\Omega_{g,n}=\Theta_{g,n}\in H^*(\overline{\modm}_{g,n},\bc)$, $Z_{\Omega}(\hbar,\{t_k\})=Z^{\Theta}(\hbar,\{t_k\})$ gives its partition function.  Property~\eqref{forget} is realised by the following homogeneity property:
\begin{equation}  \label{dilaton}
\frac{\partial}{\partial t_0}Z^{\Theta}(\hbar,t_0,t_1,...)=\sum_{i=0}^{\infty}(2i+1)t_i \frac{\partial}{\partial t_i}Z^{\Theta}(\hbar,t_0,t_1,...)+\frac18Z^{\Theta}(\hbar,t_0,t_1,...)
\end{equation}
 and for $Z^{\Theta}(\hbar,t_0,t_1,...)$ in the following proposition.
\begin{proposition}  \label{th:dilaton}
The function $Z^{\Theta}(\hbar,t_0,t_1,...)$ is homogeneous of degree $-\frac18$ with respect to $\{q=1-t_0,t_1,t_2,...\}$ with $\deg q=1$ and $\deg t_i=2i+1$ for $i>0$.  Equivalently it satisfies the dilaton equation \eqref{dilaton}.
\end{proposition}
\begin{proof}
We have
\begin{align*}  
\int_{\overline{\modm}_{g,n+1}}\Theta_{g,n+1}\cdot\prod_{j=1}^n\psi_j^{k_j}&=\int_{\overline{\modm}_{g,n+1}}\pi^*\Theta_{g,n}\cdot\psi_{n+1}\cdot\prod_{j=1}^n\psi_j^{k_j}\\
=\int_{\overline{\modm}_{g,n+1}}&\pi^*\Theta_{g,n}\cdot\psi_{n+1}\cdot\prod_{j=1}^n\pi^*\psi_j^{k_j}=\int_{\overline{\modm}_{g,n}}\Theta_{g,n}\cdot\prod_{j=1}^n\psi_j^{k_j}\cdot\pi_*\psi_{n+1}\\&=
(2g-2+n)\int_{\overline{\modm}_{g,n}}\Theta_{g,n}\cdot\prod_{j=1}^n\psi_j^{k_j}.\nonumber
\end{align*}
which uses $\psi_{n+1}\cdot\psi_j=\psi_{n+1}\cdot\pi^*\psi_j$ for $j=1,...,n$ and $\pi_*(\pi^*\omega\cdot\psi_{n+1})=\omega\cdot\pi_*\psi_{n+1}$.  In terms of the partition function $Z^{\Theta}(\hbar,t_0,t_1,...)$, this is realised by the equation \eqref{dilaton}.
\end{proof}

\subsubsection{Gromov-Witten invariants}
Let $X$ be a projective algebraic variety and consider $(C,x_1,\dots,x_n)$ a connected smooth curve of genus $g$ with $n$ distinct marked points.  For $\beta \in H_2(X,\bz)$ the moduli space of stable maps $\overline{\modm}_{g,n}(X,\beta)$ is defined by:
$$\overline{\modm}_{g,n}(X,\beta)=\{(C,x_1,\dots,x_n)\stackrel{\pi}{\rightarrow} X\mid \pi_\ast [C]=\beta\}/\sim$$
where $\pi$ is a morphism from a connected nodal curve $C$ containing distinct points $\{x_1,\dots,x_n\}$ that avoid the nodes.  Any genus zero irreducible component of $C$ with fewer than three distinguished points (nodal or marked), or genus one irreducible component of $C$ with no distinguished point, must not be collapsed to a point.  We quotient by isomorphisms of the domain $C$ that fix each $x_i$.    The moduli space of stable maps has irreducible components of different dimensions but it has a virtual class of dimension
\begin{equation}  \label{dimGW} 
\dim[\overline{\modm}_{g,n}(X,\beta)]^{\text{virt}}=(\dim X-3)(1-g)+\langle c_1(X),\beta\rangle +n.
\end{equation}
For $i=1,\dots,n$ there exist evaluation maps:
\begin{equation}
ev_i:\overline{\modm}_{g,n}(X,\beta)\longrightarrow X, \quad ev_i(\pi)=\pi(x_i)
\end{equation}
and classes $\gamma\in H^*(X,\bz)$ pull back to classes in $H^*(\overline{\modm}_{g,n}(X,\beta),\bc)$
\begin{equation}
ev_i^\ast:H^*(X,\bz)\longrightarrow H^*(\overline{\modm}_{g,n}(X,\beta),\bc).
\end{equation}
The forgetful map $p:\overline{\modm}_{g,n}(X,\beta)\to \overline{\modm}_{g,n}$ maps a stable map to its domain curve followed by contraction of unstable components.  The push-forward map $p_*$ on cohomology defines a CohFT $\Omega_X$ on the even part of the cohomology $H=H^{\text{even}}(X,\bc)$ (and a generalisation of a CohFT on $H^*(X,\bc)$) equipped with the symmetric, bilinear, nondegenerate form 
$$\eta(\alpha,\beta)=\int_X\alpha\wedge\beta.$$
We have $(\Omega_X)_{g,n}:H^{\text{even}}(X,\bc)^{\otimes n}\to H^*(\overline{\modm}_{g,n},\bc)$ defined by 
$$(\Omega_X)_{g,n}(\alpha_1,...\alpha_n)=\sum_{\beta}p_*\left(\prod_{i=1}^nev_i^\ast(\alpha_i)\cap[\overline{\modm}_{g,n}(X,\beta)]^{\text{virt}}\right)\in H^*(\overline{\modm}_{g,n},\bc).$$
Note that it is the dependence of $p=p(g,n,\beta)$ on $\beta$ (which is suppressed) that allows $(\Omega_X)_{g,n}(\alpha_1,...\alpha_n)$ to be composed of different degree terms.  The partition function of the CohFT $\Omega_X$ with respect to a chosen basis $e_{\alpha}$ of $H^{\text{even}}(X;\bc)$ is
$$Z_{\Omega_X}(\hbar,\{t^{\alpha}_k\})=\exp\hspace{-3mm}\sum_{\scalemath{0.8}{\begin{array}{c}g,n,\vec{k}\\ \vec{\alpha},\beta\end{array}}}\hspace{-3mm}\frac{\hbar^{g-1}}{n!}\int_{\overline{\modm}_{g,n}}\hspace{-3mm}p_*\left(\prod_{i=1}^nev_i^\ast(e_{\alpha_i})\cap[\overline{\modm}_{g,n}(X,\beta)]^{\text{virt}}\right)\prod_{j=1}^n\psi_j^{k_j}\prod t^{\alpha_j}_{k_j}.
$$
It stores {\em ancestor} invariants.  These are different to {\em descendant} invariants which use in place of $\psi_j=c_1(L_j)$, $\Psi_j=c_1(\mathcal{L}_j)$ for line bundles $\mathcal{L}_j\to \overline{\modm}_{g,n}(X,\beta)$ defined as the cotangent bundle over the $i$th marked point.

Following Definition~\ref{cohftheta}, we define $\Omega_X^{\Theta}$ by
$$(\Omega_X^{\Theta})_{g,n}(\alpha_1,...\alpha_n)=\Theta_{g,n}\cdot\sum_{\beta}p_*\left(\prod_{i=1}^nev_i^\ast(\alpha_i)\right)\in H^*(\overline{\modm}_{g,n},\bc).$$
and
$$Z^{\Theta}_{\Omega_X}(\hbar,\{t^{\alpha}_k\})=\exp\hspace{-3mm}\sum_{\scalemath{0.8}{\begin{array}{c}g,n,\vec{k}\\ \vec{\alpha},\beta\end{array}}}\hspace{-2mm}\frac{\hbar^{g-1}}{n!}\int_{\overline{\modm}_{g,n}}\Theta_{g,n}\cdot p_*\left(\prod_{i=1}^nev_i^\ast(e_{\alpha_i})\right)\cdot\prod_{j=1}^n\psi_j^{k_j}\prod t^{\alpha_j}_{k_j}.
$$
Let $\Theta_{g,n}^{\text{PD}}\subset A_{g-1}(\overline{\modm}_{g,n},\bc)$ be the $(g-1)$-dimensional Chow class given by the push-forward of the top Chern class of the bundle $E_{g,n}$ defined in Definition~\ref{obsbun}.  The virtual dimension of the pull-back of $\Theta_{g,n}^{\text{PD}}$ is:
\begin{equation}  \label{dimGWtheta}
\dim \left\{[\overline{\modm}_{g,n}(X,{\bf d})]^{\text{vir}}\cap p^{-1}(\Theta_{g,n}^{\text{PD}})\right\}=(\dim X-1)(1-g)+\langle c_1(X),\beta\rangle.
\end{equation}  
Comparing the dimension formulae \eqref{dimGW} and \eqref{dimGWtheta} we see that elliptic curves now take the place of Calabi-Yau 3-folds to give virtual dimension zero moduli spaces, independent of genus and degree.  The invariants of a target curve $X$ are trivial when the genus of $X$ is greater than 1 and computable when $X=\bp^1$, \cite{NorGro}, producing results analogous to the usual Gromov-Witten invariants in \cite{NScGro}.  For $c_1(X)=0$ and $\dim X>1$, the invariants vanish for $g>1$, while for $g=1$ it seems to predict an invariant associated to maps of elliptic curves to $X$.

\subsubsection{Weil-Petersson volumes} 
A fundamental example of a 1-dimensional CohFT is given by 
$$\Omega_{g,n}=\exp(2\pi^2\kappa_1)\in H^*(\overline{\modm}_{g,n},\br).$$  
Its partition function stores Weil-Petersson volumes 
$$V_{g,n}=\frac{(2\pi^2)^{3g-3+n}}{(3g-3+n)!}\int_{\overline{\modm}_{g,n}}\kappa_1^{3g-3+n}$$
and deformed Weil-Petersson volumes studied by Mirzakhani \cite{MirSim}.
Weil-Petersson volumes of the subvariety of $\overline{\modm}_{g,n}$ dual to $\Theta_{g,n}$ make sense even before we find such a subvariety.  They are given by 
$$V^{\Theta}_{g,n}=\frac{(2\pi^2)^{g-1}}{(g-1)!}\int_{\overline{\modm}_{g,n}}\Theta_{g,n}\cdot\kappa_1^{g-1}$$ 
which are calculable since they are given by a translation of $Z^{\text{BGW}}$.  If we include $\psi$ classes, we get polynomials $V^{\Theta}_{g,n}(L_1,...,L_n)$ which give the deformed volumes analogous to Mirzakhani's volumes.  In \cite{NorEnu,SWiJTG} the polynomials $V^{\Theta}_{g,n}(L_1,...,L_n)$ are related to the volume of the moduli space of Super Riemann surfaces.

\subsubsection{ELSV formula}
Another example of a 1-dimensional CohFT is given by 
$$\Omega_{g,n}=c(E^\vee)=1-\lambda_1+...+(-1)^{g}\lambda_{g}\in H^*(\overline{\modm}_{g,n},\bc)$$ 
where $\lambda_i=c_i(E)$ is the $i$th Chern class of the Hodge bundle $E\to\overline{\modm}_{g,n}$ defined to have fibres $H^0(\omega_C)$ over a nodal curve $C$.

Hurwitz \cite{HurRie} studied the problem of connected curves $\Sigma$ of genus $g$ covering $\bp^1$, branched over $r+1$ fixed points $\{p_1,p_2,...,p_r,p_{r+1}\}$ with arbitrary profile $\mu=(\mu_1,...,\mu_n)$ over $p_{r+1}$.  Over the other $r$ branch points one specifies simple ramification, i.e. the partition $(2,1,1,....)$.   The Riemann-Hurwitz formula determines the number $r$ of simple branch points via $2-2g-n=|\mu|-r$.  
\begin{definition} 
Define the simple Hurwitz number $H_{g,\mu}$ to be the weighted count of genus $g$ connected covers of $\bp^1$ with ramification $\mu=(\mu_1,...,\mu_n)$ over $\infty$ and simple ramification elsewhere.
Each cover $\pi$ is counted with weight $1/|{\rm Aut}(\pi)|$.
\end{definition}
Coefficients of the partition function of the CohFT $\Omega_{g,n}=c(E^\vee)$ appear naturally in the ELSV formula \cite{ELSVHur} which relates the Hurwitz numbers $H_{g,\mu}$ to the Hodge classes.  The ELSV formula is:
\[ H_{g,\mu}=\frac{r(g,\mu)!}{|{\rm Aut\ }\mu|}\prod_{i=1}^n\frac{\mu_i^{\mu_i}}{\mu_i!}\int_{\overline{\modm}_{g,n}}\frac{1-\lambda_1+...+(-1)^g\lambda_g}{(1-\mu_1\psi_1)...(1-\mu_n\psi_n)}\]
where $\mu=(\mu_1,...,\mu_n)$ and $r(g,\mu)=2g-2+n+|\mu|$.

Using $\Omega^{\Theta}_{g,n}=\Theta\cdot c(E^\vee)$ we can define an analogue of the ELSV formula:
$$H^{\Theta}_{g,\mu}=\frac{(2g-2+n+|\mu|)!}{|{\rm Aut\ }\mu|}\prod_{i=1}^n\frac{\mu_i^{\mu_i}}{\mu_i!}\int_{\overline{\modm}_{g,n}}\Theta_{g,n}\cdot\frac{1-\lambda_1+...+(-1)^{g-1}\lambda_{g-1}}{(1-\mu_1\psi_1)...(1-\mu_n\psi_n)}.
$$
It may be that $H^{\Theta}_{g,\mu}$ has an interpretation of enumerating simple Hurwitz covers.   
Note that it makes sense to set all $\mu_i=0$, and in particular there are non-trivial primary invariants over $\overline{\modm}_g$, unlike for simple Hurwitz numbers.  An example calculation:
$$\int_{\overline{\modm}_2}\Theta_2\lambda_1=\frac{1}{5}\cdot\frac{1}{8}\cdot\frac{1}{8}\cdot\frac{1}{2}+\frac{1}{10}\cdot\frac{1}{8}\cdot\frac{1}{2}=\frac{1}{128}\quad\quad\Leftarrow\quad
\lambda_1=\frac{1}{10}(2\delta_{1,1}+\delta_{\text{irr}}).
$$

\subsubsection{The versal deformation space of the $A_2$ singularity.}   \label{A2sing}
The $A_2$ singularity has a two-dimensional versal deformation space $M\cong\bc^2=\{(t_1,t_2)\}$ parametrising the family
\[ W_t(z)=z^3-t_2z+t_1
\]
that admits a semisimple Frobenius manifold structure.  Dubrovin \cite{DubGeo} associated a family of linear systems \eqref{compat} depending on the canonical coordinates $(u_1,...,u_N)$ of any semisimple Frobenius manifold $M$.  This produces a CohFT $\Omega^{A_2}$ defined on $\bc^2$ from the $A_2$ singularity using Definition~\ref{givact} in Section~\ref{givental}.  More generally, to any point of a Frobenius manifold one can associate a cohomological field theory and conversely the genus zero part of a cohomological field theory defines a Frobenius manifold \cite{DubGeo}.   

Recall that a Frobenius manifold is a complex manifold $M$ equipped with an associative product on its tangent bundle compatible with a flat metric---a nondegenerate symmetric bilinear form---on the manifold.  It is encoded by a single function $F(t_1,...,t_{N})$, known as the {\em prepotential}, that satisfies a nonlinear partial differential equation known as the Witten-Dijkgraaf-Verlinde-Verlinde (WDVV) equation:
$$F_{ijm}\eta^{mn}F_{k\ell n}=F_{i\ell m}\eta^{mn}F_{jkn},\quad \eta_{ij}=F_{1ij}
$$
where $\eta^{ik}\eta_{kj}=\delta_{ij}$, $F_i=\frac{\partial}{\partial t_i}F$, $\frac{\partial}{\partial t_1}=\un$ corresponds to the flat unit vector field for the product, and $\{t_1,...,t_{N}\}$ are (flat) local coordinates of $M$. The Frobenius manifold is {\em conformal} if it comes equipped with an Euler vector field $E$ which describes symmetries of the Frobenius manifold, neatly encoded by 
\[E\cdot F(t_1,...,t_{N})=c\cdot F(t_1,...,t_{N})+\text{quadratic polynomial},\quad c\in\bc.\]  
For a semisimple conformal Frobenius manifold, multiplication by the Euler vector field $E$ produces an endomorphism $U$ with eigenvalues $\{u_1,...,u_N\}$ known as {\em canonical coordinates} on $M$.  They give rise to vector fields  $\partial/\partial u_i$ with respect to which the metric $\eta$, product $\Cdot$ and Euler vector field $E$ are diagonal:
$$\frac{\partial}{\partial u_i}\Cdot\frac{\partial}{\partial u_j}=\delta_{ij}\frac{\partial}{\partial u_i},\quad \eta\left(\frac{\partial}{\partial u_i},\frac{\partial}{\partial u_j}\right)=\delta_{ij}\Delta_i,\quad E=\sum u_i\frac{\partial}{\partial u_i}.$$

The differential equation \eqref{compat} in $z$ is defined at any point of the Frobenius manifold using $U$, the endomorphism defined by multiplication by the Euler vector field $E$, and the endomorphism $V=[\Gamma,U]$ where $\Gamma_{ij}=\frac{\partial_{u_i}\Delta_j}{2\sqrt{\Delta_i\Delta_j}}$ for $i\neq j$ are the so-called rotation coefficients of the metric $\eta$ in the normalised canonical basis.  Hence associated to each point of the Frobenius manifold is an element $R(z)=\sum R_kz^k$ of the twisted loop group.  From $(H=T_pM,\Cdot,E,V)$ Teleman \cite{TelStr} defined the endomorphisms $R_k$ of $H$ recursively from $R_1$ by $R_0=I$ and 
\begin{equation}   \label{teleman}
[R_{k+1},U]=(k+V)R_k,\quad k=0,1,...
\end{equation}
which is a direct consequence of substitution of $Y=R(z^{-1})e^{zU}$ into \eqref{compat} with $z\mapsto z^{-1}$
$$
0=\left(\frac{d}{dz}+\frac{U}{z^2}+\frac{V}{z}\right)R(z)e^{U/z}
=\left(\frac{d}{dz}R(z)+\frac{1}{z^2}[U,R(z)]+\frac{1}{z}VR(z)\right)e^{U/z}.
$$
It is useful to consider three natural bases of the tangent space $H=T_pM\cong\bc^N$ at any point $p$ of a semisimple Frobenius manifold.  The flat basis $\{\partial/\partial t_i\}$ which gives a constant metric $\eta$, the canonical basis $\{\partial/\partial u_i\}$ which gives a trivial product $\Cdot$, and the normalised canonical basis $\{v_i\}$, for $v_i=\Delta_i^{-1/2}\partial/\partial u_i$, which gives a trivial metric $\eta$.    (A different choice of square root of $\Delta_i$ would simply give a different choice of normalised canonical basis.)  The transition matrix $\Psi$ from flat coordinates to normalised canonical coordinates sends the metric $\eta$ to the dot product, i.e. $\Psi^T\Psi=\eta$.  The topological field theory structure on $H$ induced from $\eta$ and $\Cdot$ is diagonal in the normalised canonical basis.  It is given by
$$
\Omega_{g,n}(v_i^{\otimes n})=\Delta_i^{1-g-\frac12 n}
$$
and vanishes on mixed products of $v_i$ and $v_j$, $i\neq j$.  In the normalised canonical basis, the unit vector is given by 
$$\un=(\Delta_1^\frac12,...,\Delta_N^\frac12)$$
hence it uniquely determines the topological field theory.
We find the normalised canonical basis most useful for comparisons with topological recursion---see Section~\ref{sec:TR}.

The Frobenius manifold structure on the versal deformation space $M$  of the $A_2$ singularity was constructed in \cite{DubGeo,SaiPer}.  The product on tangent spaces of the family $W_t(z)=z^3-t_2z+t_1$ is induced from the isomorphism 
\begin{align*}
T_tM&\cong\bc[z]/W'_t(z)
\end{align*}
given by $\tfrac{\partial}{\partial t_k}\mapsto \tfrac{\partial}{\partial t_k}W_t=(-z)^{k-1}$ producing
$$\frac{\partial}{\partial t_1}\Cdot\frac{\partial}{\partial t_1}=\frac{\partial}{\partial t_1},\quad \frac{\partial}{\partial t_1}\Cdot\frac{\partial}{\partial t_2}=\frac{\partial}{\partial t_2},\quad \frac{\partial}{\partial t_2}\Cdot\frac{\partial}{\partial t_2}=\frac13 t_2\frac{\partial}{\partial t_1}.
$$
The metric is given by
\[\eta(p(z),q(z))=-3\Res_{\infty}\frac{p(z)q(z)dz}{W'_t(z)}.\]
With respect to the basis $\{\frac{\partial}{\partial t_1},\frac{\partial}{\partial t_2}\}$  it is constant hence flat:
$$\eta=\left(\begin{array}{cc}0&1\\1&0\end{array}\right).
$$
The Frobenius manifold structure on $M$ is conformal.
The unit and Euler vector fields are $\un=\frac{\partial}{\partial t_1}$ and $E=t_1\frac{\partial}{\partial t_1}+\frac23t_2\frac{\partial}{\partial t_2}$, which correspond respectively to the images of 1 and $W_t(z)$ in $\bc[z]/W'_t(z)$.

The prepotential is produced via $\eta_{ij}=F_{1ij}$ and $\eta(\partial/\partial t_i\Cdot\partial/\partial t_j,\partial/\partial t_k)=F_{ijk}$
$$F(t_1,t_2)=\frac12 t_1^2t_2+\frac{1}{72}t_2^4.
$$
and satisfies $E\cdot F(t_1,t_2)=\frac83 F(t_1,t_2)$.
The canonical coordinates are  
$$ u_1=t_1+\frac{2}{3\sqrt{3}}t_2^{3/2},\quad  u_2=t_1-\frac{2}{3\sqrt{3}}t_2^{3/2}.
$$
In the normalised canonical basis, the rotation coefficients $\Gamma_{12}=\dfrac{-i\sqrt{3}}{8}t_2^{-3/2}=\Gamma_{21}$ give rise to
$V=[\Gamma,U]=\frac{i\sqrt{3}}{2}t_2^{-3/2}\left(\begin{array}{cc}0&-1\\1&0\end{array}\right)
$.  In canonical coordinates we have
\begin{equation}   \label{A2V}  U=\left(\begin{array}{cc}u_1&0\\0&u_2\end{array}\right),\quad V=\frac{2i}{3(u_1-u_2)}\left(\begin{array}{cc}0&1\\-1&0\end{array}\right).
\end{equation}
The metric $\eta$ applied to the vector fields  $\frac{\partial}{\partial u_i}=\frac12\left(\frac{\partial}{\partial t_1}-(-1)^i\left(\frac{3}{t_2}\right)^{1/2}\frac{\partial}{\partial t_2}\right)$ is $\eta\left(\frac{\partial}{\partial u_i},\frac{\partial}{\partial u_j}\right)=\delta_{ij}\Delta_i$ where $\Delta_1=\frac{\sqrt{3}}{2}t_2^{-1/2}=-\Delta_2$.  
Restrict to the point of $M$ with coordinates $(u_1,u_2)=(2,-2)$, or equivalently $(t_1,t_2)=(0,3)$.   Then $\Delta_1=1/2=-\Delta_2$ determines the TFT and
$$U=\left(\begin{array}{cc}2&0\\0&-2\end{array}\right),\quad V=\frac{1}{6}\left(\begin{array}{cc}0&i\\-i&0\end{array}\right)
$$
determines $R(z)\in  L^{(2)}GL(2,\bc)$ and $T(z)\in z^2\bc^2[[z]]$  via \eqref{teleman} to get:
\begin{align}  \label{A2RT}
R(z)&=\sum_m\frac{(6m)!}{(6m-1)(3m)!(2m)!}\left(\begin{array}{cc}-1&(-1)^m6mi\\-6mi&(-1)^{m-1}\end{array}\right)\left(\frac{z}{1728}\right)^m\\
T(z) &= z(\un-R^{-1}(z)(\un)),\quad \un=\frac{1}{\sqrt{2}}\left(\begin{array}{c}1\\i\end{array}\right). \nonumber
\end{align}

\begin{remark}
The matrix $R(z)$ defined in \eqref{A2RT}---which uses the normalised canonical basis for $H$ so that $\eta$ is the dot product---is related to the matrix $R(z)$ in \cite{PPZRel} by conjugation by the transition matrix $\Psi$ from flat coordinates to normalised canonical coordinates
$$ R(z)=\Psi\sum_m\frac{(6m)!}{(3m)!(2m)!}\left(\begin{array}{cc}\frac{1+6m}{1-6m}&0\\0&1\end{array}\right)\left(\begin{array}{cc}0&1\\1&0\end{array}\right)^m\left(\frac{z}{1728}\right)^m\Psi^{-1}
$$  
for
$$ \Psi=\frac{1}{\sqrt{2}}\left(\begin{array}{cc}1&1\\i&-i\end{array}\right).$$
\end{remark}

The triple $(R(z),T(z),\un)\in L^{(2)}GL(N,\bc)\times z^2\bc^N[[z]]\times\bc^N$ in \eqref{A2RT} produces the cohomological field theory $\Omega^{A_2}$ associated to the $A_2$ singularity at the point $(t_1,t_2)=(0,3)$ via Definition~\ref{givact} below.

\section{Givental construction of cohomological field theories.}  \label{givental}

Givental produced a construction of partition functions of cohomological field theories in \cite{GivGro}.  He defined an action of the twisted loop group, and elements of $z^2\bc^N[[z]]$ known as translations, on partition functions of cohomological field theories and used this to build partition functions of semisimple cohomological field theories out of the basic building block $Z^{\text{KW}}(\hbar,t_0,t_1,...)$ combined with the vector $\un\in \bc^N$ which represents the topological field theory.  This action was interpreted as an action on the actual cohomology classes in $H^*(\overline{\modm}_{g,n},\bc)$ independently, by Katzarkov-Kontsevich-Pantev, Kazarian and Teleman---see \cite{PPZRel,ShaBCOV}.  

The Givental action is defined on more general sequences of cohomology classes in $H^*(\overline{\modm}_{g,n},\bc)$ such as the collection of classes $\Theta_{g,n}$ or $\Omega^\Theta_{g,n}$ defined from any CohFT $\Omega_{g,n}$  in Definition~\ref{cohftheta}.  If $\Omega_{g,n}$ is semisimple the classes $\Omega^\Theta_{g,n}$ can be obtained by applying Givental's action to the collection $\Theta_{g,n}$.

\subsubsection{The twisted loop group action} \label{twloop}

The loop group $LGL(N,\bc)$ is the group of formal series
$$R(z)  =  \sum_{k=0}^\infty R_k z^k$$
where $R_k$ are $N\times N$ matrices and $R_0\in GL(N,\bc)$.  Define the {\em twisted loop group} $L^{(2)}GL(N,\bc)\subset LGL(N,\bc)$ to be the subgroup of elements satisfying $R_0=I$ and
$$ R(z)R(-z)^T=I.
$$
Elements of $L^{(2)}GL(N,\bc)$ naturally arise out of solutions to the linear system
\begin{equation}   \label{compat} 
\left(\frac{d}{dz}-U-\frac{V}{z}\right)Y=0.
\end{equation}
where $Y(z)\in\bc^N$, $U=\diag(u_1,...,u_N)$ for $u_i$ distinct and $V$ is skew symmetric.  One can choose a solution of \eqref{compat} which behaves asymptotically for $z\to\infty$ as
$$Y(z)=R(z^{-1})e^{zU},\quad R(z)=I+R_1z+R_2z^2+...\ .$$
This defines a power series $R(z)$ with coefficients given by $N\times N$ matrices which is easily shown to satisfy $R(z)R^T(-z)=I$, hence $R(z)\in L^{(2)}GL(N,\bc)$.

Givental \cite{GivGro} constructed an action on CohFTs using a triple 
$$(R(z),T(z),\un)\in L^{(2)}GL(N,\bc)\times z^2\bc^N[[z]]\times\bc^N$$
as follows.
  For a given stable graph $\Gamma$ of genus $g$ and with $n$ external edges we have
$$\phi_{\Gamma}:\overline{\modm}_{\Gamma}=\prod_{v\in V(\Gamma)}\overline{\modm}_{g(v),n(v)}\to\overline{\modm}_{g,n}.$$
Given $(R(z),T(z),\un)\in L^{(2)}GL(N,\bc)\times z^2\bc^N[[z]]\times\bc^N$, Givental's action is defined via weighted sums over stable graphs.  For $R(z)\in L^{(2)}GL(N,\bc)$, define 
$$\ce(z,w)=\frac{I-R^{-1}(z)R^{-1}(w)^T}{z+w}=\sum_{i,j\geq 0}\ce_{ij}w^iz^j$$
which has the power series expansion on the right since $R^{-1}(z)$ is also an element of the twisted loop group so the numerator $I-R^{-1}(z)R^{-1}(w)^T$ vanishes at $w=-z$.

\begin{definition}
For a stable graph $\Gamma$ denote by
$$V(\Gamma),\quad E(\Gamma),\quad H(\Gamma),\quad L(\Gamma)=L^*(\Gamma)\sqcup L^\bullet(\Gamma)$$
its set of vertices, edges, half-edges and leaves.  The disjoint splitting of $L(\Gamma)$ into ordinary leaves $L^*$ and dilaton leaves $L^\bullet$ is part of the structure on $\Gamma$.  The set of half-edges consists of leaves and oriented edges so there is an injective map $L(\Gamma)\to H(\Gamma)$ and a multiply-defined map $E(\Gamma)\to H(\Gamma)$ denoted by $E(\Gamma)\ni e\mapsto \{e^+,e^-\}\subset H(\Gamma)$.
The map sending a half-edge to its vertex is given by $v:H(\Gamma)\to V(\Gamma)$.  Decorate $\Gamma$ by  functions:
\begin{align*}
g&:V(\Gamma)\to\bn\\
\alpha&:V(\Gamma)\to\{1,...,N\}\\
p&:L^*(\Gamma)\stackrel{\cong}{\to}\{1,2,...,n\}\\
k&:H(\Gamma)\to\bn
\end{align*}
such that  $k|_{L^\bullet(\Gamma)}>1$ and $n=|L^*(\Gamma)|$.  We write $g_v=g(v)$, $\alpha_v=\alpha(v)$, $\alpha_\ell=\alpha(v(\ell))$, $p_\ell=p(\ell)$, $k_\ell=k(\ell)$.
The {\em genus} of $\Gamma$ is $g(\Gamma)=\displaystyle b_1(\Gamma)+\hspace{-2mm}\sum_{v\in V(\Gamma)}\hspace{-2mm}g(v)$.  We say $\Gamma$ is {\em stable} if any vertex labeled by $g=0$ is of valency $\geq 3$ and there are no isolated vertices labeled by $g=1$.   We write $n_v$ for the valency of the vertex $v$.
Define $G_{g,n}$ to be the finite set of all stable, connected, genus $g$, decorated graphs with $n$ ordinary leaves and at most $3g-3+n$ dilaton leaves.  
\end{definition}
\begin{definition}[\cite{PPZRel,ShaBCOV}]   \label{givact}
Given a CohFT $\Omega'=\{\Omega'_{g,n}\}$ and 
\[(R(z),T(z))\in L^{(2)}GL(N,\bc)\times z^2\bc^N[[z]]\]
 define $R\cdot T\cdot\Omega'=\Omega=\{\Omega_{g,n}\}$
 by a weighted sum over stable graphs,
\begin{equation}   \label{graphsum}
 \Omega_{g,n}:=\sum_{\Gamma\in G_{g,n}}\frac{1}{|{\rm Aut}(\Gamma)|}(\phi_{\Gamma})_*\pi_*\hspace{-2mm}\prod_{v\in V(\Gamma)}\hspace{-2mm}w(v)\hspace{-2mm} \prod_{e\in E(\Gamma)}\hspace{-2mm}w(e)\hspace{-2mm} \prod_{\ell\in L(\Gamma)}\hspace{-2mm}w(\ell)\in H^*(\overline{\modm}_{g,n},\bc)
\end{equation}

where $\pi$ is the map that forgets dilaton leaves.  Weights are defined as follows:
\begin{enumerate}[(i)]
\item {\em Vertex weight:} $w(v)=\Omega'_{g(v),n_v}$ at each vertex $v$; 
\item {\em Edge weight:} $w(e)=\ce(\psi_e',\psi_e'')$ at each edge $e$;
\item {\em Leaf weight:} $w(\ell)=\left\{\begin{array}{ll} R^{-1}(\psi_{p(\ell)}) &\text{at each ordinary leaf } \ell\in L^*\\
T(\psi_{p(\ell)})& \text{at each dilaton leaf }\ell\in L^\bullet.\end{array}\right.$
\end{enumerate}
\end{definition}
We consider only the even part of $H^*(\overline{\modm}_{g,n},\bc)$ so \eqref{graphsum} is independent of the order in which we take the product of cohomology classes.  If $\{\Omega_{g,n}\}$ is a CohFT defined on $(\bc,\eta)$ for $H\cong\bc^N$, then the classes $\{\Omega_{g,n}\}$ in \eqref{graphsum} satisfy the same restriction conditions and hence define a CohFT on $(\bc,\eta)$ with the same degree zero, or topological field theory, terms as those of $\Omega'$.  If we choose $T(z)\equiv 0$, then the sum in \eqref{graphsum}, which is over stable graphs without dilaton leaves, defines the action of the twisted loop group on CohFTs.  If we choose $R(z)\equiv I$, then \eqref{graphsum} is a graphical realisation of the translation action of $T(z)\in z^2H[[z]]$ on a CohFT $\Omega'_{g,n}$ defined by:
\[
(T\cdot\Omega')_{g,n}(v_1\otimes..\otimes v_n)=\sum_{m\geq 0}\frac{1}{m!}\pi_*\Omega'_{g,n+m}(v_1\otimes..\otimes v_n\otimes T(\psi_{n+1})\otimes..\otimes  T(\psi_{n+m}))
\]
where $\pi:\overline{\modm}_{g,n+m}\to\overline{\modm}_{g,n}$ is the forgetful map.  The sum over $m\in\bn$ defining $(T\cdot\Omega')_{g,n}$ is finite since $T(z)\in z^2H[[z]]$, so $\dim\overline{\modm}_{g,n+m}=3g-3+n+m$ grows more slowly in $m$ than the degree $2m$ coming from $T$ resulting in at most $3g-3+n$ terms.  We can relax this condition and allow $T(z)\in zH[[z]]$ if we control the growth of the degrees of all terms of $\Omega'_{g,n}$ in $n$ to ensure $T(z)$ produces a finite sum.  In particular, $\Theta_{g,n}$, and more generally $\Omega'^\Theta_{g,n}$ for any CohFT $\Omega'_{g,n}$, is annihilated by terms of degree $>g-1$ hence the sum defining $(T\Omega')_{g,n}$ consists of at most $g-1$ terms when $T(z)\in zH[[z]]$.

The tensor product $\Omega\mapsto\Omega^\Theta$ given in Definition~\ref{cohftheta} commutes with the action of $R$ and commutes with the action of $T$ up to rescaling.  For a CohFT $\Omega$, and $R(z)\in L^{(2)}GL(N,\bc)$ and $T(z)\in z\bc^N[[z]]$
\begin{equation} \label{thetacom}
(R\cdot\Omega)^\Theta=R\cdot\Omega^\Theta,\quad (zT)\cdot\Omega^\Theta=T\cdot\Omega^\Theta.
\end{equation} 
The first relation in \eqref{thetacom} uses the restriction properties \eqref{glue} of $\Theta_{g,n}$ and the second of these uses the forgetful property \eqref{forget} of $\Theta_{g,n}$ as follows: 
\begin{align*}
\pi_*\Omega^\Theta_{g,n+m}\left(\bigotimes_{i=1}^nv_i\otimes \bigotimes_{i=1}^mT(\psi_{n+i})\right)&=\pi_*\Omega_{g,n+m}\left(\bigotimes_{i=1}^nv_i\otimes \bigotimes_{i=1}^mT(\psi_{n+i})\Theta_{g,n+m}\right)\\
=&\Theta_{g,n}\pi_*\Omega_{g,n+m}\left(\bigotimes_{i=1}^nv_i\otimes \bigotimes_{i=1}^mT(\psi_{n+i})\prod_{i=1}^m\psi_{n+i}\right)\\
=&\Theta_{g,n}\pi_*\Omega_{g,n+m}\left(\bigotimes_{i=1}^nv_i\otimes \bigotimes_{i=1}^m\psi_{n+i}T(\psi_{n+i})\right)
\end{align*}
and sum over $m$ to get $T\cdot\Omega^\Theta=(zT)\cdot\Omega^\Theta$.

The Givental-Teleman theorem \cite{GivGro,TelStr} proves that the action defined in Definition~\ref{givact} is transitive on semisimple CohFTs.  In particular, a  semisimple CohFT defined on a vector space of dimension $N$ can be constructed via the Givental action on $N$ copies of the trivial CohFT.  Given a semisimple CohFT $\Omega$, there exists 
\[(R(z),T(z),\un)\in L^{(2)}GL(N,\bc)\times z^2\bc^N[[z]]\times\bc^N\]
such that $\Omega_{g,n}$ is defined by the weighted sum over graphs \eqref{graphsum} using $R(z)$, $T(z)$ and $\Omega'_{g,n}$ given by the topological field theory underlying $\Omega_{g,n}$.  Note that a semi-simple topological field theory of dimension $N$ is equivalent to $\un\in\bc^N$ which gives the unit vector in terms of a basis in which the product is diagonal and the metric $\eta$ is the dot product, known as a {\em normalised canonical basis}.

On the level of partition functions, the construction of a semisimple CohFT from the trivial CohFT is realised via an action of quantised differential operators $\hat{R}$ and $\hat{T}$ on products of $Z^{\text{KW}}(\hbar,t_0,t_1,...)$, a KdV tau function defined in the next section.  
\begin{definition} \label{ops}
For $R(z)\hspace{-1mm}=\hspace{-.5mm}\displaystyle\exp(\sum_{\ell>0} r_\ell z^\ell)\hspace{-1mm}\in L^{(2)}GL(N,\bc)$, $\displaystyle T(z)=\hspace{-1mm} \sum_{k>0} T^\alpha_kz^k\hspace{-1mm}\in z\bc^N[[z]]$, 
\begin{align*}
\hat{R}&:=\exp\left\{\sum_{\ell=1}^{\infty}\sum_{\alpha,\beta}
\left(\sum_{k=0}^{\infty}(r_k)_\beta^\alpha t_k^{\beta}\frac{\partial}{\partial t_{k+\ell}^{\alpha}}+\frac{\hbar}{2}\sum_{m=0}^{\ell-1}(-1)^{m+1}(r_\ell)^{\alpha}_{\beta}\frac{\partial^2}{\partial t_m^{\alpha}\partial t_{\ell-m-1}^{\beta}}\right)
\right\}\\
\hat{T}&:=\exp\left(\sum_{\alpha=1}^m\displaystyle\sum_{k>0} T^\alpha_k\frac{\partial}{\partial t_k^{\alpha}}\right).
\end{align*}
\end{definition}
The partition function of \eqref{graphsum} is given in \cite{DSSGiv,GivGro,ShaBCOV} by the following formula: 
\begin{align}  \label{givpart}
Z_{\Omega}(\hbar,\{t^{\alpha}_k\})&=\hat{R}\cdot\hat{T}\cdot\hat{\un}\cdot Z^{\text{KW}}(\hbar,\{t_k^{1}\})\cdots Z^{\text{KW}}(\hbar,\{t_k^{N}\})\\
&=\exp\Big\{\sum_{g,n}\hbar^{g-1}\hspace{-1mm}\sum_{\Gamma\in G_{g,n}}\frac{1}{|{\rm Aut}(\Gamma)|}\prod_{v\in V(\Gamma)}\hspace{-2mm}\hat{w}(v)\hspace{-1mm} \prod_{e\in E(\Gamma)}\hspace{-2mm}\hat{w}(e)\hspace{-1mm} \prod_{\ell\in L(\Gamma)}\hspace{-2mm}\hat{w}(\ell)\Big\}. \nonumber
\end{align}
The operator $\hat{\un}$ rescales the variables $\hat{\Delta}\cdot Z^{\text{KW}}(\hbar,\{t_k^{\alpha}\})= Z^{\text{KW}}((\un^\alpha)^2\hbar,\{\un^\alpha t_k^{\alpha}\})$.  Vertex weights $\hat{w}(v)$  store products of $Z^{\text{KW}}$ corresponding to the partition function of a topological field theory, edge weights $\hat{w}(e)$ store coefficients of the series $\ce(w,z)$, and leaf weights $\hat{w}(\ell)$ store the variables $t_k^\alpha$ in a series weighted by coefficients of the series $R^{-1}(-z)$.  We do not give explicit formulae for the weights, see \cite{DSSGiv,GivGro,ShaBCOV}, and instead use an equivalent elegant formulation given by topological recursion, defined in Section~\ref{TR}.
 
 A consequence of the relations \eqref{thetacom} is the following proposition which modifies the construction of a semisimple CohFT $\Omega$ to produce $\Omega^\Theta$.
\begin{proposition}   \label{givtheta}
Given a semisimple CohFT $\Omega$ defined  via \eqref{graphsum} using
\[(R(z),T(z),\un)\in L^{(2)}GL(N,\bc)\times z^2\bc^N[[z]]\times\bc^N\]
then the collection of classes $\Omega^\Theta$ is defined via \eqref{graphsum} using
\[(R(z),\frac{1}{z}T(z),\un)\in L^{(2)}GL(N,\bc)\times z\bc^N[[z]]\times\bc^N\]
and 
\[\Omega'_{g,n}=\Theta_{g,n}\otimes\Omega^{(0)}_{g,n}:H^{\otimes n}\to H^{4g-4+2n}(\overline{\modm}_{g,n},\bc)\] 
for $\Omega^{(0)}_{g,n}$ the degree 0 part of $\Omega_{g,n}$ determined by the vector $\un\in\bc^N$.  Its partition function $Z_{\Omega^\Theta}(\hbar,\{t^{\alpha}_k\})$ is obtained by replacing each copy of $Z^{\text{KW}}(\hbar,\{t_k\})$ in \eqref{givpart} by a copy of $Z^\Theta(\hbar,\{t_k\})$ and shifting the operator $\hat{T}$.
\end{proposition}

\subsection{KdV tau functions}   \label{sec:kdv}

The KdV hierarchy is a sequence of partial differential equations beginning with the KdV equation.
\begin{equation}\label{kdv}
U_{t_1}=UU_{t_0}+\frac{\hbar}{12}U_{t_0t_0t_0},\quad U(t_0,0,0,...)=f(t_0).
\end{equation}
A tau function $Z(t_0,t_1,...)$ of the KdV hierarchy (equivalently the KP hierarchy in odd times $p_{2m+1}=t_m/(2m+1)!!$) gives rise to a solution $U=\hbar\frac{\partial^2}{\partial t_0^2}\log Z$ of the KdV hierarchy. 
The first equation in the hierarchy is the KdV equation \eqref{kdv}, and later equations $U_{t_k}=P_k(U,U_{t_0},U_{t_0t_0},...)$ for $k>1$ determine $U$ uniquely from $U(t_0,0,0,...)$.  See \cite{MJDSol} for the full definition.

 The Kontsevich-Witten tau function $Z^{\text{KW}}$ is defined by the initial condition $U^{\text{KW}}(t_0,0,0,...)=t_0$ for $U^{\text{KW}}=\hbar\frac{\partial^2}{\partial t_0^2}\log Z^{\text{KW}}$.  The low genus terms of $\log Z^{\text{KW}}$ are 
$$\log Z^{\text{KW}}(\hbar,t_0,t_1,...)=\hbar^{-1}(\frac{t_0^3}{3!}+\frac{t_0^3t_1}{3!}+\frac{t_0^4t_2}{4!}+...)+\frac{t_1}{24}+...
$$
\begin{theorem}[Witten-Kontsevich 1992 \cite{KonInt,WitTwo}]  \label{KW}
$$Z^{\text{KW}}(\hbar,t_0,t_1,...)=\exp\sum_{g,n}\hbar^{g-1}\frac{1}{n!}\sum_{\vec{k}\in\bn^n}\int_{\overline{\modm}_{g,n}}\prod_{i=1}^n\psi_i^{m_i}t_{m_i}
$$
is a tau function of the KdV hierarchy.
\end{theorem}

The Br\'ezin-Gross-Witten solution $U^{\text{BGW}}=\hbar\frac{\partial^2}{\partial t_0^2}\log Z^{\text{BGW}}$ of the KdV hierarchy arises out of a unitary matrix model studied in \cite{BGrExt,GWiPos}.  It is defined by the initial condition
$$
U^{\text{BGW}}(t_0,0,0,...)=\frac{\hbar}{8(1-t_0)^2}.
$$
The low genus $g$ terms (= coefficient of $\hbar^{g-1}$) of $\log Z^{\text{BGW}}$ are
\begin{align}  \label{lowg}
\log Z^{\text{BGW}}=&-\frac{1}{8}\log(1-t_0)+\hbar\frac{3}{128}\frac{t_1}{(1-t_0)^3}+\hbar^2\frac{15}{1024}\frac{t_2}{(1-t_0)^5}\\
&+\hbar^2\frac{63}{1024}\frac{t_1^2}{(1-t_0)^6}+O(\hbar^3)\nonumber\\
=\frac{1}{8}t_0+&\frac{1}{16}t_0^2+..+\hbar(\frac{3}{128}t_1+\frac{9}{128}t_0t_1+..)+\hbar^2(\frac{15}{1024}t_2+\frac{63}{1024}t_1^2+..)\nonumber
\end{align}
It satisfies the homogeneity property 
\[\frac{\partial}{\partial t_0}Z^{\text{BGW}}(\hbar,t_0,t_1,...)=\sum_{i=0}^{\infty}(2i+1)t_i \frac{\partial}{\partial t_i}Z^{\text{BGW}}(\hbar,t_0,t_1,...)+\frac18Z^{\text{BGW}}(\hbar,t_0,t_1,...)\]
which coincides with \eqref{dilaton} satisfied by $Z^\Theta(\hbar,t_0,t_1,...)$.
A proof of this homogeneity property for $Z^{\text{BGW}}$ can be found in \cite{AleCut,DNoTop}.

The tau function $Z^{\text{BGW}}(\hbar,t_0,t_1,...)$ shares many properties of the famous Kont\-sevich-Witten tau function $Z^{\text{KW}}(\hbar,t_0,t_1,...)$ introduced in \cite{WitTwo}.   An analogue of Theorem~\ref{KW} is given by Conjecture~\ref{kdvconj} which postulates that the function
\[ Z^{\Theta}(\hbar,t_0,t_1,...)=\exp\sum_{g,n,\vec{k}}\frac{\hbar^{g-1}}{n!}\int_{\overline{\modm}_{g,n}}\Theta_{g,n}\cdot\prod_{j=1}^n\psi_j^{k_j}\prod t_{k_j}\]
coincides with $Z^{\text{BGW}}(\hbar,t_0,t_1,...)$.

The tau function $Z^{\text{BGW}}$ appears in a generalisation of Givental's decomposition of CohFTs in \cite{CNoTop}.  
\begin{definition}  \label{BGWCohFTpart}
Given a semisimple CohFT $\Omega$ with partition function $Z_\Omega(\hbar,\{t^{\alpha}_k\})$ constructed as a graphical sum via \eqref{givpart}
\[Z_{\Omega}(\hbar,\{t^{\alpha}_k\})=\hat{R}\cdot\hat{T}\cdot\hat{\un}\cdot Z^{\text{KW}}(\hbar,\{t_k^{1}\})\cdots Z^{\text{KW}}(\hbar,\{t_k^{N}\})\] 
define
\[Z^{\text{BGW}}_{\Omega}(\hbar,\{t^{\alpha}_k\})=\hat{R}\cdot\hat{T_0}\cdot\hat{\un}\cdot Z^{\text{BGW}}(\hbar,\{t_k^{1}\})\cdots Z^{\text{BGW}}(\hbar,\{t_k^{N}\})\] 
where $T_0=\frac{1}{z}T(z)$.
\end{definition}
The same shift $T_0=\frac{1}{z}T(z)$ is used by $Z^{\text{BGW}}(\hbar,\{t_k\})$ and  $Z^\Theta(\hbar,\{t_k\})$ due to their common homogeneity property \eqref{dilaton}.  
One can also replace only some copies of $Z^{\text{KW}}(\hbar,\{t_k\})$ in \eqref{givpart} by copies of $Z^{\text{BGW}}(\hbar,\{t_k\})$ and shift components of $\hat{T}$.
For example, in \cite{DNoTopI} the enumeration of bipartite dessins d'enfant is shown to have partition function 
\begin{equation}   \label{despart}
Z(\hbar,\{t_k^{\alpha}\})
=\hat{R}\cdot \hat{T}\cdot Z^{\text{BGW}}(-\frac12\hbar,\{\frac{i}{\sqrt{2}}t_k^{1}\})Z^{\text{KW}}(32\hbar,\{4\sqrt{2}t_k^{2}\}) 
\end{equation}
for $R$ and $T$ determined by the curve $xy^2+xy+1=0$ as described in Section~\ref{TR}.

\subsection{Topological recursion.}   \label{TR}
Figure~\ref{fig:cons} summarises the contents of this section.  
\begin{figure}[H]
\centering

\begin{tikzpicture}[scale=0.5]
\draw (1,4) node {$R(z)\in L^{(2)}GL(N,\bc)$};
\draw (1,3) node {$T(z)\in z^2 \bc^N[[z]],\un\in\bc^N$};
\draw [->, line width=1pt] (5,4)--(13,4);
\draw (9,5) node {Givental construction};
\draw (15.5,4) node {$Z_\Omega(\hbar,\{t^{\alpha}_k\})$};
\draw [<->, line width=1pt] (1,2.5)--(1,-1.5);
\draw [<->, line width=1pt] (15.5,2.5)--(15.5,-1.5);
\draw (1,-2.5) node {$S=(C,x,y,B)$};
\draw (15.5,-2.5) node {$Z^S(\hbar,\{t^{\alpha}_k\})$};
\draw [->, line width=1pt] (5,-2.5)--(13,-2.5);
\draw (9,-2) node {topological recursion};
\end{tikzpicture}
\caption{Constructions of CohFT partition functions} \label{fig:cons}
\end{figure}
The upper horizontal arrow in the figure represents Givental's construction of a partition function defined in \eqref{givpart} and Definition~\ref{givact}.  Topological recursion is defined in Section~\ref{TR}---it produces a partition function from a spectral curve $S=(C,x,y,B)$ consisting of a Riemann surface $C$ equipped with meromorphic functions $x$ and $y$ and a bidifferential $B$.   We begin with a description of the left vertical arrow which represents the construction of an element $R(z)\in L^{(2)}GL(N,\bc)$ from $(C,x,B)$ in \eqref{eq:BtoR} and $T(z)$ and $\un$ from $(C,x,y)$ in \eqref{ytoT} and \eqref{ytoU}.  We then define topological recursion in \ref{sec:TR} and state the result of \cite{DOSSIde}, that topological recursion encodes the graphical construction in \eqref{givpart} and gives equality of partition functions, represented by the right vertical arrow.

An element of the twisted loop group $R(z)\in L^{(2)}GL(N,\bc)$ can be naturally constructed from a Riemann surface $\Sigma$ equipped with a bidifferential $B(p_1,p_2)$ on $\Sigma\times\Sigma$ and a meromorphic function $x:\Sigma\to\bc$, for $N=$ the number of zeros of $dx$.  A basic example is the function $x=z^2$ on $\Sigma=\bc$ which gives rise to the constant element $R(z)=1\in GL(1,\bc)$.  More generally, any function $x$ that looks like this example locally, i.e. $x=s^2+c$ for $s$ a local coordinate around a zero of $dx$ and $c\in\bc$, gives $R(z)=I+R_1z+...\in L^{(2)}GL(N,\bc)$ which is in some sense a deformation of $I\in GL(N,\bc)$, or $N$ copies of the basic example.

\begin{definition} \label{bidiff} On any compact Riemann surface $(\Sigma,\{ {\cal A}_i\}_{i=1,...,g})$ with a choice of $\cal A$-cycles, define a {\em fundamental normalised bidifferential of the second kind}  $B(p,p')$ to be a symmetric tensor product of differentials on $\Sigma\times\Sigma$, uniquely defined by the properties that it has a double pole on the diagonal of zero residue, double residue equal to $1$, no further singularities and normalised by $\int_{p\in{\cal A}_i}B(p,p')=0$, $i=1,...,g$, \cite{FayThe}.  On a rational curve, which is sufficient for the examples in this paper, $B$ is the Cauchy kernel
$$
B(z_1,z_2)=\frac{dz_1dz_2}{(z_1-z_2)^2}.$$  
\end{definition}
The bidifferential $B(p,p')$ acts as a kernel for producing meromorphic differentials on the Riemann surface $\Sigma$ via $\omega(p)=\int_{\Lambda}\lambda(p')B(p,p')$ where $\lambda$ is a function defined along the contour $\Lambda\subset\Sigma$.  Depending on the choice of $(\Lambda,\lambda)$, $\omega$ can be a differential of the 1st kind (holomorphic), 2nd kind (zero residues) or 3rd kind (simple poles).  
\begin{definition}\label{evaluationform}
For $(\Sigma,x)$, a Riemann surface equipped with a meromorphic function, define evaluation of any meromorphic differential $\omega$ at a simple zero $\cp$ of $dx$ by
$$
\omega(\cp):=\Res_{p=\cp}\frac{\omega(p)}{\sqrt{2(x(p)-x(\cp))}}
$$
where we choose a branch of $\sqrt{x(p)-x(\cp)}$ once and for all at $\cp$ to remove the $\pm1$ ambiguity.
\end{definition}
A fundamental example of Definition~\ref{evaluationform} required here is $B(\cp,p)$ which is a normalised (trivial $\cal A$-periods) differential of the second kind holomorphic on $\Sigma\backslash\cp$ with a double pole at the simple zero $\cp$ of $dx$.  

In order to produce an element of the twisted loop group,
Shramchenko \cite{ShrRie} constructed a solution $Y(z)$ of the linear system \eqref{compat} using $V=[\cb,U]$ for $\cb_{\alpha\beta}=B(\cp_\alpha,\cp_\beta)$ (defined for $\alpha\neq\beta$) given by
$$Y(z)^\alpha_\beta= -\frac{\sqrt{z}}{\sqrt{2\pi}}\int_{\Gamma_\beta} B(\cp_\alpha,p)\cdot e^{\frac{-x(p)}{z}}.$$
The proof in \cite{ShrRie} is indirect, showing that $Y(z)^i_j$ satisfies an associated set of PDEs in $u_i$, and using the Rauch variational formula to calculate $\partial_{u_k}B(\cp_\alpha,p)$.  Instead, here we work directly with the associated element $R(z)$ of the twisted loop group.
\begin{definition}  \label{BtoR}
Define the asymptotic series in the limit $z\to 0$ by
\beq  \label{eq:BtoR}
R^{-1}(z)^\alpha_\beta = -\frac{\sqrt{z}}{\sqrt{2\pi}}\int_{\Gamma_\beta} B(\cp_\alpha,p)\cdot e^{\frac{x(\cp_\beta)-x(p)}{ z}}
\eeq
where $\Gamma_\beta$ is a path of steepest descent for $-x(p)/z$ containing $x(\cp_\beta)$.  
\end{definition}
Note that the asymptotic expansion of the contour integral \eqref{eq:BtoR} for $z\to 0$ depends only the intersection of $\Gamma_\beta$ with a neighbourhood of $p=\cp_\beta$.  When $\alpha=\beta$, the integrand has zero residue at $p=\cp_\beta$ so we deform $\Gamma_\beta$ to go around $\cp_\beta$ to get a well-defined integral.  Locally, this is the same as defining $\int_\br s^{-2}\exp(-s^2)ds=-2\sqrt{\pi}$ by  integrating the analytic function $z^{-2}\exp(-z^2)$ along the real line in $\bc$ deformed to avoid 0.
\begin{lemma} [\cite{ShrRie}]  \label{eynlem}
The asymptotic series $R(z)$ defined in \eqref{eq:BtoR} satisfies the twisted loop group condition
\beq\label{twloopgp}
R(z) R^T(-z) = Id.
\eeq
 \end{lemma}
 \begin{proof}
The proof here is taken from \cite{DNOPSPri}.  We have
 \begin{align}  \label{unberg}
\sum_{\alpha=1}^N\Res_{q=\cp_\alpha}\frac{B(p,q)B(p',q)}{dx(q)}&=-\Res_{q=p}\frac{B(p,q)B(p',q)}{dx(q)}-\Res_{q=p'}\frac{B(p,q)B(p',q)}{dx(q)}\\ \nonumber
 &=-d_p\left(\frac{B(p,p')}{dx(p)}\right)-d_{p'}\left(\frac{B(p,p')}{dx(p')}\right)
 \end{align}
where the first equality uses the fact that the only poles of the integrand occur at $\{p,p',\cp_\alpha,\alpha=1,...,N\}$, and the second equality uses the Cauchy formula satisfied by the Bergman kernel.  Define the Laplace transform of the Bergman kernel by
$$
\check{B}^{\alpha,\beta}(z_1,z_2)  =
\frac{ e^{ \frac{x(\cp_\alpha) }{ z_1}+ \frac{x(\cp_\beta)}{ z_2}}}{2 \pi \sqrt{z_1z_2}} \int_{\Gamma_\alpha}
\int_{\Gamma_\beta} B(p,p') e^{- \frac{x(p) }{ z_1}- \frac{x(p')}{ z_2}} .
$$
The Laplace transform of the LHS of \eqref{unberg} is
 \begin{align*} 
 \frac{ e^{ \frac{x(\cp_\alpha) }{ z_1}+ \frac{x(\cp_\beta)}{ z_2}}}{2 \pi \sqrt{z_1z_2}} \int_{\Gamma_\alpha}
\int_{\Gamma_\beta} &e^{- \frac{x(p) }{ z_1}- \frac{x(p')}{ z_2}}\sum_{\gamma=1}^N\Res_{q=\cp_\gamma}\frac{B(p,q)B(p',q)}{dx(q)}\\
&=\sum_{\gamma=1}^N\frac{ e^{ \frac{x(\cp_\alpha)}{ z_1}+ \frac{x(\cp_\beta)}{ z_2}}}{2 \pi \sqrt{z_1z_2}} \int_{\Gamma_\alpha}
e^{- \frac{x(p) }{ z_1}}B(p,\cp_\gamma) \int_{\Gamma_\beta}e^{- \frac{x(p')}{ z_2}}B(p',\cp_\gamma)\\
&=\sum_{\gamma=1}^N\frac{\left[R^{-1}(z_1)\right]^\gamma_\alpha\left[R^{-1}(z_2)\right]_{\beta}^\gamma}{ z_1z_2}.
 \end{align*}
Since the Laplace transform satisfies
$\displaystyle\int_{\Gamma_\alpha} d\left(\frac{\omega(p)}{dx(p)}\right)e^{- \frac{x(p) }{ z}}=\frac{1}{z}\int_{\Gamma_\alpha} \omega(p)e^{- \frac{x(p) }{ z}}
$
for any differential $\omega(p)$, by integration by parts,
then the Laplace transform of the RHS of \eqref{unberg} is
 \begin{align*} 
 -\frac{ e^{ \frac{x(\cp_\alpha)}{ z_1}+ \frac{x(\cp_\beta)}{ z_2}}}{2 \pi \sqrt{z_1z_2}} \int_{\Gamma_\alpha}
\int_{\Gamma_\beta} e^{- \frac{x(p) }{ z_1}- \frac{x(p')}{ z_2}}
&\left\{d_p\left(\frac{B(p,p')}{dx(p)}\right)+d_{p'}\left(\frac{B(p,p')}{dx(p')}\right)\right\}\\
=&-\left(\frac{1}{z_1}+\frac{1}{z_2}\right)\check{B}^{\alpha,\beta}(z_1,z_2).
 \end{align*}
Putting the two sides together yields the following result due to Eynard \cite{EynInv}
\beq\label{shape2}
\check{B}^{\alpha,\beta}(z_1,z_2) =    - \frac{\sum_{\gamma=1}^{N} \left[R^{-1}(z_1)\right]^\gamma_\alpha \left[R^{-1}(z_2)\right]^k_\beta }{ z_1+z_2}.
\eeq
Equation \eqref{twloopgp} is an immediate consequence of \eqref{shape2} and the finiteness of $\check{B}^{\alpha,\beta}(z_1,z_2)$ at $z_2=-z_1$.
\end{proof}

\subsubsection{Topological recursion}   \label{sec:TR}

Topological recursion is a procedure which takes as input a spectral curve, defined below, and produces a collection of symmetric tensor products of meromorphic 1-forms $\omega_{g,n}$ on $C^n$.   The correlators store enumerative information in different ways.  Periods of the correlators store top intersection numbers of tautological classes in the moduli space of stable curves $\overline{\modm}_{g,n}$ and local expansions of the correlators can serve as generating functions for enumerative problems.  

A spectral curve $S=(C,x,y,B)$ is a Riemann surface $C$ equipped with two meromorphic functions $x, y: C\to \mathbb{C}$ and a bidifferential $B(p_1,p_2)$ defined in \eqref{bidiff}, which is the Cauchy kernel in this paper.   Topological recursion, as developed by Chekhov, Eynard, Orantin \cite{CEyHer,EOrInv}, is a procedure that produces from a spectral curve $S=(C,x,y,B)$ a symmetric tensor product of meromorphic 1-forms $\omega_{g,n}$ on $C^n$ for integers $g\geq 0$ and $n\geq 1$, which we refer to as {\em correlation differentials} or {\em correlators}.  The correlation differentials $\omega_{g,n}$ are defined by
\[
\omega_{0,1}(p_1) = -y(p_1) \, d x(p_1) \qquad \text{and} \qquad \omega_{0,2}(p_1, p_2) = B(p_1,p_2) 
\]
and for $2g-2+n>0$ they are defined recursively via the following equation.
\[
\omega_{g,n}(p_1, p_L) \hspace{-1mm}=\hspace{-3mm} \sum_{_{d x(\alpha) = 0}} \hspace{-2.5mm}\mathop{\text{Res}}_{p=\alpha} K(p_1, p)\hspace{-.8mm} \Bigg[\hspace{-.5mm} \omega_{g-1,n+1}(p, \hat{p}, p_L)\hspace{-.5mm} + \hspace{-4.5mm}\mathop{\sum_{g_1+g_2=g}}_{I \sqcup J = L}^\circ\hspace{-4mm} \omega_{g_1,|I|+1}(p, p_I) \, \omega_{g_2,|J|+1}(\hat{p}, p_J)\hspace{-.5mm} \Bigg]
\]
Here, we use the notation $L = \{2, 3, \ldots, n\}$ and $p_I = \{p_{i_1}, p_{i_2}, \ldots, p_{i_k}\}$ for $I = \{i_1, i_2, \ldots, i_k\}$. The outer summation is over the zeroes $\alpha$ of $dx$ and $p \mapsto \hat{p}$ is the involution defined locally near $\alpha$ satisfying $x(\hat{p}) = x(p)$ and $\hat{p} \neq p$. The symbol $\circ$ over the inner summation means that we exclude any term that involves $\omega_{0,1}$. Finally, the recursion kernel is given by
\[
K(p_1,p) = -\frac{1}{2}\frac{\int_{\hat{p}}^p \omega_{0,2}(p_1, \,\cdot\,)}{[y(p)-y(\hat{p})] \, d x(p)}.
\]
which is well-defined in the vicinity of each zero of $dx$.   It acts on differentials in $p$ and produces differentials in $p_1$ since the quotient of a differential in $p$ by the differential $dx(p)$ is a meromorphic function.  For $2g-2+n>0$, each $\omega_{g,n}$ is a symmetric tensor product of meromorphic 1-forms on $C^n$ with residueless poles at the zeros of $dx$ and holomorphic elsewhere.  A zero $\alpha$ of $dx$ is {\em regular}, respectively irregular, if $y$ is regular, respectively has a simple pole, at $\alpha$.  A spectral curve is irregular if it contains an irregular zero of $dx$.  The order of the pole in each variable of $\omega_{g,n}$ at a regular, respectively irregular, zero of $dx$ is $6g-4+2n$, respectively $2g$.  Define $\Phi(p)$ up to an additive constant by $d\Phi(p)=y(p)dx(p)$.  For $2g-2+n>0$, the invariants satisfy the dilaton equation~\cite{EOrInv}
\[
\sum_{\alpha}\Res_{p=\alpha}\Phi(p)\, \omega_{g,n+1}(p,p_1, \ldots ,p_n)=(2g-2+n) \,\omega_{g,n}(p_1, \ldots, p_n),
\] 
where the sum is over the zeros $\alpha$ of $dx$.  This enables the definition of the so-called {\em symplectic invariants}
\[ F_g=\sum_{\alpha}\Res_{p=\alpha}\Phi(p)\omega_{g,1}(p).\]
The correlators $\omega_{g,n}$ are normalised differentials of the second kind in each variable since they have zero $\ca$-periods, and poles only at the zeros $\cp_\alpha$ of $dx$ of zero residue.  Their principal parts are skew-invariant under the local involution $p\mapsto\hat{p}$.  A basis of such normalised differentials of the second kind is constructed from $x$ and $B$ in the following definition. 
\begin{definition}\label{auxdif}
For a Riemann surface $C$ equipped with a meromorphic function $x:C\to\bc$ and bidifferential $B(p_1,p_2)$ define the auxiliary differentials on $C$ as follows.  For each zero $\cp_\alpha$ of $dx$,  define
\begin{equation}  \label{Vdiff}
V^\alpha_0(p)=B(\cp_\alpha,p),\quad V^\alpha_{k+1}(p)=-d\left(\frac{V^\alpha_k(p)}{dx(p)}\right),\ \alpha=1,...,N,\quad k=0,1,2,...
\end{equation}
where evaluation $B(\cp_\alpha,p)$ at  $\cp_\alpha$ is given in Definition~\ref{evaluationform}.
\end{definition}
The correlators $\omega_{g,n}$ are polynomials in the auxiliary differentials $V^\alpha_k(p)$.  To any spectral curve $S$, one can define a partition function $Z^S$ by assembling the polynomials built out of the correlators $\omega_{g,n}$ \cite{DOSSIde,EynInv,NScPol}.
\begin{definition}  \label{TRpart}
$$Z^S(\hbar,\{u^{\alpha}_k\}):=\left.\exp\sum_{g,n}\frac{\hbar^{g-1}}{n!}\omega^S_{g,n}\right|_{V^{\alpha}_k(p_i)=u^{\alpha}_k}.
$$
\end{definition}
As usual define $F_g$ to be the contribution from $\omega_{g,n}$:
$$\log Z^S(\hbar,\{u^{\alpha}_k\})=\sum_{g\geq 0}\hbar^{g-1}F_g^S(\{u^{\alpha}_k\}).$$

\subsubsection{From topological recursion to Givental's construction}   \label{sec:TR2Giv}
The input data for Givental's construction is a triple $(R(z),T(z),\un)\in L^{(2)}GL(N,\bc)\times z^2\bc^N[[z]]\times \bc^N$.  Its output is a CohFT $\Omega$, and its partition function $Z_\Omega(\hbar,\{t^{\alpha}_k\})$.  The input data for topological recursion is a spectral curve $S=(C,x,y,B)$.  Its output is the correlators $\omega_{g,n}$ which can be assembled into a partition function $Z^S(\hbar,\{t^{\alpha}_k\})$. 

From a compact spectral curve define a triple
$$S=(C,x,y,B)\to(R(z),T(z),\un)\in L^{(2)}GL(N,\bc)\times z\bc^N[[z]]\times\bc^N$$ 
by
$$(C,x,B)\mapsto R(z)\in L^{(2)}GL(N,\bc)$$
via \eqref{eq:BtoR}, 
\begin{equation} \label{ytoU}
\un^i=\left\{\begin{array}{ll}dy(\cp_\alpha),&\cp_\alpha\text{ regular}\\
(ydx)(\cp_\alpha),&\cp_\alpha\text{ irregular}\end{array}\right.
\end{equation}
which is the unit in normalised canonical coordinates, and 
\begin{equation} \label{ytoT}
T(z)^\alpha=\left\{\begin{array}{ll}z\left(\un^\alpha- \frac{1}{\sqrt{2\pi z}}\int_{\Gamma_\alpha}dy(p)\cdot e^{\frac{x(\cp_\alpha)-x(p) }{z}}\right),&\cp_\alpha\text{ regular}\\
\un^\alpha- \frac{1}{\sqrt{2\pi z}}\int_{\Gamma_\alpha}y(p)dx(p)\cdot e^{\frac{x(\cp_\alpha)-x(p)}{z}},&\cp_\alpha\text{ irregular}\end{array}\right.
\end{equation}
Note that 
$$\lim_{z\to0}\frac{1}{\sqrt{2\pi z}}\int_{\Gamma_\alpha}dy(p)\cdot e^{\frac{x(\cp_\alpha)-x(p)}{ z}}=\left\{\begin{array}{ll}dy(\cp_\alpha),&\cp_\alpha\text{ regular}\\(ydx)(\cp_\alpha),&\cp_\alpha\text{ irregular}\end{array}\right.$$ 
which defines $\un$, hence the right hand side of \eqref{ytoT} lives in $z^2\bc^N[[z]]$, respectively $z\bc^N[[z]]$, when $\cp_\alpha$ is regular, respectively irregular.  When $\Omega$ is a CohFT with flat identity---see \eqref{cohforget} in Section~\ref{sec:cohft}---given by $\un\in\bc^N$, then $\un$ determines the translation via $T(z)=z\left(\un-R^{-1}(z)\un\right)\in z^2\bc^N[[z]]$.
In this special case $y$ satisfies
\begin{equation} \label{ytoTunit}
(R^{-1}(z)\un)^\alpha=\sum_{k=1}^N R^{-1}(z)^\alpha_k \cdot\Delta_k^{1/2} = \frac{1}{\sqrt{2\pi z}}\int_{\Gamma_\alpha}dy(p)\cdot e^{\frac{x(\cp_\alpha)-x(p) }{ z}}
\end{equation}
which uniquely determines $y$ from its first order data $\{dy(\cp_\alpha)\}$ at each $\cp_\alpha$.

The map $(C,x,y,B)\mapsto(R(z),T(z),\un)$ produces the left vertical arrow in Figure~\ref{fig:cons} and its generalisation to irregular spectral curves, i.e. a correspondence between the input data, and via the graphical construction \eqref{givpart} this produces the same output $Z_\Omega(\hbar,\{t^{\alpha}_k\})=Z^S(\hbar,\{t^{\alpha}_k\})$ which is the main result of \cite{DOSSIde} stated in the following theorem. 
\begin{theorem}[\cite{DOSSIde}]   \label{DOSS}
Given a CohFT $\Omega$ built from 
\[R(z)\in L^{(2)}GL(N,\bc),\quad T(z)\in z^2\bc^N[[z]],\quad \un\in\bc^N
\] via Definition~\ref{givact}, there exists a local spectral curve \
\[S=(C,x,y,B)\mapsto(R(z),T(z),\un)\] 
on which $x$ and $B$ correspond to $R(z)$ via Definition~\ref{BtoR} and $y$ corresponds to $T(z)$ and $\un$ via \eqref{ytoT} and \eqref{ytoU}, giving the partition function of the CohFT
$$Z_\Omega(\hbar,\{t^{\alpha}_k\})=Z^S(\hbar,\{t^{\alpha}_k\}).$$
\end{theorem}
In general, the spectral curve $S$ in Theorem~\ref{DOSS} is a local spectral curve which is a collection of disks neighbourhoods of zeros of $dx$ on which $B$ and $y$ are define locally, although we only consider compact spectral curves $S$ in this paper.  Theorem~\ref{DOSS} was proven only in the case $T(z)=z\left(\un-R^{-1}(z)\un\right)$ in \cite{DOSSIde} but it has been generalised to allow any $T(z)\in z^2\bc^N[[z]]$---see \cite{CNoTop,LPSZChi}.
We will use the converse of Theorem~\ref{DOSS} proven in \cite{DNOPSPri}, beginning instead from $S$.  Theorem~\ref{DOSS} was also generalised in \cite{CNoTop} to show that the operators $\hat{\Psi}$, $\hat{R}$ and $\hat{T}$ acting on copies of $Z^{\text{BGW}}$ analogous to \eqref{givpart}  
arises by applying topological recursion to an irregular spectral curve.   Equivalently, periods of the correlators of an irregular spectral curve store linear combinations of coefficients of $\log Z^{\text{BGW}}$.  The appearance of $Z^{\text{BGW}}$ is due to its relationship with topological recursion applied to the curve $x=\frac{1}{2}z^2$, $y=\frac{1}{z}$ \cite{DNoTop}.

\subsubsection{Spectral curve examples}
We demonstrate Theorem~\ref{DOSS} with four key examples of rational spectral curves equipped with the bidifferential $B(p_1,p_2)$ given by the Cauchy kernel. The spectral curves in the Examples~\ref{specAiry} and \ref{specBes}, denoted  $S_{\text{Airy}}$ and $S_{\text{Bes}}$, have partition functions $Z^{\text{KW}}$ and $Z^{\text{BGW}}$ respectively.  
Any spectral curve at regular, respectively irregular, zeros of $dx$ is locally isomorphic to $S_{\text{Airy}}$, respectively $S_{\text{Bes}}$.   
A consequence is that the tau functions $Z^{\text{KW}}$ and $Z^{\text{BGW}}$ are fundamental to the correlators produced from topological recursion.  Moreover, the topological recursion partition function $Z^S$ is constructed via \eqref{givpart}, using a product of copies of $Z^{\text{KW}}$ {\em and} copies of $Z^{\text{BGW}}$, as in \eqref{despart}, where $R$ and $T$ are obtained from the spectral curve as described in Section~\ref{sec:TR2Giv}.  The third example, given by Theorem~\ref{LPSZ}, brings together $Z^{\text{KW}}$ and $Z^{\Theta}$ and conjecturally $Z^{\text{BGW}}$ in the limit. Proposition~\ref{givtheta}, which gives the relationship between the Givental construction of a semisimple CohFT $\Omega$ and its associated $\Omega^{\text{BGW}}$, has an elegant consequence for spectral curves.  This is demonstrated explicitly in the fourth example which shows the relationship between the spectral curves of a CohFT $\Omega^{A_2}$ associated to the $A_2$ singularity and $(\Omega^{A_2})^{\text{BGW}}$.

Examples~\ref{specAiry} and \ref{specBes} below use the differentials 
\[\xi_m(z)=(2m+1)!!z^{-(2m+2)}dz\] 
defined by \eqref{Vdiff} for $x=\frac12z^2$ with respect to a global rational parameter $z$ for the curve $C\cong\bc$.  

\begin{example}   \label{specAiry}
Topological recursion applied to the Airy curve 
$$
S_{\text{Airy}}=\left(\bc,x=\frac{1}{2}z^2,\ y=z,\ B=\frac{dzdz'}{(z-z')^2}\right)
$$
produces correlators which are proven in \cite{EOrTop} to store intersection numbers
$$\omega_{g,n}^{\text{Airy}}=\sum_{\vec{m}\in\bz_+^n}\int_{\overline{\modm}_{g,n}}\prod_{i=1}^n\psi_i^{m_i}(2m_i+1)!!\frac{dz_i}{z_i^{2m_i+2}}
$$
and the coefficient is non-zero only for $\sum_{i=1}^n m_i=3g-3+n$.  Hence
\begin{align*}
Z^{\text{KW}}(\hbar,t_0,t_1,...)&=Z^{S_{\text{Airy}}}(\hbar,t_0,t_1,...)=\left.\exp\sum_{g,n}\frac{\hbar^{g-1}}{n!}\omega_{g,n}^{\text{Airy}}\right|_{\xi_m(z_i)=t_m}\\
&=\exp\sum_{g,n,\vec{m}}\frac{\hbar^{g-1}}{n!}\int_{\overline{\modm}_{g,n}}\prod_{i=1}^n\left(\psi_i^{m_i} t_{m_i}\right).
\end{align*}
\end{example}
\begin{example}   \label{specBes}
Topological recursion applied to the Bessel curve
$$
S_{\text{Bes}}=\left(\bc,x=\frac{1}{2}z^2,\ y=\frac{1}{z},\ B=\frac{dzdz'}{(z-z')^2}\right)
$$
produces correlators
$$
\omega_{g,n}^{\text{Bes}}=\sum_{\vec{k}\in\bz_+^n}b_g(m_1,...,m_n)\prod_{i=1}^n(2m_i+1)!!\frac{dz_i}{z_i^{2m_i+2}}
$$
where $b_g(m_1,...,m_n)\neq 0$ only for $\sum_{i=1}^n m_i=g-1$.    It is proven in \cite{DNoTop} that
$$Z^{\text{BGW}}(\hbar,t_0,t_1,...)=Z^{S_{\text{Bes}}}(\hbar,t_0,t_1,...)=\left.\exp\sum_{g,n}\frac{\hbar^{g-1}}{n!}\omega_{g,n}^{\text{Bes}}\right|_{\xi_m(z_i)=t_m}.
$$
\end{example}

For the next example define the following collection of differentials $\xi^{\alpha}_m(z,t)$ using $x=\frac12z^2-t\cdot\log z$ by 
\begin{align}   \label{chidiff}
\xi^0_{-1}(z,t)&=t^{-1/2}zdz,\quad \xi^1_{-1}(z,t)=dz,\\
\nonumber
 \xi^\alpha_{m+1}(z,t)&=-d\left(\frac{\xi^\sigma_{m}(z,t)}{dx(z)}\right),\ \sigma=0,1,\quad m=-1,0,1,2,...
\end{align}
For $m\geq 0$, these are linear combinations of the $V^i_m(p)$ defined in \eqref{Vdiff}.  
The following theorem uses the Chern polynomial
$$c\left(E_{g,n}^{\vec{\sigma}},t\right)=1+t\cdot c_1(E_{g,n}^{\vec{\sigma}})+t^2\cdot c_2(E_{g,n}^{\vec{\sigma}})+...\in H^*(\overline{\modm}_{g,n,\vec{\sigma}}^{\rm spin},\bq),\quad \vec{\sigma}\in\{0,1\}^{n}.
$$
\begin{theorem}[\cite{LPSZChi}] \label{LPSZ}
Topological recursion applied to the spectral curve 
\begin{equation}   \label{specfam}
x=\frac12z^2-t\cdot\log z,\quad y=z^{-1},\quad B=\frac{dzdz'}{(z-z')^2}
\end{equation}
produces correlators $\omega_{g,n}$ satisfying
\begin{equation}   \label{totalchern}
\omega_{g,n}(t,z_1,...,z_n)=\sum_{\vec{\sigma},\vec{m}} (-1)^nt^{2g-2+n} 2^{1-g}\hspace{-.5mm}\int_{\overline{\modm}_{g,n}}\hspace{-.5mm}p_*c\left(E_{g,n}^{\vec{\sigma}},\frac{2}{t}\right)\prod_{i=1}^n\psi_i^{m_i}\xi_{m_i}^{\sigma_i}(z_i,t).
\end{equation}
\end{theorem}
\begin{proof}[Proof of Theorem~\ref{LPSZ}]
Theorem~\ref{LPSZ} is a specialisation of a theorem in \cite{LPSZChi} which applies to a generalisation of the moduli space of spin curves to the moduli space of $r$-spin curves 
$$\overline{\modm}_{g,n}^{1/r}=\{(\cc,\theta,p_1,...,p_n,\phi)\mid \phi:\theta^r\stackrel{\cong}{\longrightarrow}\omega_{\cc}^{\text{log}}\}.
$$
For any $s\in\bz$, there is a line bundle $\ce$ on the universal $r$-spin curve over $\overline{\modm}_{g,n}^{1/r}$ with fibres given by the universal $r$th root of $\left(\omega_{\cc}^{\text{log}}\right)^s$.  Its derived push-forward $R^*\pi_*\ce$ defines a virtual bundle over $\overline{\modm}_{g,n}^{1/r}$.  For example, when $s=1$ and $r=1$, $-R^*\pi_*\ce$ is the Hodge bundle and when $s=-1$ and $r=2$, $-R^*\pi_*\ce=E_{g,n}$ coincides with Definition~\ref{obsbun} (where $\ce^\vee$ has now become $\ce$ due to $s=-1$.)  Note that \cite{LPSZChi} considers $r$th roots of $\displaystyle\left(\omega_C^{\text{log}}\right)^s\left(-\sum_{i=1}^n\sigma_ip_i\right)$ for $C$ the underlying coarse curve of $\cc$ with forgetful map $\rho:\cc\to C$.  The $r$th roots in \cite{LPSZChi} coincide with the push-forward $|\theta|=\rho_*\theta$ which is the locally free sheaf of $\bz_2$-invariant sections of the push-forward sheaf of $\theta$, and the isotropy representation at $p_i$ determines $\sigma_i$ as described in Section~\ref{existence}.  For $r=2$, i.e. $\theta^2\cong\omega_{\cc}^{\text{log}}$, at any point $p_i$ banded by $1/2$ the push-forward locally satisfies $|\theta|^2=\omega_C(2p_i)=\omega_C^{\text{log}}(p_i)$ hence $(|\theta|^\vee)^2=\left(\omega_C^{\text{log}}\right)^{-1}(-p_i)$ which corresponds to $\sigma_i=1$.  At any point $p_i$ banded by $0$ the push-forward does not change local degree and corresponds to $\sigma_i=0$.   

The Chern character of the virtual bundle $-R^*\pi_*\ce$ is given by Chiodo's generalisation of Mumford's formula for the Chern character of the Hodge bundle.   For $\sigma\in\{0,1,...,r-1\}$, let $j_\sigma:\text{Sing}_\sigma\to\overline{\modm}_{g,n}^{1/r}$ be the map from the singular set of the universal spin curve banded by $\sigma/r$ where now the local isotropy is $\bz_n$.  Let $B_m(x)$ be the $m$th Bernoulli polynomial.  Chiodo proved the following formula in \cite{ChiTow}:
\begin{align*}
\text{ch}(R^*\pi_*\ce)=\sum_{m\geq 0}&\Big(\frac{B_{m+1}(s/r)}{(m+1)!}\kappa_m-\sum_{i=1}^n\frac{B_{m+1}(m_i/r)}{(m+1)!}\psi_i^m\\
&+\frac{r}{2}\sum_{\sigma=0}^{r-1}\frac{B_{m+1}(\sigma/r)}{(m+1)!}(j_\sigma)_*\frac{\psi_+^m+(-1)^{m-1}\psi_-^m}{\psi_++\psi_-}\Big)
\end{align*}
The total Chern class of a virtual bundle $c(E-F):=c(E )/c(F)$ can be calculated from its Chern character and in this case is given by
$$c(-R^*\pi_*\ce)=\exp\left(\sum_{m=1}^\infty(-1)^m(m-1)!\text{ch}_m(R^*\pi_*\ce)\right).
$$
The components of $\overline{\modm}_{g,n}^{1/r}$ are given by $\overline{\modm}_{g,n,\vec{\sigma}}^{1/r}$ for $\vec{\sigma}\in\bz_r^n$.  The push-forward of the restriction of $c(-R^*\pi_*\ce)$ to a component is known as the Chiodo class 
$$C_{g,n}(r,s;\vec{\sigma}):=p_*c\left(-R^*\pi_*\ce|_{\overline{\modm}_{g,n,\vec{\sigma}}^{1/r}}\right)\in H^*(\overline{\modm}_{g,n},\bq)
$$
The sum of this push-forward over all components of $\overline{\modm}_{g,n}^{1/r}$ is expressed as a weighted sum over stable graphs in \cite{JPPZDou} which encodes a twisted loop group action as described in Section~\ref{givental}, with edge and vertex weights proven in \cite[Theorem 4.5]{LPSZChi} to exactly match the edge and vertex weights arising from the following spectral curve:
$$\hat{x}=z^r-\log z,\quad\hat{y}=\frac{r^{1+\frac{s}{r}}}{s}z^s,\quad B=\frac{dzdz'}{(z-z')^2}.
$$
In particular, the term $\exp\Big(-\displaystyle{\sum_m}\textstyle\frac{B_{m+1}(s/r)}{m(m+1)}\kappa_m\Big)$ which arises from the $\displaystyle{\sum_m}\textstyle\frac{B_{m+1}(s/r)}{(m+1)!}\kappa_m$ terms in Chiodo's formula exactly matches the local expansion of $dy$.  More precisely by \cite[Lemma 4.1]{LPSZChi}
\begin{equation}  \label{localy} 
\frac{1}{\sqrt{2\pi \hbar}}\int_{\Gamma_\alpha}dy(p)\cdot e^{\frac{x(\cp_\alpha)-x(p) }{ \hbar}}\sim dy(\cp_\alpha)\exp\left(-\displaystyle{\sum_m}\textstyle\frac{B_{m+1}(s/r)}{m(m+1)}(-\hbar)^m\right)
\end{equation}
where $\sim$ means the asymptotic expansion in the limit $\hbar\to 0$.  

Hence topological recursion applied to this spectral curve produces correlators with expansion in terms of the local coordinate $e^{-\hat{x}_i}=e^{-\hat{x}(z_i)}=z_ie^{-z_i^r}$ around $z_i=0$,
\begin{equation}  \label{TRchiodo} 
\hat{\omega}_{g,n}(z_1,...,z_n)\sim\sum_{\vec{k}\in\bz_+^n}\prod_{i=1}^nc(k_i)r^{\frac{(k_i)_r}{r}}d(e^{-k_i\hat{x}_i})\int_{\overline{\modm}_{g,n}}\frac{C_{g,n}(r,s;(-\vec{k})_r)}{\prod_{i=1}^n(1-\frac{k_i}{r}\psi_i)}
\end{equation}
where $\sim$ means expansion in a local coordinate, $(-\vec{k})_r\in\{0,...,r-1\}^n$ the residue class of $-\vec{k}$ modulo $r$, and
$$ c(k)=\frac{k^{\lfloor\frac{k}{r}\rfloor}}{\lfloor\frac{k}{r}\rfloor!}.
$$
We have used $\hat{x}=z^r-\log z$ and $y=\frac{r^{1+\frac{s}{r}}}{s}z^s$ here, rather than $\hat{x}=-z^r+\log z$ and $y=z^s$ used in \cite{LPSZChi}, because the convention for the kernel $K(p_1,p)$ used here differs by sign from \cite{LPSZChi}, and also to remove a factor of $\left(\frac{r^{1+\frac{s}{r}}}{s}\right)^{2-2g-n}$  from the correlators.   Chiodo's formula and the asymptotic expansion \eqref{localy} are true for any $s\in\bz$, hence \eqref{TRchiodo} holds for any $s\in\bz$, although it is stated only for $s\geq 0$ in \cite{LPSZChi}. 

In \cite{LPSZChi} $(-\vec{k})_r\in\{1,...,r\}^n$, however replacing $k_i=r$ by $k_i=0$ leaves the Chiodo class invariant since it does not change the component, rather it twists the universal bundle $\ce$ over the component resulting in adding a direct summand of a trivial bundle to the virtual bundle $-R^*\pi_*\ce$ which does not affect the total Chern class.  The invariance of the total Chern class, or equivalently the positive degree terms of the Chern character, can also be seen in Chiodo's formula via properties of the Bernoulli polynomials.

We will use \eqref{TRchiodo} in the case $r=2$.  Define
$$\hat{\xi}^0_{-1}=2zdz,\quad \hat{\xi}^1_{-1}=dz,\quad \hat{\xi}^\sigma_{m}(z)=-d\left(\frac{\hat{\xi}^\sigma_{m-1}(z)}{d\hat{x}(z)}\right),\quad \sigma\in\{0,1\},\ m\in\{0,1,2,...\}
$$
which have local expansion at $z=0$ given by
$$\hat{\xi}^\sigma_{m}(z)\sim \hspace{-3mm}\mathop{\sum_{k\in\bz_+}}_{^{k\equiv\sigma(\text{mod\ }2)}}\hspace{-2mm}k^{m}c(k)d(e^{-k\hat{x}}).
$$
Each $\psi_i$ in the denominator of the right hand side of \eqref{TRchiodo} produces monomials $(\frac{1}{2}k_i\psi_i)^{m_i}$, hence \eqref{TRchiodo} with $r=2$ becomes
$$\hat{\omega}_{g,n}(z_1,...,z_n)=\sum_{\vec{\sigma},\vec{m}}\int_{\overline{\modm}_{g,n}}C_{g,n}(2,s;\vec{\sigma})\prod_{i=1}^n\psi_i^{m_i}\hat{\xi}_{m_i}^{\sigma_i}(z_i)2^{\frac12\sigma_i-m_i}.
$$

Change $(\hat{x},\hat{y})\mapsto(x,y)$ by
$$x=t\hat{x}\left(\frac{z}{\sqrt{2t}}\right)-\frac12t\log(2t)=\frac12z^2-t\cdot\log z,\quad y=\frac{s}{2}t^{\frac{s}{2}}\hat{y}\left(\frac{z}{\sqrt{2t}}\right)=z^s.
$$
The differentials defined in \eqref{chidiff} using $x$ are given by
$$\xi^\sigma_{m}(z,t)=t^{-m-\frac12}2^{\frac{\sigma}{2}}\hat{\xi}^\sigma_{m}\left(\frac{z}{\sqrt{2t}}\right).
$$
Hence
\begin{align*}
\omega_{g,n}(t,z_1,...,&z_n)=\left(\frac{s}{2}t^{\frac{s}{2}+1}\right)^{2-2g-n}\hat{\omega}_{g,n}\left(\frac{z_1}{\sqrt{2t}},...,\frac{z_n}{\sqrt{2t}}\right)\\
&=\left(\frac{s}{2}t^{\frac{s}{2}+1}\right)^{2-2g-n}\sum_{\vec{\sigma},\vec{m}}\int_{\overline{\modm}_{g,n}}\hspace{-2mm}C_{g,n}(2,s;\vec{\sigma})\prod_{i=1}^n\psi_i^{m_i}\hat{\xi}_{m_i}^{\sigma_i}\left(\frac{z_i}{\sqrt{2t}}\right)2^{\frac12\sigma_i-m_i}\\
&=\left(\frac{s}{2}t^{\frac{s}{2}+1}\right)^{2-2g-n}\sum_{\vec{\sigma},\vec{m}}\int_{\overline{\modm}_{g,n}}\hspace{-2mm}C_{g,n}(2,s;\vec{\sigma})\prod_{i=1}^nt^{m_i+\frac12}\psi_i^{m_i}\xi_{m_i}^{\sigma_i}(z_i)2^{-m_i}\\
&=\left(\frac{s}{2}t^{\frac{s}{2}+1}\right)^{2-2g-n}t^{\frac{n}{2}}\sum_{\vec{\sigma},\vec{m}}\int_{\overline{\modm}_{g,n}}\hspace{-2mm}C_{g,n}(2,s;\vec{\sigma})\prod_{i=1}^n\left(\frac{t}{2}\right)^{m_i}\psi_i^{m_i}\xi_{m_i}^{\sigma_i}(z_i)\\
&=\sum_{\vec{\sigma},\vec{m}}t^{\frac12(1-s)(2g-2+n)}2^{1-g}s^{2-2g-n}\int_{\overline{\modm}_{g,n}}\hspace{-2mm}C_{g,n}(2,s;\vec{\sigma},\frac{2}{t})\prod_{i=1}^n\psi_i^{m_i}\xi_{m_i}^{\sigma_i}(z_i,t)
\end{align*}
where the last equality uses $(t/2)^{\sum m_i}=(t/2)^{3g-3+n-\deg}$ for the degree operator $\deg c_k(E_{g,n}^{\vec{\sigma}})=k$ then $(t/2)^{-\deg}$ is absorbed into the Chern polynomial.  Set $s=-1$ to get the desired result.
\end{proof}
The classes $\Theta_{g,n}$ arise in the limit
\[\lim_{t\to 0}\omega_{g,n}(t,z_1,...,z_n)=\sum_{\vec{m}}\int_{\overline{\modm}_{g,n}}\hspace{-3mm}\Theta_{g,n}\prod_{i=1}^n\psi_i^{m_i}\xi_{m_i}(z)
\]
for  $\xi_m(z)=(2m+1)!!z^{-(2m+2)}dz.$  We explain the relationship of this limit with Conjecture~\ref{kdvconj} in Proposition~\ref{TRlim}.

\subsubsection{$A_2$ singularity}  \label{A2spec}
In this section we calculate the spectral curves of the CohFT $\Omega^{A_2}$ and $(\Omega^{A_2})^\Theta$.  We begin with a general result relating the spectral curve of any semisimple CohFT $\Omega$ with the spectral curve of $\Omega^{\text{BGW}}$.
\begin{proposition}   \label{specOmegaBGW}
Given a semisimple CohFT $\Omega$ with partition function $Z_\Omega(\hbar,\{t^{\alpha}_k\})$ encoded by the spectral curve 
\[S=(C,x,y,B)\] 
via Theorem~\ref{DOSS} then $Z^{\text{BGW}}_\Omega(\hbar,\{t^{\alpha}_k\})$is encoded by the spectral curve 
\[\hat{S}=(C,x,\hat{y}=\frac{dy}{dx},B).\]
\end{proposition}
\begin{proof} 
Note that the spectral curves $S$ and $\hat{S}$ share the same $(C,x,B)$ and hence produce the same operator $\hat{R}(z)$ used in the construction of both $Z_\Omega$ and $Z^{\text{BGW}}_\Omega$.

Proposition~\ref{givtheta} shows that a shift in the translation operator $T(z)\mapsto\frac{1}{z}T(z)$ combined with replacing each copy of $Z^{\text{KW}}(\hbar,\{t_k\})$ in \eqref{givpart} by a copy of $Z^\Theta(\hbar,\{t_k\})$ produces the partition function of $\Omega^\Theta$.  It relied upon the homogeneity property \eqref{dilaton} satisfied by $Z^\Theta(\hbar,\{t_k\})$.  But $Z^{\text{BGW}}(\hbar,\{t_k\})$ also satisfies \eqref{dilaton} hence an identical argument to that in Proposition~\ref{th:dilaton} proves that
for a semisimple CohFT $\Omega$, the partition function $Z^{\text{BGW}}_{\Omega}(\hbar,\{t^{\alpha}_k\})$ is obtained by replacing each copy of $Z^{\text{KW}}(\hbar,\{t_k\})$ in \eqref{givpart} by a copy of $Z^{\text{BGW}}(\hbar,\{t_k\})$ and replacing the translation operator by $T(z)\mapsto\frac{1}{z}T(z)$.

Given an irregular spectral curve, it is proven in \cite{CNoTop} that its partition function is obtained from \eqref{givpart} with translation operator given by \eqref{ytoT}
\[T(z)^\alpha=\left\{\begin{array}{ll}z\left(\un^\alpha- \frac{1}{\sqrt{2\pi z}}\int_{\Gamma_\alpha}dy(p)\cdot e^{\frac{x(\cp_\alpha)-x(p) }{ z}}\right),&\cp_\alpha\text{ regular}\\
\un^\alpha- \frac{1}{\sqrt{2\pi z}}\int_{\Gamma_\alpha}y(p)dx(p)\cdot e^{\frac{x(\cp_\alpha)-x(p) }{z}},&\cp_\alpha\text{ irregular}\end{array}\right. \]
Given a semisimple CohFT $\Omega$ encoded by the regular spectral curve 
$S=(C,x,y,B)$, define $\hat{y}=\frac{dy}{dx}$.  Then we see that since $dy=\hat{y}dx$ the translation operator shifts by $T(z)^\alpha\mapsto\frac{1}{z}T(z)^\alpha$ which proves that $\Omega^{\text{BGW}}$ is encoded by the spectral curve 
$\hat{S}=(C,x,\hat{y}=\frac{dy}{dx},B)$.
\end{proof}

Define the spectral curves  
\begin{equation}  \label{BGWA2}
\begin{array}{ccc}
S_{A_2}&=&\left(\bc,x=z^3-3z,\ y=z\sqrt{-3},\ B=\frac{dzdz'}{(z-z')^2}\right)   \\
&&\\
S^{\text{BGW}}_{A_2}&=&\left(\bc,x=z^3-3z,\ \hat{y}=\frac{\sqrt{-3}}{3z^2-3},\ B=\frac{dzdz'}{(z-z')^2}\right).\end{array}  
\end{equation}
The partition functions associated to $S=S_{A_2}$ defined in \eqref{A2sing} and $S=S^{\Theta}_{A_2}$ are built out of correlators $\omega^S_{g,n}$ by
$$Z^S(\hbar,\{t^{\alpha}_k\})=\left.\exp\sum_{g,n}\frac{\hbar^{g-1}}{n!}\omega^S_{g,n}\right|_{\xi^{\alpha}_k(z_i)=t^{\alpha}_k}
$$  
using the differentials $\xi^{\alpha}_k(z)$ defined on $\bc$ by 
\begin{equation}   \label{flatdiff}
\xi^{\alpha}_0=\frac{dz}{(1-z)^2}-\frac{(-1)^{\alpha}dz}{(1+z)^2},\quad \xi^\alpha_{k+1}(p)=d\left(\frac{\xi^\alpha_k(p)}{dx(p)}\right),\ \alpha\in\{1,2\},\  k\in\bn.
\end{equation}  
These are linear combinations of the $V^i_k(p)$ defined in \eqref{Vdiff} with $x=z^3-3z$. The $V^i_k(p)$ correspond to normalised canonical coordinates while the $\xi^{\alpha}_k(p)$ correspond to flat coordinates.  We have
\[ Z_{\Omega^{A_2}}=Z^{S_{A_2}},\qquad Z_{(\Omega^{A_2})^{\Theta}}=Z^{S^{\Theta}_{A_2}}.
\]
The equality $Z_{A_2}=Z^{S_{A_2}}$ was proven in \cite{DNOPSDub} hence $Z_{(\Omega^{A_2})^{\Theta}}=Z^{S^{\Theta}_{A_2}}$ by Proposition~\ref{specOmegaBGW}.
We verify this by giving the local expansions of $B$ and $\hat{y}$ for $S_{A_2}$ which helps to deal with different normalisations in the references.   Choose a local coordinate $t$ around $z=-1=\cp_1$ so that $x(t)=\frac12 t^2+2$.  Then
\begin{align*}
B(\cp_1,t)&=\frac{-i}{\sqrt{6}}\frac{dz}{(z+1)^2}=dt\left(t^{-2}-\frac{1}{144}+\frac{35}{41472}t^2+...+{\rm odd\ terms}\right)\\
B(\cp_2,t)&=\frac{1}{\sqrt{6}}\frac{dz}{(z-1)^2}=dt\left(-\frac{i}{24}+\frac{35i}{3456}t^2+...+{\rm odd\ terms}\right)
\end{align*}
Around $z=1=\cp_2$ the local expansions of $B(\cp_\alpha,z)$ are the same as those above, up to sign.
The odd terms are annihilated by the Laplace transform, and we get
\begin{align*}R^{-1}(z)^\alpha_\alpha &= -\frac{\sqrt{z}}{\sqrt{2\pi}}\int_{\Gamma_\alpha} B(\cp_\alpha,t)\cdot e^{\frac{-\frac12t^2}{ z}}=1-(-1)^\alpha\frac{1}{144}z-\frac{35}{41472}z^2+...\\
R^{-1}(z)^\alpha_{3-\alpha} &= -\frac{\sqrt{z}}{\sqrt{2\pi}}\int_{\Gamma_\alpha} B(\cp_{3-\alpha},t)\cdot e^{\frac{-\frac12t^2}{ z}}=\frac{i}{24}z+(-1)^\alpha\frac{35i}{3456}z^2+...
\end{align*}
Hence $R^{-1}(z)=I-R_1z+(R_1^2-R_2)z^2+...=I-R_1^Tz+R_2^Tz^2+...$ gives 
$$ R_1=\frac{1}{144}\left(\begin{array}{cc}-1&-6i\\-6i&1\end{array}\right),\quad
R_2=\frac{35}{41472}\left(\begin{array}{cc}-1&12i\\-12i&-1\end{array}\right)
$$
which determines all other $R_k$ via \eqref{teleman} and agrees with \eqref{A2RT} for $\Omega^{A_2}$.

The topological field theory is defined by $\{dy(\cp_\alpha)\}$, $i=1,2$.  The translation operator $T(z)$ is determined by the (Laplace transform of the) local expansion of $y$ given by \eqref{ytoT}.  Moreover, $\Omega^{A_2}$ has flat identity, so in this case the odd expansions of $dy$ is determined by $R^{-1}(z)\un$ via \eqref{ytoTunit}, hence uniquely determined by the terms $dy(\cp_\alpha)$, $\alpha=1,2$.  This is visible on the spectral curve by the fact that the poles of $dy$ are dominated by the poles of $dx$, i.e. $dy/dx$ has poles only at the zeros $\cp_1$ and $\cp_2$ of $dx$, hence by the Cauchy formula $dy$ satisfies 
\begin{equation} \label{DOSStest} 
d\left(\frac{dy}{dx}(p)\right)=-\sum_{\alpha=1}^N \Res_{p'=\cp_\alpha}\frac{dy}{dx}(p')B(p',p) \end{equation}
which is proven in \cite{DNOPSDub} to imply \eqref{ytoTunit}.
Thus, it remains to show that $y$ defines the correct topological field theory, representing $\un$ in normalised canonical coordinates.  The local expansion of $dy=\sqrt{-3} dz$ around $\cp_1=-1$  in the local coordinate $x(t)=\frac12 t^2+2$ is:
$$ dy=\sqrt{-3} dz=\left(\frac{1}{\sqrt{2} }-\frac{5}{144\sqrt{2}}t^2+\frac{385}{124416\sqrt{2}}t^4+...+{\rm odd\ terms}\right)dt
$$
and around $\cp_2=1$ replace $t$ by $it$.
Hence the Laplace transform is:
\begin{align*}
\left\{\frac{1}{\sqrt{2\pi z}}\int_{\Gamma_\alpha}dy(p)\cdot e^{\frac{(x(\cp_k)-x(p))}{z}}\right\}&=R^{-1}(z)\un\\
=\frac{1}{\sqrt{2}}\left(\begin{array}{c}1\\i\end{array}\right)&+\frac{5}{144\sqrt{2}}\left(\begin{array}{c}-1\\i\end{array}\right)z+\frac{385}{41472\sqrt{2}}\left(\begin{array}{c}1\\i\end{array}\right)z^2+...
\end{align*}
Note that $dy(\cp_1)=\frac{1}{\sqrt{2}}=\sqrt{\un^1}$ and $dy(\cp_2)=\frac{i}{\sqrt{2}}=\sqrt{\un^2}$ 
gives the unit $\un$, hence the TFT.  
Thus $S_{A_2}\mapsto (R(z),T(z),\un)$ for $\Omega^{A_2}$ as required.

\section{Progress towards a proof of conjecture \ref{kdvconj}}   \label{conj}
A consequence of the homogeneity property \eqref{dilaton} satisfied by both partition functions $Z^{\Theta}(\hbar,t_0,t_1,...)$ and $Z^{\text{BGW}}(\hbar,t_0,t_1,...)$ is that for $g>1$ the coefficient of $\hbar^{g-1}$ of the logarithm of the partition function, i.e. its genus $g$ part, is a {\em finite} sum of rational functions.  They are both of the form:
\[
\log Z(\hbar,t_0,t_1,...)=-\frac{1}{8}\log(1-t_0)+\sum_{g=2}^\infty\hbar^{g-1}\sum_{\mu\vdash g-1}\frac{c_\mu t_\mu}{(1-t_0)^{2g-2+n}}
\]
where $t_\mu:=\prod t_{\mu_i}$ for a partition $\mu=(\mu_1,...,\mu_n)$.  Hence for each $g$ one needs only match the finite set of coefficients $c_\mu$, parametrised by partitions $\mu$ of $g-1$, of $\log Z^{\Theta}(\hbar,t_0,t_1,...)$ with those of $\log Z^{\text{BGW}}(\hbar,t_0,t_1,...)$, to determine equality.
 
The initial value $\int_{\overline{\modm}_{1,1}}\Theta_{1,1}=\frac18$ and \eqref{dilaton} produces all genus 1 terms of $\log Z^{\Theta}$, and the calculation
$\int_{\overline{\modm}_{2,1}}\Theta_{2,1}\cdot\psi_1=\frac{3}{128}$ from Example~\ref{gen2rel} together with \eqref{dilaton} produces all genus 2 terms giving:
\[ 
\log Z^{\Theta}=-\frac{1}{8}\log(1-t_0)+\hbar\frac{3}{128}\frac{t_1}{(1-t_0)^3}+O(\hbar^2).
\]
Further calculations, such as the genus 3 calculation in Appendix~\ref{sec:calc} and calculations up to $g=7$ and $n=6$ using admcycles \cite{DSZadm} prove 
\begin{equation}  \label{lowgtheta} 
\log Z^{\Theta}(\hbar,t_0,t_1,...)=\log Z^{\text{BGW}}(\hbar,t_0,t_1,...)+O(\hbar^8).
\end{equation}

Conjecture~\ref{kdvconj} is reduced to a purely combinatorial or analytic problem in the following proposition.  Recall the spectral curve \eqref{specfam} given by
\[
x=\frac12z^2-t\cdot\log z,\quad y=z^{-1},\quad B=\frac{dzdz'}{(z-z')^2}
\]
with correlators $\omega_{g,n}(t,z_1,...,z_n)$.

\begin{proposition} \label{TRlim}
Conjecture \ref{kdvconj} is equivalent to
\begin{equation}  \label{TRlimit}
\lim_{t\to 0}\omega_{g,n}(t,z_1,...,z_n)=\omega_{g,n}^{\text{Bes}}(z_1,...,z_n).
\end{equation}
\end{proposition}
\begin{proof}
By Theorem~\ref{LPSZ}
\[\omega_{g,n}(t,z_1,...,z_n)=\sum_{\vec{\sigma},\vec{m}} (-1)^nt^{2g-2+n} 2^{1-g}\hspace{-.5mm}\int_{\overline{\modm}_{g,n}}\hspace{-.5mm}p_*c\left(E_{g,n}^{\vec{\sigma}},\frac{2}{t}\right)\prod_{i=1}^n\psi_i^{m_i}\xi_{m_i}^{\sigma_i}(z_i,t)\]
which is regular in $t$ since
$$\text{rank}\ E_{g,n}^{\vec{\sigma}}=2g-2+\tfrac12(n+|\vec{\sigma}|)
$$
so the Chern polynomial has degree at most $2g-2+n$ in $t^{-1}$.  Hence for $|\vec{\sigma}|=n$
$$\lim_{t\to 0}(-1)^nt^{2g-2+n}2^{1-g}p_*c\left(E_{g,n}^{\vec{\sigma}},\frac{2}{t}\right)=
(-1)^n2^{g-1+n}p_*c_{2g-2+n}\left(E_{g,n}^{\vec{\sigma}}\right)=\Theta_{g,n}
$$
while for $|\vec{\sigma}|<n$, $\text{rank}\ E_{g,n}^{\vec{\sigma}}<2g-2+n$ so
$$\lim_{t\to 0}(-1)^nt^{2g-2+n}2^{1-g}p_*c\left(E_{g,n}^{\vec{\sigma}},\frac{2}{t}\right)=0.
$$
Thus the $t\to0$ limit  exists to give
$$\lim_{t\to 0}\sum_{\vec{\sigma},\vec{m}}\int_{\overline{\modm}_{g,n}}\hspace{-1mm}(-1)^nt^{2g-2+n}2^{1-g}p_*c\left(E_{g,n}^{\vec{\sigma}},\frac{2}{t}\right)\prod_{i=1}^n\psi_i^{m_i}\xi_{m_i}^{\sigma_i}=\sum_{\vec{m}}\int_{\overline{\modm}_{g,n}}\hspace{-3mm}\Theta_{g,n}\prod_{i=1}^n\psi_i^{m_i}\xi_{m_i}(z)
$$
for 
$$\xi_m(z)=\lim_{t\to 0}\xi^1_m(z,t)=(2m+1)!!z^{-(2m+2)}dz.$$
Also $\displaystyle\lim_{t\to 0}\xi^0_m(z,t)=0$ for $m\geq 0$. The $t\to0$ limit of the spectral curve \eqref{specfam} gives the Bessel spectral curve of Example~\ref{specBes} with correlators proven in \cite{DNoTop} to be given by
\[\omega_{g,n}^{\text{Bes}}(z_1,...,z_n)=\sum_{\vec{m}}\frac{\partial^nF^{\text{BGW}}(\hbar,\{t_k\})}{\partial t_{m_1}...\partial t_{m_n}}\prod_{i=1}^n\xi_{m_i}(z).
\]
Hence the conjectured limit \eqref{TRlimit} yields
\[\sum_{\vec{m}}\int_{\overline{\modm}_{g,n}}\hspace{-1mm}\Theta_{g,n}\prod_{i=1}^n\psi_i^{m_i}\xi_{m_i}(z)=\sum_{\vec{m}}\frac{\partial^nF^{\text{BGW}}(\hbar,\{t_k\})}{\partial t_{m_1}...\partial t_{m_n}}\prod_{i=1}^n\xi_{m_i}(z)\]
which is equivalent to Conjecture \ref{kdvconj}.
\end{proof}
The subtlety of the limit \eqref{TRlimit}, which is known up to $g=7$ for all $n$ by the verification of Conjecture \ref{kdvconj} in these cases, can be seen as follows.  
The correlators are regular in $t$, for example
\[\omega_{0,3}(t,z_1,z_2,z_3)=O(t)\quad\Rightarrow\quad\lim_{t\to 0}\omega_{0,3}(t,z_1,z_2,z_3)=0.
\]
However, the coefficients in the recursion can be irregular in $t$ i.e. blow up as $t\to 0$.  For example, in the following calculation of $\omega_{1,2}(t,z_1,z_2)$ we introduce the parameter $a$ to keep track of the contribution of $\omega_{0,3}(t,z_1,z_2,z_3)$.  We can set $a=1$ at the end.
\begin{align*}
\omega_{1,2}(t,z_1,z_2)=\sum_{dx(\alpha)=0}\Res_{z=\alpha}&K(z_1,z) \bigg[a\cdot\omega_{0,3}(t,z,\sigma_\alpha(z),z_2)\\
&  +\omega_{0,2}(z,z_2) \, \omega_{1,1}(t,\sigma_\alpha(z))
+  \omega_{0,2}(\sigma_\alpha(z),z_2) \, \omega_{1,1}(t,z) \bigg]
\end{align*}
\[\lim_{t\to 0}\omega_{1,2}(t,z_1,z_2)=\left(\frac{74a+61}{1080}\right)\frac{dz_1dz_2}{z_1^2z_2^2}
\]
This gives the expected limit of $\omega_{1,2}^{\text{Bes}}(z_1,z_2)$ when $a=1$, and shows the dependence of $\displaystyle\lim_{t\to 0}\omega_{1,2}(t,z_1,z_2)$ on $\omega_{0,3}(t,z_1,z_2,z_3)$ due to coefficients in the recursion which are irregular in $t$.

\subsection{Pixton relations}  \label{pixrel}

A collection of relations in the tautological ring $RH^*(\overline{\modm}_{g,n})$ was conjectured by Pixton and proven in \cite{PPZRel} using the CohFT $\Omega^{A_2}$.   Such tautological relations can be used to produce topological recursion relations for CohFTs such as Gromov-Witten invariants.  Similarly, the intersections of $\Theta_{g,n}$ with Pixton's relations produce topological recursion relations satisfied by the intersection numbers $\int_{\overline{\modm}_{g,n}}\Theta_{g,n}\prod_{i=1}^n\psi_i^{m_i}$.

The key idea behind the proof of Pixton's relations in \cite{PPZRel} is a degree bound on the cohomology classes \[\deg\Omega^{A_2}_{g,n}\leq\frac13(g-1+n )<3g-3+n\]
combined with Givental's construction of $\Omega^{A_2}_{g,n}$ in Definition~\ref{givact} from the triple
$(R(z),T(z),\un)\in L^{(2)}GL(N,\bc)\times z^2\bc^N[[z]]\times\bc^N$ obtained from the Frobenius manifold structure on the versal deformation space of the $A_2$ singularity, \ref{A2sing}.  Givental's construction produces $\Omega^{A_2}_{g,n}$, although it does not know about the degree bound and produces classes in the degrees where $\Omega^{A_2}_{g,n}$ vanishes.  This leads to sums of tautological classes representing the zero class, i.e. relations given by the degree $d>\frac13(g-1+n )$ part of the sum over stable graphs in \eqref{graphsum} of the form
\[ \Omega_{g,n}^{A_2}=\sum_{\Gamma\in G_{g,n}}\frac{1}{|{\rm Aut}(\Gamma)|}(\phi_{\Gamma})_*\omega^{R,T,\un}_{\Gamma}.
\]
Since $\Omega^{A_2}$ has flat identity, the push-forward classes in \eqref{graphsum} produce $\kappa$  polynomials, hence only graphs without dilaton leaves in the sum are required and the classes $\omega^{R,T,\un}_{\Gamma}$ consist of products of $\psi$ and $\kappa$ classes associated to each vertex of $\Gamma$.  The main result of \cite{PPZRel} is the construction of elements $R^d_{g,A}\in S_{g,n}$ for $A=(a_1,...,a_n)$, $a_\alpha\in\{0,1\}$ satisfying $q(R^d_{g,A})=0$ which push forward to tautological relations in $H^{2d}(\overline{\modm}_{g,n},\bq)$.   They are defined by $R^d_{g,A}=$ degree $d$ part of $\Omega_{g,n}^{A_2}(v_A)$ for a basis $\{v_0,v_1\}$.
The element $R^1_2\in H^{2}(\overline{\modm}_{2},\bq)$ is given in Example~\ref{gen2rel}.

When $n\leq g-1$ and $g>1$ we have $d=g-1>\frac13(g-1+n )$ hence there exist non-trivial relations $R^{g-1}_{g,A}$.  This produces the following sum over graphs 
\[ \Theta_{g,n}\cdot R^{g-1}_{g,A}=0
\]
which defines a relation, for each $A$, between intersection numbers of $\psi$ classes with $\Theta_{g,n}$, i.e. coefficients of $Z^\Theta(\hbar,\{t_k\})$.  This uses $\Theta_{g,n}\cdot(\phi_{\Gamma})_*=(\phi_{\Gamma})_*\Theta_{\Gamma}$ together with Remark~\ref{removekappa} to replace $\kappa$ classes by $\psi$ classes.  We saw this in Example~\ref{gen2rel}  arising from the genus two Pixton relation
\begin{equation}  \label{relg2}
\int_{\overline{\modm}_{2,1}}\Theta_{2,1}\cdot\psi_1-\frac{7}{10}\cdot\int_{\overline{\modm}_{1,1}}\Theta_{1,1}\cdot\int_{\overline{\modm}_{1,1}}\Theta_{1,1}-\frac{1}{10}\cdot\int_{\overline{\modm}_{1,2}}\Theta_{1,2}=0
\end{equation}
which determines $\int_{\overline{\modm}_{2,1}}\Theta_{2}\cdot\psi_1^2$ from $\int_{\overline{\modm}_{1,1}}\Theta_{1,1}$ and $\int_{\overline{\modm}_{1,2}}\Theta_{1,2}$.  Similarly, Appendix~\ref{sec:calc} uses genus three relations to deduce $\int_{\overline{\modm}_{3,2}}\Theta_{3,2}\cdot\psi_1^2$ and $\int_{\overline{\modm}_{3,2}}\Theta_{3,2}\cdot\psi_1\psi_2$ from lower genus coefficients of $Z^{\Theta}(\hbar,t_0,t_1,...)$.  

The following theorem proves that the coefficients of $Z^{\text{BGW}}(\hbar,t_0,t_1,...)$ also satisfy \eqref{relg2}, and more generally an infinite set of relations satisfied by coefficients of $Z^{\Theta}(\hbar,t_0,t_1,...)$ arising from Pixton relations.  
\begin{theorem}   \label{infrel}
Pixton relations produce infinitely many non-trivial relations satisfied by the coefficients of both $Z^{\Theta}(\hbar,t_0,t_1,...)$ and $Z^{\text{BGW}}(\hbar,t_0,t_1,...)$.
\end{theorem}  
\begin{proof} 
For each $g>1$, $n$ and $\lfloor\frac{n+1}{2}\rfloor$ possible $A\in\{0,1\}^n$ (due to symmetry and vanishing of half for parity reasons), $R^{g-1}_{g,A}=0$ defines a non-trivial Pixton relation.  For each of these choices of $g$, $n$ and $A$, due to restriction and pull-back properties of $\Theta_{g,n}$ as explained above,  $\Theta_{g,n}\cdot R^{g-1}_{g,A}=0$ defines a relation between coefficients of $Z^\Theta(\hbar,\{t_k\})$, such as  \eqref{relg2}.

The main aim of the proof is to prove that the corresponding coefficients of $Z^{\text{BGW}}(\hbar,\{t_k\})$ also satisfy this infinite set of relations.  To do this, we study the partition function $Z^{\text{BGW}}_{\Omega^{A_2}}$, defined in Definition~\ref{BGWCohFTpart} via the spectral curve $S^{\text{BGW}}_{A_2}$ defined in \eqref{BGWA2}.  The relations between coefficients of $Z^{\text{BGW}}(\hbar,\{t_k\})$ will be stored in the spectral curve.    This will produce identical relations satisfied by both the coefficients of $Z^{\text{BGW}}$ and $Z^{\Theta}$.  To summarise, we have vanishing of certain coefficients of $Z^{\Theta}_{A_2}(\hbar,\{t^{\alpha}_k\})$ due to the cohomological viewpoint shown in the upper row in Figure~\ref{fig:cons}, and vanishing of corresponding coefficients of $Z^{\text{BGW}}_{\Omega^{A_2}}(\hbar,\{t^{\alpha}_k\})$ due to Givental's construction neatly encoded by topological recursion shown in the lower row in Figure~\ref{fig:cons}.

Pixton relations induce relations between intersection numbers of $\psi$ and $\kappa$ classes or $\psi$ classes alone, i.e. coefficients of $Z^{\text{KW}}(\hbar,\{t_k\})$.  These relations are realised by unexpected vanishing of  coefficients of the partition function $Z_{A_2}(\hbar,\{t^{\alpha}_k\})$.  Similarly, unexpected vanishing of coefficients of the partition function $Z^{\text{BGW}}_{A_2}(\hbar,\{t^{\alpha}_k\})$ correspond to relations between coefficients of $Z^{\text{BGW}}(\hbar,\{t_k\})$.

The coefficients of $\log Z^{\text{BGW}}_{A_2}(\hbar,\{t^{\alpha}_k\})$ are obtained from the correlators $\omega^{\text{BGW},A_2}_{g,n}$ of $S^{\text{BGW}}_{A_2}$ by
\begin{equation}  \label{insert}
\left.\frac{\partial^n}{\partial t^{\alpha_1}_{k_1}...\partial t^{\alpha_n}_{k_n}}(F^{\text{BGW}}_{A_2})_g(\{t^{\alpha}_k\})\right|_{t^{\alpha}_k=0}=\Res_{z_1=\infty}...\Res_{z_n=\infty}\prod_{i=1}^n p_{\alpha_i,k_i}(z_i)\omega^{\text{BGW},A_2}_{g,n}(z_1,...,z_n)
\end{equation}
for polynomials $p_{\alpha,k}(z)=\sqrt{-3}\frac{(-1)^{\alpha}}{\alpha}z^{3k+\alpha}+\text{lower order terms}$ for $\alpha\in\{1,2\}$ and $k\in\bn$ chosen so that the residues are dual to the differentials $\xi^{\alpha}_k$ defined in \eqref{flatdiff}.  The lower order terms (and the top coefficient) will not be important here because we will only consider vanishing of \eqref{insert} arising from high enough order vanishing of $\omega^{A_2}_{g,n}(z_1,...,z_n)$ at $z_i=\infty$ so that the integrand in \eqref{insert} is holomorphic at $z_i=\infty$.  Equation~\eqref{insert} is a special case of the more general phenomena, proven in \cite{DNOPSPri}, that periods of $\omega_{g,n}$ are dual to insertions of vectors in a CohFT.    Thus we have shown that {\em relations between coefficients of $Z^{\text{BGW}}(\hbar,\{t_k\})$ induced from Pixton relations are detected by high order vanishing of $\omega_{g,n}^{\text{BGW},A_2}(z_1,...,z_n)$ at $z_i=\infty$}.
The same is true for high order vanishing $\omega_{g,n}^{A_2}(z_1,...,z_n)$ at $z_i=\infty$ which is shown in the following calculation:
$$\omega^{A_2}_{2,1}(z) = \frac{35}{243}\frac{z(11z^4+14z^2+2)}{(z^2-1)^{10}}dz\quad\Rightarrow \Res_{z=\infty}z^m\omega^{A_2}_{2,1}(z)=0,\  m\in\{0,1,...,12\}
$$
Hence \eqref{insert} vanishes for $k_1=0,1,2,3$ and $\alpha_1\equiv k_1(\hspace{-2mm}\mod 2)$ which gives the following relations between intersection numbers, or coefficients of $Z^{\text{KW}}(\hbar,\{t_k\})$,
\begin{equation}  \label{pixrelTR}
\int_{\overline{\modm}_{2,1}}R^{d}_{2,\bar{d}}\psi_1^{4-d}=0,\quad d=1,2,3,4
\end{equation}
where $R^{d}_{2,\bar{d}}$ is a non-trivial Pixton relation, for $\bar{d}\equiv d(\hspace{-2mm}\mod 2)$, between cohomology classes in $H^{2d}(\overline{\modm}_{2,1})$ proven in \cite{PPZRel}, such as $R^{2}_{2,0}=\psi_1^2+$ boundary terms $=0$.  
\begin{lemma}  \label{vanA2BGW}
$$\sum_{i=1}^n\text{ord}_{z_i=\infty\ }\omega^{{\text{BGW}},A_2}_{g,n}(z_1,...,z_n)\geq 2g-2
$$
where $\text{ord}_{z=\infty\ }\eta(z)$ is the order of vanishing of the differential at $z=\infty$.
\end{lemma}
\begin{proof}
We can make the rational differential
$$\omega^{A_2}_{g,n}(z_1,...,z_n)=\frac{p_{g,n}(z_1,....,z_n)}{\prod_{i=1}^n(z_i^2-1)^{2g}}dz_1...dz_n
$$
homogeneous by applying topological recursion to $x(z)=z^3-3Q^2z$ and $y=\sqrt{-3}/x'(z)$ which are homogeneous in $z$ and $Q$.  Then $\omega^{A_2}_{g,n}(Q,z_1,...,z_n)$ is homogeneous in $z$ and $Q$ of degree $2-2g-n$:
$$\omega^{A_2}_{g,n}(Q,z_1,...,z_n)=\lambda^{2-2g-n}\omega^{A_2}_{g,n}(\lambda Q,\lambda z_1,...,\lambda z_n).
$$
The degree of homogeneity uses the fact that $(z,Q)\mapsto(\lambda z,\lambda Q)$ $\Rightarrow$ $ydx\mapsto\lambda ydx$ $\Rightarrow$ $\omega_{g,n}\mapsto\lambda^{2-2g-n}\omega_{g,n}$
because $ydx$ appears in the kernel $K(p_1,p)$ with homogeneous degree $-1$ which easily leads to degree $2-2g-n$ for $\omega_{g,n}$.  The degree $2-2g-n$ homogeneity of
$$\omega^{A_2}_{g,n}(Q,z_1,...,z_n)=\frac{p_{g,n}(Q,z_1,....,z_n)}{\prod_{i=1}^n(z_i^2-Q^2)^{2g}}dz_1...dz_n
$$
implies that $\deg p_{g,n}(Q,z_1,....,z_n)=4gn-n+2-2g-n$.  But we also know that $\omega^{A_2}_{g,n}(Q,z_1,...,z_n)$ is well-defined as $Q\to 0$---the limit becomes $\omega_{g,n}$ of the spectral curve $x(z)=z^3$ and $y=\sqrt{-3}/x'(z)$ using the topological recursion defined by Bouchard and Eynard \cite{BEyThi}---so $\deg p_{g,n}(z_1,....,z_n)\leq 4gn-n+2-2g-n$.  Note that $dz_i$ is homogeneous of degree 1 but has a pole of order 2 at $z_i=\infty$, hence
$$\sum_{i=1}^n\text{ord}_{z_i=\infty\ }\omega^{{\text{BGW}},A_2}_{g,n}(z_1,...,z_n)=4gn-\deg p_{g,n}(z_1,....,z_n)-2n\geq 2g-2.
$$
\end{proof}

Primary invariants of a partition function are those coefficients of $\prod_{i=1}^nt^{\alpha_i}_{k_i}$ with all $k_i=0$.  They correspond to intersections in $\overline{\modm}_{g,n}$ with no $\psi$ classes.  The primary invariants of $Z^{\Theta}_{A_2}(\hbar,\{t^{\alpha}_k\})$ vanish for $n<2g-2$.  This uses $\deg\Omega^{A_2}_{g,n}\leq\frac13(g-1+n)$ so $\deg\Omega^{A_2}_{g,n}\cdot\Theta_{g,n}\leq\frac13(g-1+n)+2g-2+n<3g-3+n$ when $n<2g-2$.  These vanishing coefficients correspond to the relations $\Theta_{g,n}\cdot R^{g-1}_{g,A}=0$ which, as discussed above, give relations between coefficients of $Z^\Theta(\hbar,\{t_k\})$.

The primary coefficients of $Z^{\text{BGW}}_{A_2}(\hbar,\{t^{\alpha}_k\})$  correspond to $$ \Res_{z_1=\infty}...\Res_{z_n=\infty}\prod_{i=1}^nz_i^{\epsilon_i}\omega^{A_2}_{g,n}(z_1,...,z_n)$$ for $\epsilon_i=1$ or 2.  Different choices of $\epsilon_i$ give different relations (except half which vanish for parity reasons).  By Lemma~\ref{vanA2BGW} the sums of the orders of vanishing of $\omega^{{\text{BGW}},A_2}_{g,n}(z_1,...,z_n)$ at $z_i=\infty$ is bounded below by $2g-2$.   For $n< 2g-2$, since $\sum_{i=1}^n\text{ord}_{z_i=\infty\ }\omega^{{\text{BGW}},A_2}_{g,n}(z_1,...,z_n)\geq 2g-2$ by Lemma~\ref{vanA2BGW}, there exists an $i$ such that $\text{ord}_{z_i=\infty\ }\omega^{{\text{BGW}},A_2}_{g,n}(z_1,...,z_n)\geq 2$.  Hence $z_i^{\epsilon_i}\omega^{A_2}_{g,n}(z_1,...,z_n)$ is holomorphic at $z_i=\infty$, so 
$$\displaystyle\Res_{z_i=\infty}z_i^{\epsilon_i}\omega^{A_2}_{g,n}(z_1,...,z_n)=0$$ 
and we have
\begin{equation}  \label{ninputs}
n< 2g-2\Rightarrow \Res_{z_1=\infty}...\Res_{z_n=\infty}\prod_{i=1}^nz_i^{\epsilon_i}\omega^{A_2}_{g,n}(z_1,...,z_n)=0.
\end{equation}
Hence the primary coefficients of $Z^{\text{BGW}}_{A_2}(\hbar,\{t^{\alpha}_k\})$ vanish for $n< 2g-2$ yielding a common set of relations satisfied by both the coefficients of $Z^{\Theta}(\hbar,t_0,t_1,...)$ and $Z^{\text{BGW}}(\hbar,t_0,t_1,...)$.  
\end{proof}
An example, of a genus 2 relation produced by Theorem~\ref{infrel}:
$$\quad\omega^{{\text{BGW}},A_2}_{2,1}(z)=\frac{-5z^2-1}{16\sqrt{-3}(z-1)^4(z+1)^4}dz.
$$
It immediately follows that $\displaystyle\Res_{z=\infty}\frac{\sqrt{-3}}{2}z\cdot\omega_{2,1}(z)=0$ which signifies a relation between coefficients of $Z^{\text{BGW}}(\hbar,t_0,t_1,...)$.  We will write the relations using $\Theta_{g,n}$ however the relations are between coefficients of $Z^{\text{BGW}}(\hbar,t_0,t_1,...)$ and what we are showing here is that these coefficients satisfy the same relations as intersection numbers involving $\Theta_{g,n}$, or equivalently coefficients of $Z^{\Theta}(\hbar,t_0,t_1,...)$.  The graphical expansion encoded by both Givental's construction and toplogical recursion is  given by:
\begin{center}
\begin{tikzpicture}
\node(0) at (-5,0)   {}    ;
\node(1) at (-4,0)[shape=circle,draw]        {2};
\path [-] (0)      edge   node   {}     (1);
\node(5) at (-2.5,0)   {}    ;
\node(6) at (-1.5,0)[shape=circle,draw]        {2};
\node(7) at (-.5,.5)   {};
\tikzset{decoration={snake,amplitude=.4mm,segment length=2mm,
                       post length=0mm,pre length=0mm}}
  \draw[decorate] (6) -- (7);
\node(2) at (6,0)(text) {};
\path [-] (5)      edge   node   {}     (6);

\node(3) at (1,0)[shape=circle,draw]        {1};
\node(4) at (3,0)[shape=circle,draw]        {1};
\path [-] (3)      edge   node [above]  {}     (4);
\node(10) at (0,0)        {};
\path [-] (10)      edge   node   {}         (3);
\circledarrow{}{text}{.8cm};
\node(8) at (5.25,0)[shape=circle,draw]        {1};
\node(11) at (4.25,0)        {};
\path [-] (11)      edge   node   {}         (8);
\end{tikzpicture}
\end{center}
(plus graphs containing genus 0 vertices on which $\Theta_{2,1}$ vanishes)	  
which contributes 
\begin{align*}
2^2\cdot\frac{60}{1728}\cdot&\int_{\overline{\modm}_{2,1}}\Theta_{2,1}\cdot\psi_1+2^2\cdot\frac{-60}{1728}\cdot\int_{\overline{\modm}_{2,1}}\Theta_{2,1}\cdot\kappa_1\\
+&2^2\cdot\frac{84}{1728}\cdot\int_{\overline{\modm}_{1,2}}\Theta_{1,2}\cdot\int_{\overline{\modm}_{1,1}}\Theta_{1,1}+\frac{2}{2}\cdot\frac{84-60}{1728}\cdot\int_{\overline{\modm}_{1,3}}\Theta_{1,3}
\end{align*}
which agrees with the expansion in weighted graphs of $\displaystyle\Res_{z=\infty}\frac{\sqrt{-3}}{2}z\cdot\omega_{2,1}(z)=0$ given by
$$\frac{5}{1536}-\frac{15}{1536}+\frac{7}{2304}+\frac{1}{288}=0. 
$$

\appendix

\section{Calculations}  \label{sec:calc}
Here we show explicitly the equality $Z^{\text{BGW}}=Z^{\Theta}$ up to genus 3.  The coefficients of the Br\'ezin-Gross-Witten tau function are calculated recursively since it is a tau function of the KdV hierarchy.  It has low genus  $g$ (= coefficient of $\hbar^{g-1}$) terms given by:
\begin{align*} 
\log Z^{\text{BGW}}=&-\frac{1}{8}\log(1-t_0)+\hbar\frac{3}{128}\frac{t_1}{(1-t_0)^3}+\hbar^2\frac{15}{1024}\frac{t_2}{(1-t_0)^5}\\
&+\hbar^2\frac{63}{1024}\frac{t_1^2}{(1-t_0)^6}+O(\hbar^3)\\
=\frac{1}{8}t_0+&\frac{1}{16}t_0^2+..+\hbar(\frac{3}{128}t_1+\frac{9}{128}t_0t_1+..)+\hbar^2(\frac{15}{1024}t_2+\frac{63}{1024}t_1^2+..)
\end{align*}
The intersection numbers of $\Theta_{g,n}$ stored in
$$\log Z^{\Theta}(\hbar,t_0,t_1,...)=\sum_{g,n,\vec{k}}\frac{\hbar^{g-1}}{n!}\int_{\overline{\modm}_{g,n}}\Theta_{g,n}\cdot\prod_{j=1}^n\psi_j^{k_j}\prod t_{k_j}$$
are calculated recursively via relations among tautological classes in $H^*(\overline{\modm}_{g,n},\bq)$.  The calculation of these intersection numbers up to genus 2 can be found throughout the text.  We assemble them here for convenience, then present the genus 3 calculations.
\begin{itemize}
\item[$\fbox{g\ =\ 0}$]  Theorem~\ref{main} property \eqref{genus0} gives $\Theta_{0,n}=0$ which agrees with the vanishing of all genus 0 terms in $Z^{\text{BGW}}$. 
\item[$\fbox{g\ =\ 1}$]  Proposition~\ref{theta11} gives $\Theta_{1,1}=3\psi_1$ hence $\int_{\overline{\modm}_{1,1}}\Theta_{1,1}=\frac{1}{8}$.  We use this together with the dilaton equation to get  $\int_{\overline{\modm}_{1,n}}\Theta_{1,n}=\frac{(n-1)!}{8}$.  This agrees with $-\frac{1}{8}\log(1-t_0)$ in $\log Z^{\text{BGW}}$.
\item[$\fbox{g\ =\ 2}$]    Using Mumford's relation \cite{MumTow} $\kappa_1=$ sum of boundary terms in $\overline{\modm}_{2}$  which coincides with a genus 2 Pixton relation, Example~\ref{gen2rel} produced the genus 2 intersection numbers from the genus 1 intersection numbers.
\begin{align*}  \int_{\overline{\modm}_{2}}\Theta_{2}\cdot\kappa_1&=\frac{7}{5}\cdot\int_{\overline{\modm}_{1,1}}\Theta_{1,1}\cdot\int_{\overline{\modm}_{1,1}}\Theta_{1,1}\cdot\frac{1}{|\text{Aut}(\Gamma_1)|}+\frac{1}{5}\cdot\int_{\overline{\modm}_{1,2}}\Theta_{1,2}\cdot\frac{1}{|\text{Aut}(\Gamma_2)|}\\&=\frac{7}{5}\cdot\frac{1}{8}\cdot\frac{1}{8}\cdot\frac{1}{2}+\frac{1}{5}\cdot\frac{1}{8}\cdot\frac{1}{2}=\frac{3}{128}.
\end{align*} 
Note that $\int_{\overline{\modm}_{2,1}}\Theta_{2,1}\cdot\psi_1=\int_{\overline{\modm}_{2,1}}\pi^*\Theta_{2}\cdot\psi_1^2=\int_{\overline{\modm}_{2}}\Theta_{2}\cdot\kappa_1$.   Using the dilaton equation we then get  $\int_{\overline{\modm}_{2,n}}\Theta_{2,n}\cdot\psi_1=\frac{3(n+1)!}{256}$ which agrees with the $\hbar\frac{3}{128}\frac{t_1}{(1-t_0)^3}$ term in $\log Z^{\text{BGW}}$.

\item[$\fbox{g\ =\ 3}$]   There are two independent genus 3 Pixton relations expressing $\kappa_2$ and $\kappa_1^2$ as sums of boundary terms in $\overline{\modm}_{3}$.  The relations correspond to sums over stable graphs in $\overline{\modm}_{3}$ hence they contain many terms.  In place of these, we use the equivalent relations discovered earlier in \cite{KLiGen,KLiTop} which push forward to relations in $\overline{\modm}_{3}$.  In $\overline{\modm}_{3,1}$ there is a relation $\psi_1^3=$ sum of boundary terms, which yields 
\begin{align*} \int_{\overline{\modm}_{3,1}}&\Theta_{3,1}\cdot\psi_1^2=\int_{\overline{\modm}_{3,1}}\pi^*\Theta_3\cdot\psi_1^3\\
=&\frac{41}{21}\cdot\int_{\overline{\modm}_{2,1}}\Theta_{2,1}\cdot\psi_1\cdot\int_{\overline{\modm}_{1,1}}\Theta_{1,1}+\frac{5}{42}\cdot\int_{\overline{\modm}_{2,2}}\Theta_{2,2}\cdot\psi_1\\
&-\frac{1}{105}\cdot\int_{\overline{\modm}_{1,1}}\Theta_{1,1}\cdot\int_{\overline{\modm}_{1,3}}\Theta_{1,3}\cdot\frac{1}{|\text{Aut}|}+\frac{11}{70}\cdot\int_{\overline{\modm}_{1,2}}\Theta_{1,2}\cdot\int_{\overline{\modm}_{1,2}}\Theta_{1,2}\cdot\frac{1}{|\text{Aut}|}\\
&-\frac{4}{35}\cdot\int_{\overline{\modm}_{1,1}}\Theta_{1,1}\cdot\int_{\overline{\modm}_{1,2}}\Theta_{1,2}\cdot\int_{\overline{\modm}_{1,1}}\Theta_{1,1}-\frac{1}{105}\cdot\int_{\overline{\modm}_{1,1}}\Theta_{1,1}\cdot\hspace{-1mm}\int_{\overline{\modm}_{1,3}}\Theta_{1,3}\cdot\frac{1}{|\text{Aut}|}\\
&-\frac{1}{1260}\cdot\int_{\overline{\modm}_{1,4}}\Theta_{1,4}\cdot\frac{1}{|\text{Aut}|}
\\
=&\frac{41}{21}\cdot\frac{3}{128}\cdot\frac{1}{8}+\frac{5}{42}\cdot\frac{9}{128}-\frac{1}{105}\cdot\frac{1}{8}\cdot\frac{2}{8}\cdot\frac{1}{2}+\frac{11}{70}\cdot\frac{1}{8}\cdot\frac{1}{8}\cdot\frac{1}{2}
-\frac{4}{35}\cdot\frac{1}{8}\cdot\frac{1}{8}\cdot\frac{1}{8}\\
&-\frac{1}{105}\cdot\frac{1}{8}\cdot\frac{2}{8}\cdot\frac{1}{2}-\frac{1}{1260}\cdot\frac{6}{8}\cdot\frac{1}{4}\\=&\frac{15}{1024}
\end{align*} 
In $\overline{\modm}_{3,2}$, there is a relation $\psi_1^2\psi_2-\psi_1\psi_2^2=$ sum of boundary terms, which yields
\begin{align*} 7\int_{\overline{\modm}_{3,2}}&\Theta_{3,2}\cdot(\psi_1^2-\psi_1\psi_2)=
7\int_{\overline{\modm}_{3,2}}\pi^*\Theta_{3,1}\cdot(\psi_1^2\psi_2-\psi_1\psi_2^2)\\
=&-\frac{16}{3}\cdot\int_{\overline{\modm}_{2,2}}\Theta_{2,2}\cdot\psi_2\cdot\int_{\overline{\modm}_{1,1}}\Theta_{1,1}-5\int_{\overline{\modm}_{2,2}}\Theta_{2,2}\cdot\psi_1\cdot\int_{\overline{\modm}_{1,1}}\Theta_{1,1}\\
-&\frac{40}{3}\cdot\int_{\overline{\modm}_{2,1}}\Theta_{2,1}\cdot\psi_1\cdot\int_{\overline{\modm}_{1,2}}\Theta_{1,2}
-\frac{1}{6}\cdot\int_{\overline{\modm}_{2,3}}\Theta_{2,3}\cdot\psi_1-\int_{\overline{\modm}_{2,3}}\Theta_{2,3}\cdot\psi_1\cdot\frac{1}{|\text{Aut}|}\\
&-\frac{1}{15}\cdot\int_{\overline{\modm}_{1,1}}\Theta_{1,1}\cdot\int_{\overline{\modm}_{1,4}}\Theta_{1,4}\cdot\frac{1}{|\text{Aut}|}-\frac{9}{10}\cdot\int_{\overline{\modm}_{1,3}}\Theta_{1,3}\cdot\int_{\overline{\modm}_{1,2}}\Theta_{1,2}\\
&-\frac{1}{15}\cdot\int_{\overline{\modm}_{1,1}}\Theta_{1,1}\cdot\int_{\overline{\modm}_{1,4}}\Theta_{1,4}\cdot\frac{1}{|\text{Aut}|}+\frac{4}{15}\cdot\int_{\overline{\modm}_{1,2}}\Theta_{1,2}\cdot\int_{\overline{\modm}_{1,3}}\Theta_{1,3}\cdot\frac{1}{|\text{Aut}|}\\
&-\frac{4}{5}\cdot\int_{\overline{\modm}_{1,1}}\Theta_{1,1}\cdot\int_{\overline{\modm}_{1,3}}\Theta_{1,3}\cdot\int_{\overline{\modm}_{1,1}}\Theta_{1,1}\\
&+\frac{16}{5}\cdot\int_{\overline{\modm}_{1,1}}\Theta_{1,1}\cdot\int_{\overline{\modm}_{1,2}}\Theta_{1,2}\cdot\int_{\overline{\modm}_{1,2}}\Theta_{1,2}-\frac{1}{180}\cdot\int_{\overline{\modm}_{1,5}}\Theta_{1,5}\cdot\frac{1}{|\text{Aut}|}\\
=&-\frac{16}{3}\cdot\frac{9}{128}\cdot\frac{1}{8}-5\frac{9}{128}\cdot\frac{1}{8}-\frac{40}{3}\cdot\frac{3}{128}\cdot\frac{1}{8}
-\frac{1}{6}\cdot\frac{36}{128}-\frac{36}{128}\cdot\frac{1}{2}\\
&-\frac{1}{15}\cdot\frac{1}{8}\cdot\frac{6}{8}\cdot\frac{1}{2}-\frac{9}{10}\cdot\frac{2}{8}\cdot\frac{1}{8}-\frac{1}{15}\cdot\frac{1}{8}\cdot\frac{6}{8}\cdot\frac{1}{2}+\frac{4}{15}\cdot\frac{1}{8}\cdot\frac{2}{8}\cdot\frac{1}{2}\\
&-\frac{4}{5}\cdot\frac{1}{8}\cdot\frac{2}{8}\cdot\frac{1}{8}+\frac{16}{5}\cdot\frac{1}{8}\cdot\frac{1}{8}\cdot\frac{1}{8}
-\frac{1}{180}\cdot\frac{24}{8}\cdot\frac{1}{4}\\=&-\frac{357}{1024}
\end{align*} 
Hence
$$ \int_{\overline{\modm}_{3,2}}\Theta_{3,2}\cdot\psi_1\psi_2=\int_{\overline{\modm}_{3,2}}\Theta_{3,2}\cdot\psi_1^2+\frac{1}{7}\frac{357}{1024}=\frac{75}{1024}+\frac{51}{1024}=\frac{63}{512}
$$
where $\int_{\overline{\modm}_{3,2}}\Theta_{3,2}\cdot\psi_1^2=\frac{75}{1024}$ is obtained from $\int_{\overline{\modm}_{3,1}}\Theta_{3,1}\cdot\psi_1^2=\frac{15}{1024}$ via the dilaton equation.  The dilaton equation then yields $\int_{\overline{\modm}_{3,n}}\Theta_{3,n}\cdot\psi_1^2=\frac{75}{1024}\frac{(n+3)!}{5!}$ and  $ \int_{\overline{\modm}_{3,n}}\Theta_{3,n}\cdot\psi_1\psi_2=\frac{63}{512}\frac{(n+3)!}{5!}$ which agree with the terms $\hbar^2\frac{15}{1024}\frac{t_2}{(1-t_0)^5}$+$\hbar^2\frac{63}{1024}\frac{t_1^2}{(1-t_0)^6}$ in $\log Z^{\text{BGW}}$.
\end{itemize}


\begin{thebibliography}{99}

\bibitem{AJaMod} Abramovich, Dan and Jarvis, Tyler.
\emph{Moduli of twisted spin curves.}
Proc. Amer. Math. Soc. {\bf 131} (2003), 685-699. 

\bibitem{AleCut} Alexandrov, A.
Adv. Theor. Math. Phys. {\bf 22} (2018), 1347-1399.


\bibitem{BerCoh} Bergstr{\"o}m, J.
\emph{Cohomology of moduli spaces of curves of genus three via point counts.}
J. f\"ur die Reine und Angew. Math., {\bf 622} (2008), 155-187.
 
\bibitem{BToRat} Bergstr{\"o}m, J. and Tommasi, O.
\emph{The rational cohomology of $\overline{\modm}_4$.}
Math. Ann. {\bf 338} (2007), 207-239.


\bibitem{BEyThi} V. Bouchard and B. Eynard. 
\emph{Think globally, compute locally.}
J. High Energy Phys., {\bf 143}  34 pp, (2013).

\bibitem{BGrExt} Br\'{e}zin, E. and Gross, D.J. 
\emph{The external field problem in the large $N$ limit of QCD.}
 Phys. Lett. B97 (1980) 120.
 
\bibitem{CLLWit} Chang, Huai-Liang; Li, Jun and Li, Wei-Ping.
\emph{Witten's top Chern class via cosection localization.}
Invent. Math. {\bf 200} (2015), 1015-1063.

\bibitem{CEyHer} Chekhov, L. and Eynard, B.
\emph{Hermitian matrix model free energy: Feynman graph technique for all genera.} 
J. High Energy Phys., no. 3, 014, 18 pp., (2006).
 
\bibitem{CNoTop} Chekhov, L. and Norbury, P.
\emph{Topological recursion with hard edges.}
Inter. Journal Math. {\bf 30} (2019) 1950014 (29 pages).

\bibitem{ChiTow} Chiodo, Alessandro
\emph{Towards an enumerative geometry of the moduli space of twisted curves and $r$th roots.}
Compositio Mathematica {\bf 144} (2008), 1461-1496.

\bibitem{CGJZPow}
Clader, E.; Grushevsky, S; Janda, F. and Zakharov, D. 
\emph{Powers of the Theta Divisor and Relations in the Tautological Ring} 
Int. Math. Res. Notices, {\bf 24} (2018), 7725-7754.



\bibitem{DSZadm} Delecroix, V.; Schmitt, J. and  van Zelm, J.
\emph{admcycles -- a Sage package for calculations in the tautological ring of the moduli space of stable curves.}
\href{http://arxiv.org/abs/2002.01709}{arXiv:2002.01709}

\bibitem{DNoTopI}
Do, Norman and Norbury, Paul.
\emph{Topological recursion for irregular spectral curves.}
To appear in {\em Journal of the London Mathematical Society.}

\bibitem{DNoTop}
Do, Norman and Norbury, Paul.
\emph{Topological recursion on the Bessel curve.}
Comm. Number Theory Phys. {\bf 12} (2018), 53-73.

\bibitem{DubGeo}
Dubrovin, B.
\emph{Geometry of 2D topological field theories },
Integrable Systems and Quantum Groups (Authors: R. Donagi, B. Dubrovin, E. Frenkel, E. Previato), Eds. M. Francaviglia, S. Greco, Springer Lecture Notes in Math. {\bf 1620} (1996), 120 -348.

\bibitem{DNOPSDub}
Dunin-Barkowski, P.; Norbury, P.; Orantin, N.; Popolitov, A. and Shadrin, S.
\emph{Dubrovin's superpotential as a global spectral curve.}
J. Inst. Math. Jussieu. {\bf 18} (2019), 449-497.

\bibitem{DNOPSPri} Dunin-Barkowski, P.;  Norbury, P.; Orantin, N.; Popolitov, A. and Shadrin, S. \emph{Primary invariants of Hurwitz Frobenius manifolds.} 
Topological Recursion and its influence in Analysis, Geometry and Physics, Proceedings of Symposia in Pure Mathematics {\bf 100} (2018), 297-332.

\bibitem{DOSSIde} Dunin-Barkowski, P.; Orantin, N.; Shadrin, S. and Spitz, L.
\emph{Identification of the Givental formula with the spectral curve topological recursion procedure.}
Comm. Math. Phys. {\bf 328}, (2014), 669-700.

\bibitem{DSSGiv}
Dunin-Barkowski, P.; Shadrin, S. and Spitz, L. 
\emph{Givental graphs and inversion symmetry},
Lett. Math. Phys. {\bf 103} (2013), 533-557.

\bibitem{ELSVHur}  Ekedahl, Torsten; Lando, Sergei; Shapiro, Michael and Vainshtein, Alek.
\emph{Hurwitz numbers and intersections on moduli spaces of curves.}
Invent. Math. {\bf 146} (2001), 297-327. 

\bibitem{EynInv} Eynard, Bertrand
\emph{Invariants of spectral curves and intersection theory of moduli spaces of complex curves.}
Comm. Number Theory Phys. {\bf 8} (2014), 541-588.

\bibitem{EOrInv}
Eynard, B. and Orantin, N. 
\emph{Invariants of algebraic curves and topological expansion.}
Commun. Number Theory Phys. {\bf 1} (2007), no.~2, 347-452.

\bibitem{EOrTop} Eynard, B. and Orantin, N.
\emph{Topological recursion in enumerative geometry and random matrices.}
J. Phys. A {\bf 42}, no. 29, 293001, 117 pp., (2009).

\bibitem{FPaRel} Faber, C. and Pandharipande, R.
\emph{Relative maps and tautological classes.}
J. Eur. Math. Soc. {\bf 7} (2005), 13-49.

\bibitem{FJRQua} Fan, Huijun; Jarvis, Tyler; Ruan, Yongbin. 
\emph{Quantum singularity theory for $A_{r-1}$ and $r$-spin theory.} 
Ann. Inst. Fourier (Grenoble) {\bf 61} (2011), 2781-2802.

\bibitem{FJRWit} Fan, Huijun; Jarvis, Tyler; Ruan, Yongbin. 
\emph{The Witten equation, mirror symmetry, and quantum singularity theory.}
Annals of Math. {\bf 178} (2013), 1-106.

\bibitem{FayThe}  Fay, J.D.
\emph{Theta functions on Riemann surfaces.} 
Lecture Notes in Mathematics, {\bf 352}. Springer-Verlag, Berlin-New York, 1973.

\bibitem{GloHod} Getzler, E. and Looijenga, E. 
\emph{The Hodge polynomial of $\overline{\modm}_{3,1}$.}
\href{http://arxiv.org/abs/9910174}{arXiv:9910174}

\bibitem{GivGro} Givental, A.
\emph{Gromov-Witten invariants and quantization of quadratic hamiltonians.}
Moscow Math. Journal {\bf 4}, (2001), 551-568.

\bibitem{GWiPos} Gross, D. and Witten, E.
\emph{Possible third-order phase transition in the large-$N$ lattice gauge theory.}
Phys. Rev. D21 (1980) 446. 



\bibitem{JKVMod} Jarvis, T.; Kimura, T. and Vaintrob, A.
\emph{Moduli Spaces of Higher Spin Curves and Integrable Hierarchies.}
Compositio Mathematica {\bf 126}, (2001), 157-212.


\bibitem{HurRie} Hurwitz, A.
\emph{\"{U}ber Riemannische Fl\"{a}chen mit gegebenen Verzweigungspunkten.} 
Math. Ann. {\bf 39} (1897), 1-60.

\bibitem{JPPZDou} Janda, F., Pandharipande, R., Pixton, A. and Zvonkine, D.
\emph{Double ramification cycles on the moduli spaces of curves.}
Publ. math. IHES {\bf 125} (2017), 221-266.

\bibitem{KLiGen} Kimura, Takashi and Liu, Xiaobo.
\emph{A Genus-3 Topological Recursion Relation.}
Comm. Math. Phys. {\bf 262}, (2006), 645-661.

\bibitem{KLiTop} Kimura, Takashi and Liu, Xiaobo.
\emph{Topological recursion relations on $\overline{\modm}_{3,2}$.}
Sci. China Math. {\bf 58} (2015),1909-1922.

\bibitem{KonInt} Kontsevich, Maxim.
\emph{Intersection theory on the moduli space of curves and the matrix {A}iry function.}
Comm. Math. Phys. {\bf 147}, (1992), 1-23.

\bibitem{LeeInv} Lee, Y.-P.
\emph{Invariance of tautological equations II: Gromov?Witten theory (with Appendix A by Y. Iwao and Y.-P. Lee)}. 
J. Am. Math. Soc. {\bf 22} (2009), 331-352.


\bibitem{LPSZChi}  Lewanski, D.; Popolitov, A.; Shadrin, S. and Zvonkine, D. 
\emph{Chiodo formulas for the $r$-th roots and topological recursion.}
Lett. Math. Phys. {\bf 107} (2017), 901-919.

\bibitem{LooTau} Looijenga, E.
\emph{On the tautological ring of $\modm_g$.} 
Invent. Math. {\bf 121} (1995), 411-419.

\bibitem{MirSim}
Mirzakhani, Maryam 
\emph{Simple geodesics and Weil-Petersson volumes of moduli spaces of bordered Riemann surfaces.}
Invent. Math. {\bf 167} (2007), 179-222.

\bibitem{MJDSol} Miwa, T.; Jimbo, M.; Date, E.
\emph{Solitons. Differential equations, symmetries and infinite-dimensional algebras.} 
Cambridge Tracts in Mathematics, {\bf 135} Cambridge University Press, Cambridge, 2000.

\bibitem{MumTow} Mumford, D.
\emph{Towards an enumerative geometry of the moduli space of curves.}
Arithmetic and geometry, Vol. II, Progr. Math. {\bf 36}, Birkh\"auser Boston, (1983), 271-328.

\bibitem{NatMod} 
Natanzon, S.
\emph{Moduli of Riemann Surfaces, Real Algebraic Curves, and Their Superanalogs.} 
Transl. Math. Monographs {\bf 225}, AMS, 2004.

\bibitem{NorEnu}
Norbury, P.
\emph{Enumerative geometry via the moduli space of super Riemann surfaces.}
\href{http://arxiv.org/abs/2005.04378}{arXiv:2005.04378}

\bibitem{NorGro}
Norbury, P. 
\emph{Gromov-Witten invariants of $\bp^1$ coupled to a KdV tau function.}
\href{http://arxiv.org/abs/1812.04221}{arXiv:1812.04221}

\bibitem{NScGro}
Norbury, P. and Scott, N.
\emph{Gromov-Witten invariants of $\mathbf{P}^1$ and Eynard-Orantin invariants.}
Geom. Topol. {\bf 18}, no. 4, 1865-1910, (2014).

\bibitem{NScPol} Norbury, P. and Scott, N.
\newblock Polynomials representing Eynard-Orantin invariants.
Quart. J. Math. {\bf 64} (2013), 515-546.

\bibitem{PPZRel} Pandharipande, R.; Pixton, A. and Zvonkine, D.
\emph{Relations on $\overline{\modm}_{g,n}$ via 3-spin structures.}
J. Amer. Math. Soc. {\bf 28} (2015), 279-309.


\bibitem{SaiPer} Saito, K.
\emph{Period Mapping Associated to a Primitive Form.}
Publ. RINS. Kyoto Univ. {\bf 19} (1983), 1231-1264.

\bibitem{ShaBCOV} Shadrin, S.
\emph{BCOV theory via Givental group action on cohomological fields theories.} 
Mosc. Math. J. {\bf 9} (2009), 411-429. 

\bibitem{ShrRie} Shramchenko, V.
\emph{Riemann-Hilbert problem associated to Frobenius manifold structures on Hurwitz spaces: irregular singularity.} 
Duke Math. J. {\bf 144} (2008), 1-52.

\bibitem{SWiJTG}
Stanford, D. and Witten, E.
\emph{JT Gravity and the Ensembles of Random Matrix Theory}
\href{http://arxiv.org/abs/1907.03363}{arXiv:1907.03363}

\bibitem{TelStr}
Teleman, C.
\emph{The structure of 2D semi-simple field theories},
Invent. Math. {\bf 188} (2012), 525-588.

\bibitem{ToeThe} Toen, B.
\emph{Th\'eor\`emes de Riemann-Roch pour les champs de Deligne-Mumford.} 
K-Theory {\bf 18} (1999), 33-76.

\bibitem{WitTwo} Witten, Edward
\emph{Two-dimensional gravity and intersection theory on moduli space.} Surveys in differential geometry (Cambridge, MA, 1990), 243-310, Lehigh Univ., Bethlehem, PA, 1991.

\end{thebibliography}
\end{document}